\documentclass[a4paper,11pt]{amsart}

\usepackage{amssymb,amsfonts,amsthm,amscd,stmaryrd}

\usepackage{ae}
\usepackage[utf8]{inputenc}
\numberwithin{figure}{section}

\usepackage{mathrsfs,a4wide}
\usepackage{esint}

\usepackage{color}

\usepackage{import}

\usepackage[leqno]{amsmath}

\numberwithin{figure}{section}

\newtheorem{theorem}{Theorem}[section]
\newtheorem{lemma}[theorem]{Lemma}
\newtheorem{proposition}[theorem]{Proposition}

\theoremstyle{definition}
\newtheorem{definition}[theorem]{Definition}

\newtheorem{remark}[theorem]{Remark}
\numberwithin{equation}{section}

\newcommand{\wto}{\rightharpoonup}
\newcommand{\de}{\delta}
\newcommand{\R}{\mathbb{R}}
\newcommand{\N}{\mathbb{N}}

\newcommand{\Ha}{\mathcal{H}}

\newcommand{\beq}{\begin{equation}}
\newcommand{\eeq}{\end{equation}}

\newcommand{\eps}{\varepsilon}
\newcommand{\e}{\varepsilon}
\newcommand{\vphi}{\varphi} 
\newcommand{\la}{\langle}
\newcommand{\ra}{\rangle}
\newcommand{\diver}{\operatorname{div}}
\newcommand{\Div}{\operatorname{div}}
\newcommand{\pa}{\partial}

\newcommand{\Om}{\Omega}
\newcommand{\medint}{-\kern -,375cm\int}
\newcommand{\medintinrigo}{-\kern -,315cm\int}
\newcommand{\C}{\mathbb{C}}

\begin{document}

\title[Surface elastic flow in 3D]{The surface diffusion flow with elasticity in three dimensions}

\author{Nicola Fusco}

\author{Vesa Julin}

\author{Massimiliano Morini}

\keywords{}

\begin{abstract} 
We establish  short-time existence of a smooth solution to the surface diffusion equation with an elastic term and without an additional curvature regularization  in three space dimensions. We also prove the asymptotic stability of strictly stable stationary sets. 
\end{abstract}

\maketitle

\tableofcontents

\section{Introduction}
Morphological evolution of  strained elastic solids, driven by stress and surface mass transport
occurs in many physical systems. One instance is the hetero-epitaxial growth of elastic films when a lattice mismatch between film and substrate is present. Another example is given by the phase separation in several small connected phases within a common elastic body, which takes place in certain alloys under specific
 temperature conditions.
 A third  situation is  represented by the nucleation and evolution of material voids inside a stressed elastic
   solid. From the mathematical point of view, such phenomena are related to a free energy functional, which is typically given
 by the sum of the stored elastic energy and the surface energy accounting for the surface tension along the  interface between the phases. In this context the 
 equilibria are identified with the local or global minimizers under a volume constraint of the aforementioned energy.
 
 All these variational problems
 can be regarded as non-local   {\em isoperimetric problems}, 
 where the non-locality is given by the elastic term. 
They are very well studied in the
 physical and numerical literature, see for instance \cite{GN, GurJab, RRV, SMV, SpMe}.
  Concerning rigorous mathematical analysis, we refer to \cite{BGZ15, Bo0, BC, DP, FFLM, FM09, GZ14} for some existence, regularity and stability results related to a  variational model  describing the equilibrium configurations of two-dimensional epitaxially strained elastic films, and to 
\cite{Bo,  CS07} for results in three-dimensions.
A hierarchy of variational principles  to describe  equilibrium shapes in the aforementioned contexts has been introduced in \cite{GuVo}. 
   
  In what follows we consider the following prototypical  energy
  \beq\label{prot1}
 \mathcal{J}(F):=\frac12\int_{\Om\setminus F}\C E(u_F):E(u_F)\, dx+\Ha^2(\pa F)\,. 
\eeq
The associated minimum problem under a volume constraint can be used to describe the equilibrium shapes of voids in elastically stressed solids (see for instance \cite{SMV}). 
Here, the set $F\subset\!\subset\Om$ represents the shape of the void that has formed within
the  elastic body $\Om$ (an open subset of  $\R^3$), $u_F$ stands for the equilibrium elastic   displacement in $\Om\setminus F$ subject to a prescribed boundary conditions $u_F=w_0$ on $\pa \Om$ (see \eqref{uf} below), $\C$ is the elasticity tensor 
of the (linearly) elastic material, $E(u_F):=(D u_F+D^T u_F)/2$ denotes the elastic strain of $u_F$, and $\Ha^2$ stands for the surface measure.
 The presence of a nontrivial Dirichlet boundary condition $u_F=w_0$ on $\pa \Om$ is what causes the solid $\Om\setminus F$ to be elastically stressed.  
We refer to \cite{ CJP, FFLMi} for related existence, regularity and stability results in two dimensions.  See also \cite{BCS} for a relaxation result  valid in all dimensions for a variant of \eqref{prot1}.

In this paper we study    the  morphological evolution of shapes towards equilibria of the functional   \eqref{prot1}, driven by stress and surface diffusion.  
 Assuming that relaxation to equilibrium  in the bulk occurs at a much faster time scale,  see \cite{Mu63},  we have, according to the Einstein-Nernst equation, that the evolution is governed by the following {\em volume preserving}  law
 \beq\label{i1}
V_t=\Delta_{\pa F_t}\mu_t \qquad\text{on $\pa F_t$} 
\eeq
where $V_t$ denotes the outer normal velocity of the evolving surface $\pa F_t$ at time $t$ and $\Delta_{\pa F_t} \mu_t$  stands for the Laplace-Beltrami operator acting  on the chemical potential $\mu_t$ along $\pa F_t$. In turn, since $\mu_t$ is given by the {\em first variation} of the free-energy functional $\mathcal{J}$ evaluated at $F_t$ and  taking into account \eqref{eq:J'} below,  \eqref{i1} reads as  
  \beq\label{i2}
 V_t = \Delta_{\pa F_t} \big(H_{F_t} -Q(E(u_{F_t})) \big)\,, 
  \eeq
where $H_{F_t}$ is the sum of the principal  curvatures of $\pa F_t$, with the orientation given by the outer normal,      $u_{F_t}$ is the elastic equilibrium in $\Om\setminus F_t$ subject to
 $u_{F_t}=w_0$  on $\pa \Om$ and  $Q(E(u_{F_t})):=\frac12 \C E(u_{F_t}):E(u_{F_t})$.  Note that the last quantitity involves the traces of the gradient of the elastic equilibrium on the evolving boundary. 

 From the mathematical point of view, \eqref{i2} is a fourth order geometric parabolic equation coupled with the elliptic Lam\'e system, which is solved time by time in the (evolving) bulk. 
 Note also that when $w_0=0$ the elastic term vanishes and thus \eqref{i2} reduces to the pure {\em surface diffusion  flow }  
 \beq\label{sdintro}
 V_t=\Delta_{\pa F_t} H_{F_t}
 \eeq
  for evolving surfaces, studied in \cite{EMS} (in the general   $n$-dimensional case).  Thus, we may also regard \eqref{i2} as a sort of  canonical  nonlocal perturbation   of \eqref{sdintro} by an additive elastic contribution.

 As observed already by Cahn and Taylor \cite{cahn-taylor94} for \eqref{sdintro}, the  equation \eqref{i2} can be seen formally as  the gradient flow of the energy functional $\mathcal{J}$ with respect to a suitable Riemannian metric of $H^{-1}$-type, see for instance \cite[Remark~3.1]{surf2D}.

 Let us mention that in the physical literature a variant of the energy \eqref{prot1} with a \emph{curvature regularization} term has also been considered, see \cite{angenent-gurtin89, BHSV,  dicarlo-gurtin-guidugli92, herring51, RRV, SMV}. This in turn leads to a variant of \eqref{i2} with a sixth order regularization term.  In particular, in \cite{FFLM3} the following regularized energy 
 $$
 \mathcal{J}_\e(F):=\frac12\int_{\Om\setminus F}\C E(u_F):E(u_F)\, dx+\int_{\pa F}\Big(1+\frac\e p|H_F|^p\Big)\, d\Ha^2
 $$
 and the associated evolution equation 
 \beq\label{6th}
V_t=\Delta_{\pa F_t}\Big[H_{F_t}-Q(E(u_{F_t}))-\e\Bigl(\Delta_{\pa F_t}(|H_{F_t}|^{p-2}H_{ F_t})
- |H_{F_t}|^{p-2}H_{F_t}\Bigl(\tfrac{p-1}{p}H_{F_t}^2-2K_{\pa F_t}\Bigr)\Bigr)\Big]
 \eeq
 are considered in the context of periodic graphs modeling the evolutions of epitaxially strained elastic films (see also \cite{FFLM2} for the two-dimensional version of the same equation). Here $K_{\pa F_t}$ stands for the Gaussian curvature of $\pa F_t$, $\e>0$ is a small parameter,  and $p>2$. 
 The local-in-time existence and the asymptotic stability results proven in \cite{FFLM3} (see also \cite{FFLM2, piove}) rely heavily on the presence of the curvature regularization, which makes the elastic contribution a lower order term  easily controlled by the sixth order leading terms of the equation.
 In fact, all the estimates provided there are $\e$-dependent and degenerate as $\e\to 0^+$. {This is  not surprising as the nonlocal elastic term  in \eqref{prot1} cannot  be  treated simply as a lower order perturbation of the perimeter, as shown by the fact that its presence may lead to formation of singularities in the static case (see \cite{FM09} and references therein). }Thus the case $\e=0$ requires completely different  methods.
 
{  A first breakthrough in this direction has been obtained in \cite{surf2D}, where short time existence result for \eqref{i2} was proved in  the two-dimensional case. In  \cite{surf2D} we also proved the asymptotic stability of strickly stable stationary sets. However,  the techniques developed there cannot be applied to  higher dimensions, as some of the crucial estimates rely on the fact that  an $L^2$-bound  of the curvature of the evolving curves provides uniform $C^{1,\alpha}$-bounds. This is of course no longer true in higher dimensions. Moreover, the higher dimensional case is of course much more involved from the geometric point of view. }

In this paper we are able to address equation \eqref{i2} in the physical three-dimensional case and we establish  short time existence and uniqueness of a  solution starting from sufficiently regular initial sets, see Theorem~\ref{thm surf}. We highlight that Theorem~\ref{thm surf} provides also quantitative estimates of the $k$-th order derivatives of the solution depending only on the $H^3$-norm of the initial datum,  somewhat in the spirit of those proved in \cite{KL}.  
We also remark that in general one cannot expect global-in-time existence. Indeed, even when no elasticity is present, singularities such as pinching may develop in finite time, see for instance \cite{GigaIto}.

In the second main result of the paper we establish  global-in-time existence and study the long-time behavior  for a  class of initial data:   we show that {\em strictly stable stationary sets}, that is, sets  $G$ that are stationary  for the energy functional $\mathcal{J}$ and  with positive second variation $\pa^2\mathcal{J}(G)$ are  {\em exponentially stable}  for the flow \eqref{i2}. More precisely, if the initial set $F_0$ is sufficiently close in $H^3$  to the strictly stable set $G$ and has the same volume, then the flow \eqref{i2} starting from $F_0$ exists for all times and converges  to $G$ exponentially fast in $C^k$ for every $k$ as $t\to+\infty$, see Theorem~\ref{thmstability} for the precise statement.

A few comments on the proofs are in order. Concerning short-time existence, as in \cite{surf2D} our strategy is based on the natural idea of thinking of the elastic contribution $Q$ as a forcing term. More precisely,  we set up a fixed point argument on the map $f\mapsto Q(E(u_{F^f_t}))$, where $F^f_t$ is the solution to the forced flow
\beq\label{forcedintro}
 V_t = \Delta_{\pa F_t} \big(H_{F_t} -f \big)\,.
\eeq
Major technical difficulties  originate  from the already mentioned fact that the nonlocal elastic term   is not in general  lower order with respect to  the perimeter.  One of the  main  technical breakthroughs obtained in the present paper is a new delicate elliptic estimate on the higher order derivatives of $Q(E(u_{F_t}))$ in terms of the higher order norms of the evolving boundaries $\pa F_t$, see Theorem~\ref{linearestimate}. The crucial and somewhat surprising point of this result is the linear structure of the estimate,  which allows us to show that the map $f\mapsto Q(E(u_{F^f_t}))$ is a contraction.

Concerning the asymptotic stability analysis, we adapt to the present situation  the methods developed in \cite{AFJM} for the  surface diffusion flow without elasticity (see also \cite{surf2D}).
The rough idea is to look at the asymptotic behavior of the map
$$
t\mapsto\int_{\pa F_t}\big|\nabla_{\pa F_t}\big(H_{F_t} -Q(E(u_{F_t})\big)\big|^2\, d\Ha^2\,,
$$
where $\nabla_{\pa F_t}$ stands for the tangential gradient on $\pa F_t$, and to show that it is decreasing and that in fact it vanishes with exponential rate as $t\to +\infty$. A crucial role in this analysis is played by the energy identity proven in Proposition~\ref{magic formula}  and by the estimates on the flow provided by Theorem~\ref{thm surf}. Let us remark  that such estimates allow us also to considerably simplify the arguments of \cite{AFJM} and to obtain stronger asymptotic convergence results.
 
This paper is organized as follows.  
 In Section~\ref{sec:preliminaries}  we set up  the problem, introduce the main notation and present some  differential geometry preliminaries that will be useful in the subsequent analysis. We also collect several auxiliary results concerning the energy functional $\mathcal{J}$ in \eqref{prot1}. In particular, we describe some properties of strictly stable stationary sets that are  crucial for the asymptotic stability analysis carried out in Section~\ref{sec:stability}. Section~\ref{sec:forced} is devoted to the study of \eqref{forcedintro}, while the  short-time existence theory for  the flow \eqref{i2} is addressed in Section~\ref{sec:existence}.
 In Section~\ref{sec:graphs} we briefly illustrate how to apply our main existence and asymptotic stability results in the case of evolving periodic graphs, that is in the geometric setting considered in \cite{FFLM3}. In particular, in Theorem~\ref{th:2dliapunov} we address  the exponential asymptotic stability of  {\em flat configurations}, thus extending to the evolutionary setting the results of \cite{Bo}.
 In the final Appendix we collect the proofs of two technical lemmas and provide the derivation of the energy identity stated in Proposition~\ref{magic formula}.

 We conclude this introduction by mentioning that it would be interesting to investigate whether  the flow \eqref{6th} studied in \cite{FFLM3} converge to   \eqref{i2} as $\e\to 0^+$. This issue could be probably addressed by adapting the methods developed in \cite{BMN}.

\section{Preliminaries}\label{sec:preliminaries}

\subsection{Geometric preliminaries}

In this section we introduce notation related to Riemannian geometry.  As an introduction to the topic we refer to \cite{AubinBook1, Lee}.  Let $\Sigma \subset \R^n$ be a smooth $(n-1)$-dimensional compact hypersurface without boundary. 
Since $\Sigma$ is embedded in $\R^n$ it has a natural metric, denoted by  $g$,  induced by the Euclidean metric. We thus have a Riemannian manifold $(\Sigma, g)$ and we denote the inner product for vector fields  $X, Y$ as $\la X,Y \ra $,
\[
\la X,Y \ra = g(X,Y) = g_{ij} X^iY^j,
\]
where the last expression is in local coordinates. Throughout the paper we adopt the Einstein summation convention.  Similarly we define the inner product of covector fields $\omega, \eta$,  which in  local coordinates can be written as
\[
\la \omega,\eta  \ra =  g^{ij} \omega_i \eta_j,
\]
where $g^{ij}$ is the inverse matrix of $g_{ij}$. The inner product extends to $\binom{k}{0}$-tensor fields $T = T_{i_1\cdots i_k}	$ and $S= S_{j_1\cdots j_k}$ as
\[
\la T,S  \ra =  g^{i_1j_1}\cdots g^{i_kj_k} T_{i_1\cdots i_k}S_{j_1\cdots j_k}.
\]
The norm of a tensor $T$ is then $|T| =\sqrt{\la T, T \ra }$ and we have the inequality $\la T,S  \ra \leq |T| |S|$. Given a $\binom{k}{0}$-tensor field $T$ we raise the first index by $T_{i_2\cdots i_{k}}^{i_1} = g^{i_1 l}T_{l \,  i_2\cdots i_{k}}$  and thus we obtain a $\binom{k-1}{1}$-tensor field. We may thus write the above inner product as
\[
\la T,S  \ra =   T^{j_1\cdots j_k} S_{j_1\cdots j_k}.
\]
The trace of a $\binom{k}{0}$-tensor field $T$, with $k\geq 2$, on the first two indeces is $\text{tr}\, T = g^{jl} T_{jl \,  i_3 \cdots  i_{k}}$. 

We denote the Riemannian connection on $(\Sigma, g)$ by $\nabla$ and $\nabla^k T = \nabla_{i_1}\cdots \nabla_{i_k} T$ means the $k$-th covariant derivative of a  tensor field $T$.   There is  a slight danger of confusion, since $\nabla^k f$ also denotes  the $k$-th component of the  gradient of a function $f$ defined by raising the index of $\nabla f$ as $\nabla^k f = g^{ki} \nabla_i f$. However, the meaning of $\nabla^k f$  will be clear from the context. We also recall that $\nabla$ is compatible with the metric $g$ which means that $\nabla g = 0$.

 In local coordinates the components of the covariant derivative of a vector field $X = X^i$ and of a covector field $\omega = \omega_k$ are 
\[
\nabla_j X^i = \frac{\pa X^i}{\pa x^j} + \Gamma_{jk}^i X^k \qquad  \text{and}\qquad  \nabla_j \omega_k=  \frac{\pa \omega_k}{\pa x^j} - \Gamma_{jk}^l \omega_l,
\] 
where $\Gamma_{ij}^k$ are the Christoffel symbols given in local coordinates by
\[
\Gamma_{ij}^k = \frac12 g^{kl}\left( \frac{\pa g_{jl}}{\pa x^i} + \frac{\pa g_{il}}{\pa x^j} - \frac{\pa g_{ij}}{\pa x^l}    \right).
\]
The covariant derivative of a $\binom{k}{l}$-tensor field $T = T_{i_1\cdots i_k}^{j_1\cdots j_l}$  is thus a $\binom{k+1}{l}$-tensor field  which in local coordinates can be written as
\[
\nabla_m  T_{i_1\cdots i_k}^{j_1\cdots j_l} = \frac{\pa T_{i_1\cdots i_k}^{j_1\cdots j_l} }{\pa x^m}  + \sum_{s=1}^l T_{i_1\cdots  i_k}^{j_1\cdots p \cdots j_l} \Gamma_{m p}^{j_s}  - \sum_{s=1}^k T_{i_1\cdots p \cdots  i_k}^{j_1\cdots j_l}  \Gamma_{m i_s}^p.
\]

The divergence of a vector field $X^i$ is $\diver X = \nabla_i X^i =  \frac{\pa X^i}{\pa x^i} + \Gamma_{ik}^i X^k $ and the Laplace-Beltrami  of a function $f$ is 
\[
\Delta f = \diver \nabla f = \nabla_i \nabla^i f .
\]
 This can be written as the trace of the covariant Hessian $\nabla^2 f$ as
\[
\Delta f = \text{tr} \,  \nabla^2 f = g^{ij} \nabla_i \nabla_j f. 
\]
We recall the divergence theorem for compact manifolds (without boundary), which states that for a vector field $X$ on $\Sigma$ it holds
\[
\int_{\Sigma} \diver X \, d\Ha^{n-1} = 0.
\]
This yields the integration by parts formula for a function $f$ and a vector field $X$ 
\[
\int_{\Sigma} X^i \nabla_i f \, d\Ha^{n-1} =  - \int_{\Sigma} f \diver X   \, d\Ha^{n-1}.
\]
The integration by parts formula  generalizes to any $\binom{k}{0}$-tensor field $T$ and $\binom{k+1}{0}$-tensor field $S$ as
\beq \label{int by parts}
\int_{\Sigma} \la  \nabla T, S \ra \, d\Ha^{n-1} =  - \int_{\Sigma} \la T , \text{tr} \, \nabla S\ra    \, d\Ha^{n-1},
\eeq
where the trace is on the first two indeces of $\nabla S$.

The Riemann curvature endomorhpism is a $\binom{3}{1}$-tensor field $R_{ijk}^l$ defined such that for vector fields $X,Y,Z$ we have
\[
R(X,Y)Z = \nabla_X \nabla_Y Z -  \nabla_Y \nabla_X Z - \nabla_{[X,Y]}Z,
\]
where $\nabla_X$ is the covariant derivative in direction of  $X$. We adopt  the convention to define the Riemann curvature tensor by lowering the index to the end, i.e., 
$R_{ijkl} = g_{lm} R_{ijk}^m$. The commutation formula of the covariant derivatives for a vector field $X^k$ thus becomes
\beq \label{comm 1}
\nabla_i \nabla_j X^k - \nabla_j \nabla_i X^k = g^{km} R_{ijlm} X^l
\eeq
and for a covector field $\omega_k$ 
$$
\nabla_i \nabla_j \omega_k - \nabla_j \nabla_i \omega_k =- g^{ml} R_{ijkm} \omega_l.
$$

Similar formulas hold for the commutation of higher order covariant derivatives. In particular, throughout the paper we will make repeated use of the fact that for any integer $k\geq 3$ there exists a constant $C>0$ such that
\beq\label{comm 5}
|\nabla_{i_1}\dots\nabla_{i_k}f-\nabla_{i_{\sigma(1)}}\dots\nabla_{i_{\sigma(k)}}f|\leq C\sum_{l=1}^{k-2}|\nabla^l f|
\eeq
for any choice of the indices $i_1, \dots, i_k$ and for any permutation $\sigma$ of $\{1,\dots, k\}$. We recall also that 
$\nabla_i\nabla_j f=\nabla_j\nabla_i f$ for any $i, j$.

Given a positive integer $k$ and $p\in[1,\infty]$ we denote by  $W^{k,p}(\Sigma)$ the Sobolev space endowed with the norm
\[
\|f\|_{W^{k,p}(\Sigma)}:=\sum_{m=0}^k\bigg(\int_\Sigma|\nabla^mf|^p\,d\mathcal H^{n-1}\bigg)^{\frac{1}{p}},
\]
when $p\in [1,\infty)$ and the obvious one when $p=\infty$. Here $\nabla^mf$ stands for the $m$-th covariant derivative of $f$. As customary, when $p=2$ we shall always write $H^k$ instead of $W^{k,2}$. We further define the  norms $\|f\|_{C^{k,\alpha}(\Sigma)}$, $\|f\|_{H^{k+1/2}(\Sigma)}$ and $\|f\|_{H^{-1/2}(\Sigma)}$ with $k \in \mathbb{N}$ and $\alpha \in (0,1)$, in a standard way using the partition of unity. Then the standard embedding theorems for  smooth domains hold also in these spaces. Moreover, we recall the following well known interpolation inequalities, see \cite[Proposition~6.5]{MantegazzaGAFA} and \cite[Theorem~3.70]{AubinBook2}. 
\begin{lemma}
\label{aubinlemma} 
Let $\Sigma\subset\R^n$ be a smooth $(n-1)$-dimensional compact manifold without boundary. Let $l,m,k$ be  integers such that $0\leq l< m$, $k\geq0$, $1\leq q,r\leq\infty$. There exists a constant $C$ with the following property: for every smooth covariant tensor $T$ of order $k$,   one has
\beq\label{aubinlemma1}
\|\nabla^lT\|_{L^p(\Sigma)}\leq C\|T\|_{W^{m,r}(\Sigma)}^\vartheta\|T\|_{L^q(\Sigma)}^{1-\vartheta},
\eeq
where 
\[
\frac{1}{p}=\frac{l}{n-1}+\vartheta\Big(\frac{1}{r}-\frac{m}{n-1}\Big)+(1-\vartheta)\frac{1}{q}
\]
for all $\vartheta\in[l/m,1)$ for which $p$ is nonnegative. Moreover, if $f$ is a smooth function then 
\[
\|\nabla^lf\|_{L^p(\Sigma)}\leq C\|\nabla^mf\|_{L^r(\Sigma)}^\vartheta\|f\|_{L^q(\Sigma)}^{1-\vartheta},
\] 
for all $\vartheta\in[l/m,1)$ for which $p$ is nonnegative, provided $l\geq 1$.
\end{lemma}
\begin{remark}\label{rm:aubinlemma}
Note that \eqref{aubinlemma1} implies also that 
$$
\|\nabla^lT\|_{L^p(\Sigma)}\leq C\|\nabla^mT\|_{L^r(\Sigma)}^\vartheta\|T\|_{L^q(\Sigma)}^{1-\vartheta}+C\|T\|_{L^{\max\{q,r\}}(\Sigma)}\,.
$$
To see this it is enough to observe that $\|T\|_{W^{m,r}(\Sigma)}=\|T\|_{W^{m-1,r}(\Sigma)}+\|\nabla^mT\|_{L^r(\Sigma)}$ and that, in turn, 
for every $l=1, \dots, m-1$ using \eqref{aubinlemma1} and Young's Inequality one gets 
$$
\|\nabla^l T\|_{L^r(\Sigma)}\leq \e\|T\|_{W^{m,r}(\Sigma)}+ C_\e\|T\|_{L^r(\Sigma)}.
$$
\end{remark}
We also recall that the Morrey's inequality implies
$$
\|f\|_{C^{1,\alpha}(\Sigma)} \leq C\|f\|_{W^{2,p}(\Sigma)} 
$$
for $p > n-1$ and $\alpha=1-(n-1)/p$.

We will also need the following result, (see the proof of  \cite[Theorem 4.19]{AubinBook2}).
\begin{lemma}
\label{interchange 1}
Let $f$ be a smooth function on $\Sigma$ and let $k$ be a positive integer. There is a constant $C$, which depends on $k$ and $\Sigma$, such that 
\beq\label{schifio1}
\|\nabla^{2k}f\|_{L^2(\Sigma)}^2 \leq \int_{\Sigma} (\Delta^k f)^2 \, d \Ha^{n-1} + C\|f\|_{H^{2k-1}(\Sigma)}^2
\eeq
and 
\beq\label{schifio2}
\|\nabla^{2k+1}f\|_{L^2(\Sigma)}^2 \leq \int_{\Sigma} |\nabla(\Delta^k f)|^2 \, d \Ha^{n-1} + C\|f\|_{H^{2k}(\Sigma)}^2.
\eeq
\end{lemma}

\begin{proof}
We only proof \eqref{schifio1} in the cases $k=1,2$, since the higher order cases and \eqref{schifio2} are analogous. Recall that Ricci tensor is  given by $R_{jm} = g^{ik}R_{ijmk}$.
Thus from \eqref{comm 1}, with $X$ equal to the covariant gradient of $f$ and taking $k=i$, we get 
 \[
\nabla_i \nabla_j \nabla^i f - \nabla_j \Delta f=  R_{jl} \nabla^l f.
\]
We multiply the above equality  by $\nabla^j f$ and use the integration by parts formula \eqref{int by parts} to obtain 
\[
- \int_{\Sigma} \nabla_i \nabla^j f  \, \nabla_j \nabla^i f  \, d \Ha^{n-1}  +  \int_{\Sigma} (\Delta f)^2\, d \Ha^{n-1}  = \int_{\Sigma} R_{ij} \nabla^i f \, \nabla^j f   \, d \Ha^{n-1}.
\]
This yields the claim since (recall that for a function $\nabla_i \nabla_j f = \nabla_j \nabla_i f$)
\[
\nabla_i \nabla^j f  \, \nabla_j \nabla^i f   =  \nabla^i \nabla^j f  \, \nabla_i \nabla_j f  = |\nabla^2 f|^2.
\]
The argument in the case $k=2$ is similar but more technical. We have by the previous  statement 
\[
\int_{\Sigma} |\Delta^{2} f|^2  \, d \Ha^{n-1}    \geq  \int_{\Sigma} |\nabla^2  \Delta f |^2 \, d \Ha^{n-1}   -  C\|f\|_{H^{3}(\Sigma)}^2.
\] 
Hence, we need to prove that 
\beq \label{interchange 3}
\int_{\Sigma} |\nabla^2  \Delta f |^2  \, d \Ha^{n-1}    \geq  \int_{\Sigma} |\nabla^4  f |^2 \, d \Ha^{n-1}   -  C\|f\|_{H^{3}(\Sigma)}^2.
\eeq
First, by the integration by parts formula \eqref{int by parts} we have
\[
\begin{split}
\int_{\Sigma} |\nabla^2  \Delta f |^2  \, d \Ha^{n-1} &=  \int_{\Sigma} (\nabla^i \nabla^j  \nabla_k \nabla^k  f) \, (\nabla_i\nabla_j \nabla^l \nabla_l f) \, d \Ha^{n-1}\\
&=-  \int_{\Sigma} (\nabla_i \nabla^i \nabla^j  \nabla_k \nabla^k  f) \,  (\nabla_j \nabla^l \nabla_l f) \, d \Ha^{n-1}.
\end{split}
\]
Then, using \eqref{comm 5}, we obtain
\[
\begin{split}
\int_{\Sigma} |\nabla^2  \Delta f |^2  \, d \Ha^{n-1} &\geq - \int_{\Sigma} (\nabla_k \nabla_i \nabla^i \nabla^j   \nabla^k  f) \,  (\nabla_j \nabla^l \nabla_l f) \, d \Ha^{n-1} - C\|f\|_{H^{3}(\Sigma)}^2\\
&= - \int_{\Sigma} (\nabla^i \nabla^j   \nabla^k  f) \,  (\nabla_i \nabla_k \nabla_j \nabla^l \nabla_l f) \, d \Ha^{n-1} - C\|f\|_{H^{3}(\Sigma)}^2,
\end{split}
\]
where the last equality follows by integration by parts. We proceed   using  formula   \eqref{comm 5} again and integration by parts 
 to deduce
\[
\begin{split}
\int_{\Sigma} |\nabla^2  \Delta f |^2  \, d \Ha^{n-1} &\geq  - \int_{\Sigma} (\nabla^i \nabla^j   \nabla^k  f) \,  (\nabla^l  \nabla_i \nabla_j \nabla_k \nabla_l f) \, d \Ha^{n-1} - C\|f\|_{H^{3}(\Sigma)}^2 \\
&= - \int_{\Sigma} (\nabla_i \nabla^l   \nabla^i \nabla^j   \nabla^k  f) \,  ( \nabla_j \nabla_k \nabla_l f) \, d \Ha^{n-1} - C\|f\|_{H^{3}(\Sigma)}^2\\
&\geq - \int_{\Sigma} (\nabla_i   \nabla^i \nabla^j   \nabla^k  \nabla^l  f) \,  ( \nabla_j \nabla_k  \nabla_l f) \, d \Ha^{n-1} - C\|f\|_{H^{3}(\Sigma)}^2\\
&=  \int_{\Sigma} ( \nabla^i \nabla^j   \nabla^k  \nabla^l  f) \,   ( \nabla_i   \nabla_j \nabla_k  \nabla_l f) \, d \Ha^{n-1} - C\|f\|_{H^{3}(\Sigma)}^2.
\end{split}
\]
Thus we have \eqref{interchange 3}, since $( \nabla^i \nabla^j   \nabla^k  \nabla^l  f) \,   ( \nabla_i   \nabla_j \nabla_k  \nabla_l f) = |\nabla^4 f|^2$.
\end{proof}

\begin{remark}
\label{rem aub 1}
In the case $k=1$  we have a more precise version of  Lemma \ref{interchange 1} for hypersurfaces. It is clear that the proof of Lemma \ref{interchange 1}  implies that 
\[
 \int_{\Sigma} |\nabla^2 f|^2 \, d \Ha^{n-1}  \leq \int_{\Sigma} (\Delta f)^2 \, d \Ha^{n-1} + (\sqrt{n-1}+1) \int_{\Sigma} |B|^2 |\nabla f|^2\, d \Ha^{n-1},
\]
where $B$ denotes the (scalar) second fundamental form (see \cite{Lee} for definition). This follows from the fact that we may estimate the  Ricci curvature by $|\text{Ric}| \leq (\sqrt{n-1}+1)|B|^2$. 
\end{remark}

\begin{remark}
\label{rem aub 2} Using  Lemma~\ref{aubinlemma} we may  write the statement of  Lemma~\ref{interchange 1} in the following way. For every $\e>0$ there exists $C_\e>0$ such that 
\[
\|f\|_{H^{2k}(\Sigma)}^2 \leq (1+\e) \int_{\Sigma} (\Delta^k f)^2 \, d \Ha^{n-1} + C_\e\|f\|_{L^{2}(\Sigma)}^2
\]
and
\[
\| f\|_{H^{2k+1}(\Sigma)}^2 \leq (1+\e) \int_{\Sigma} |\nabla(\Delta^k f)|^2 \, d \Ha^{n-1} + C_\e\|f\|_{L^{2}(\Sigma)}^2.
\]
Indeed, this follows by   the interpolation inequality  together with standard Young's inequality
\[
\begin{split}
\| \nabla^{l}f\|_{L^{2}(\Sigma)} &\leq C \| \nabla^{h}f\|^\theta_{L^{2}(\Sigma)} \|f\|^{1-\theta}_{L^{2}(\Sigma)} \\
&\leq \eps \| \nabla^{h}f\|_{L^{2}(\Sigma)} + C(\eps)  \|f\|_{L^{2}(\Sigma)}
\end{split}
\]
for every $1 \leq l \leq h-1$ and $\theta=\theta(h,l)$ is given by Lemma~\ref{aubinlemma}.

\end{remark}

For clarity we denote the standard inner product  between two vectors $x,y$ in  $\R^n$  as $x \cdot y$ and  the differential  of  the map   $F : \R^n \to \R^m$ by $DF$ to distinguish them from the inner  product on manifold and  from the 
covariant derivative. There is, however, a possibility of confusion when we denote the divergence of a vector field $X :  \R^n \to \R^n$ by $\Div X$, since ``$\Div$'' also denotes the divergence of a vector field on manifold. We will denote 
 the divergence of a vector field on the manifold $(\Sigma,g)$ by  $\Div_g$ and in  $\R^n$ by $\Div_{\R^n}$ if this is not clear from the context.     

When  the manifold $\Sigma$ is given by a  boundary of a smooth bounded set $F \subset \R^n$ it has a natural orientation and we denote by $\nu_F$ the unit outer normal. In this case we may extend the definition of divergence on $\Sigma$ to  vector fields which have values in $\R^n$. Let $X: U \to \R^n$ be a smooth vector field, where $U$ is an open neighborhood of $\Sigma$. We define the tangential divergence of $X$ on $\pa F$ by
\[
\Div_\tau X  := \Div X  -  \la DX \nu_F , \nu_F \ra  . 
\] 
The divergence theorem  states
$$
\int_{\pa F}  \Div_\tau X  \, d \Ha^{n-1} = \int_{\pa F}  H_F  \, ( X \cdot \nu_F )  \, d \Ha^{n-1} ,
$$
where $H_F$ denotes the sum of the principal curvatures of $\pa F$. We denote  the second fundamental form of $\pa F$ by $B_F$, which in our case is a symmetric $\binom{2}{0}$-tensor (or equivalently a symmetric matrix). Finally we may project a vector field $X: U \to \R^n$ to the 
tangent space of $\pa F$ by 
\beq
\label{proj tang}
 X_\tau:= X - (X  \cdot \nu_F)\nu_F.
\eeq
Then  $X_\tau$ canonically defines a vector field on $(\pa F,g)$  and   we denote  by $\Div_g X_\tau$ its divergence. For a given function $u : U \to \R$ we define the tangential gradient on $\Sigma = \pa F$ as  the projection of its gradient $D u$ 
\beq
\label{tang grad}
 D_\tau u :=(D u)_ \tau.
\eeq
The tangential gradient and   the covariant gradient are canonically isomorphic. In particular, it holds  
\beq
\label{cov tang}
|\nabla u (x)|_g =  |D_\tau u (x) |  \qquad \text{for } \, x \in \Sigma,  
\eeq
where $|\cdot|_g$ denotes the norm given by the metric tensor $g$, and $|\cdot |$ is the length of a vector in $\R^n$.

\subsection{The energy functional} 
In this section we introduce the energy functional that underlies the flow. We also introduce the proper notions of stationary points and stability that will be needed in the study of the long-time behavior of the flow. As explained in the introduction, the free energy functional is the sum of  the   perimeter and of a bulk elastic term. Throughout the paper $\Om$ will denote a fixed bounded open set of $\R^3$ with Lipschitz boundary.

Concerning the elastic part, for  $F \subset \! \subset \Omega$ and for an elastic displacement $u: \Omega\setminus F\to \R^2$ we denote by $E(u)$ the symmetric part of $D u$, that is, $E(u):= \frac{D u + (D u)^T}{2}$. In what follows,  $\C$ stands for    the {\em elasticity tensor}   acting on $3\times 3$-matrices, such that 
$\C A=\frac12\C (A+A^T)$ and $\C A$ is symmetric for all $3\times 3$-matrices $A$.  Moreover,  $\C A:A>0$ if $A$ is symmetric and $A\neq 0$. Finally we shall denote by 
 $Q(A) := \frac{1}{2}\C A : A$ the {\em elastic energy density}. 

We are now ready to write the energy functional. For a fixed {\em boundary displacement}  $w_0\in H^{\frac12}(\pa \Om)$, we set 
\begin{equation} \label{energy}
\mathcal{J}(F) :=  \int_{\Omega \setminus F} Q(E(u_F))\, dx + \Ha^2(\pa F)\, ,
\end{equation}
where $u_F$ is the  elastic equilibrium  satisfying the Dirichlet boundary condition $w_0$ on a fixed relatively open subset $\pa_D \Om\subseteq \pa \Om$. More precisely, $u_F$ is the unique solution in  $H^1(\Omega \setminus F; \R^3)$ of the following elliptic system
\beq\label{uf}
\begin{cases}
\Div \C E(u_F)=0 & \text{in }\Om\setminus F,\\
\C E(u_F)[\nu_F]=0 & \text{on }\pa F\cup (\pa \Om\setminus \pa_D\Om),\\
u_F=w_0 &\text {on }\pa_D\Om.
\end{cases}
\eeq
Note that by the second condition for every $x\in \pa F$ the vector $\C E(u_F)(x) [e]$ belongs to the tangent space of  $\pa F$ at $x$ for every  vector $e$.

Next, we provide the first and the second variation formulas for \eqref{energy}.
To this aim, for any  
 vector field $X \in C_c^1(\R^3; \R^3)$, let  $(\Phi_t)_{t\in (-1,1)}$  be  the associated flow, that is the solution of  
 \beq\label{flussoX}
 \begin{cases}
\displaystyle \frac{\pa \Phi_t}{\pa t}=X(\Phi_t),\\
 \Phi_0=Id.
 \end{cases}
 \eeq
The first and the second variation of the functional  \eqref{energy} are stated in the following theorem. Recall that $H_F$ denotes the sum of the principal curvatures  and $B_F$ the second 
fundamental form of $\pa F$. Sometimes, with a slight abuse of terminology, we will refer to $H_F$ as the mean curvature of $\pa F$.
\begin{theorem}\label{th:12var}
 Let $F\subset\subset\Om$ be a smooth set, $X \in C_c^1(\Om; \R^2)$ and let $(\Phi_t)_{t\in (-1,1)}$ be the associated flow as in \eqref{flussoX}. Set $\psi:=X\cdot \nu_F$   on $\pa F$ and let $X_\tau$ be as in \eqref{proj tang}. 
 Then, 
 \beq\label{eq:J'}
\frac{d}{dt}\mathcal{J}(\Phi_t(F))_{\bigl|_{t=0}}=\int_{\pa F} (H_F-Q(E(u_F))) \psi\, d\Ha^{2}. 
\eeq
If in addition $\Div_{\R^n} X=0$ in a neighborhood of $\pa F$ we have
\begin{align}\label{eq:J''}
\frac{d^2}{dt^2}\mathcal{J}(\Phi_t(F))_{\bigl|_{t=0}}&=
\int_{\pa F}   |\nabla \psi|^2-     |B_F|^2  \psi^2\, d \Ha^2  - 2 \int_{\Omega \setminus  F} Q(E(u_\psi)) \, dx\nonumber\\ 
&  -  \int_{\pa F}     \pa_{\nu_F} (Q(E(u_F)))  \psi^2 \, d \Ha^2
-\int_{\pa F}(H_F-Q(E(u_F)))\Div_g (\psi X_\tau)\, d\Ha^2,
\end{align}
where the function $u_\psi$ is the unique solution in $H^1(\Om\setminus F; \R^3)$, with 
$u_\psi=0$ on $\pa_D\Om$, of
\begin{equation}\label{eq u dot2}
\int_{\Om \setminus F} \C E(u_\psi) : E(\vphi) \, dx = -\int_{\pa F}\Div_g (\psi \, \C E(u_F) ) \cdot \vphi \, d \Ha^2 
\end{equation}
for all $\varphi \in H^1(\Om \setminus F; \R^2)$ such that $\varphi = 0$ on $\pa_D\Omega$.   
\end{theorem}
Formulas \eqref{eq:J'} and \eqref{eq:J''} have been derived  in    \cite{Bo}  when $F$ is the subgraph of a periodic function. The very same calculations apply to the more general situation considered here.

Throughout the paper we fix a smooth reference set $G \subset\!\subset \Om$ and  define the reference manifold as $(\Sigma, g)$, where $\Sigma= \pa G$ and $g$ is the metric induced by the Euclidean metric in $\R^3$. We denote the outer normal of $G$ simply by $\nu$. 
For every $\eta>0$ we denote 
$$
\mathcal{N}_\eta(\Sigma):=\{x\in \R^3:\, |d_G(x)|<\eta\},
$$
where $d_G$ denotes the signed distance function of $G$. Denote also  $\pi$ the orthogonal projection on the boundary of $G$.
Since $G$ is smooth, 
\beq\label{eta0}
\text{there exists $\eta_0>0$ such that $d_G$ and $\pi$ are smooth in $\mathcal{N}_{2\eta_0}(\Sigma)$.} 
\eeq
 We denote by $ \mathfrak{h}_M^k(\Sigma)$ the following class of sets, whose boundary is a  suitable normal graph over $\Sigma$. Precisely,  for any integer $k\geq 1$ and $M>0$ we say 
\beq\label{mathfrakh}
F \in \mathfrak{h}_M^k(\Sigma) \quad \text{if } \, \pa F = \{ x + h_F(x) \nu(x) :  x \in \Sigma\} \subset \mathcal{N}_{\eta_0}(\Sigma) \qquad \text{with} \quad \|h_F\|_{H^k(\Sigma)} \leq  M .  
\eeq
In particular,  by Morrey embedding any set in $\mathfrak{h}_M^3(\Sigma)$ is $C^{1, \alpha}$-diffeomorphic to the reference set $G$ for every $\alpha\in (0,1)$. 
The space $\mathfrak{h}_M^{k,\alpha}(\Sigma)$, $\alpha\in (0,1)$, is defined similarly in terms of the $C^{k,\alpha}$-norm of the function $h_F$.

Let    $G_1$, \dots, $G_m$ be the bounded open sets enclosed by the connected components $\Gamma_{G,1}$, \dots, $\Gamma_{G,m}$ of the boundary $\pa G$. Note that the $G_i$'s are not in general the connected components of $G$ and it may happen that $G_i\subset G_j$ for some $i \neq j$.  If  $F\in \mathfrak{h}^{3}_M(\Sigma)$, 
 then $F$ is  $C^1$-diffeomorphic to $G$ and thus $\pa F$ has the same number $m$ of connected components $\Gamma_{F,1}$, \dots, $\Gamma_{F,m}$, which can be numbered in such a way that 
\beq\label{numbered}
\Gamma_{F,i}=  \{ x + h_{F}(x) \nu(x):\, x\in \Gamma_{G,i}  \},
\eeq
 for a suitable $h_{F}\in H^{3}(\Sigma)$. The boundaries  $\Gamma_{F,i}$ then enclose the sets $F_{i}$, which in turn are diffeomorpic to $G_i$.

We are  interested in area preserving variations, in the following sense.
\begin{definition}\label{def:admissibleX}
 Let $F\subset\subset\Om$ be a smooth set.  Given a vector field 
 $X\in C^\infty_c(\Om; \R^3)$, we say that  the associated flow
  $(\Phi_t)_{t\in (-1,1)}$ is {\em admissible for $F$} if there exists $\e_0\in (0,1)$ such that
  $$
  |\Phi_t(F_i)|=|F_i|\quad\text{for $t\in (-\e_0,\e_0)$ and $i=1,\dots, m$.}
  $$
\end{definition}
\begin{remark}\label{rm:mn}
Note that if the flow associated with $X$ is admissible in the sense of the previous definition, then 
for $i=1,\dots, m$ we have 
$$
\int_{\Gamma_{F,i}}X\cdot\nu_F\, d\Ha^1=0.
$$
In view of this remark it is convenient to introduce the space
$\tilde H^1(\pa F)$   consisting of all functions $\psi \in H^1(\pa F)$ with zero average on each component of  $\pa F$, i.e., 
\[
\int_{\Gamma_{F,i}} \psi \, d \Ha^1 = 0 \qquad \text{for every } \, i = 1, \dots, m.
\]
Any admissible vector field $X$ thus defines a function $\psi \in \tilde H^1(\pa F)$. Conversely,   given $\psi\in \tilde H^1(\pa F)\cap C^{\infty}(\pa F)$ it is possible to construct a  sequence of  vector fields $X_n\in C^\infty_c(\Om; \R^2)$, with $\Div_{\R^n} X_n=0$ in a neighborhood of $\overline F$, such that 
$X_n\cdot \nu_F\to \psi$ in $C^1(\pa F)$, see \cite[Proof of Corollary~3.4]{AFM} for the details. Note that in particular   the flows associated with $X_n$ are admissible. 
\end{remark}

\begin{definition}
\label{def stationarity}
Let $F \subset \!\subset  \Omega$  be a set of class $C^2$. We say that $F$ is \emph{stationary} if
$$
\frac{d}{dt}\mathcal{J}(\Phi_t(F))_{\bigl|_{t=0}}=0
$$
for all admissible flows in the sense of Definition~\ref{def:admissibleX}.
\end{definition}

\begin{remark}\label{rm:station}
By Remark~\ref{rm:mn} and in view of \eqref{eq:J'} it follows that a set $F \subset \!\subset  \Omega$  of class $C^2$ is stationary if and only if  there exist constants $\lambda_1, \dots, \lambda_m$ such that  
\[
H_F - Q(E(u_F)) = \lambda_i \qquad \text{on }\, \Gamma_{F,i}
\]
for every $i = 1,\dots, m$. Note that if $F$ is a  sufficiently regular (local) minimizer of \eqref{energy} under the constraint 
$|F|=const.$, then  there exists a  constant $\lambda$ such that 
\[
H_F - Q(E(u_F)) = \lambda \qquad \text{on }\, \pa F.
\]
Thus, our notion of stationarity differs from the usual notion of criticality just recalled. Note that by a bootstrap argument it can be proved that 
a stationary set is smooth. In fact, it can be shown that it is even analytic, see \cite{KLM}. 
 Note that if $F$ is stationary, then the second variation formula \eqref{eq:J''} reduces to 
\begin{align} \label{sv}
\frac{d^2}{dt^2}\mathcal{J}(\Phi_t(F))_{\bigl|_{t=0}}= &
\int_{\pa F}   |\nabla \psi|^2-     |B_F|^2  \psi^2\, d \Ha^2  \nonumber\\
&- 2 \int_{\Omega \setminus  F} Q(E(u_\psi)) \, dx  
  -  \int_{\pa F}     \pa_{\nu_F} (Q(E(u_F)))  \psi^2 \, d \Ha^2,
\end{align}
where we recall that $\psi= X\cdot \nu_F$ and $u_\psi$ is the function satisfying \eqref{eq u dot2}. 
\end{remark}
 
 In view of \eqref{sv},  for any set $F\subset\subset\Om$ of class $C^2$ it is convenient to introduce the  quadratic form $\pa^2 \mathcal{J}(F)$ defined on $\tilde H^1(\pa F)$ as
\beq \label{eq:pa2J}
\begin{split}
\pa^2\mathcal{J}(F)[\psi] :=& \int_{\pa F}  |\nabla \psi|^2-     |B_F|^2  \psi^2\, d \Ha^2       \\
&- 2 \int_{\Omega \setminus  F} Q(E(u_\psi)) \, dx  
  -  \int_{\pa F}     \pa_{\nu_F} (Q(E(u_F)))  \psi^2 \, d \Ha^2,
\end{split}
\eeq  
where $u_\psi$ is the unique solution of  \eqref{eq u dot2} under the Dirichlet condition $u_\psi=0$ on 
$\pa_D\Om$. We may finally give the  definition of stability for a stationary point. 
\begin{definition}\label{def:stable}
Let $F\subset\subset\Om$ be a stationary set in the sense of Definition~\ref{def stationarity}. We say that $F$ is \emph{strictly stable} if 
\beq\label{j2>0}
\pa^2\mathcal{J}(F)[\psi]>0\qquad\text{for all }\psi\in \tilde H^1(\pa F)\setminus\{0\}.
\eeq
\end{definition}
It is not difficult to see that \eqref{j2>0} is equivalent to the coercivity of $\pa^2\mathcal{J}(F)$ on $ \tilde H^1(\pa F)$. More precisely, \eqref{j2>0} holds if and only if there exists $c_0>0$ such that 
\beq\label{emmepiccolo0}
\pa^2\mathcal{J}(F)[\psi]\geq c_0\|\psi\|^2_{\tilde H^1(\pa F)}\qquad\text{for all }\psi\in \tilde H^1(\pa F),
\eeq
see \cite{Bo}. In turn the latter coercivity property  is stable  with respect to small $H^{3}$-perturbations.  More precisely, we have:
\begin{lemma}\label{lemma:j2>0near}
Assume that the reference set $G\subset\subset\Om$ is  a (smooth) strictly stable stationary set in the sense of Definition~\ref{def:stable}. Then, there exists  $\sigma_0>0$ such that  for all $F \in  \mathfrak{h}_{\sigma_0}^3(\Sigma)$, {defined in \eqref{mathfrakh}}, we have
$$
\pa^2\mathcal{J}(F)[\psi]\geq \frac{c_0}2\|\psi\|^2_{\tilde H^1(\pa F)} \text{ for all $\psi\in \tilde H^1(\pa  F)$,}
$$
where $c_0$ is the constant in \eqref{emmepiccolo0}.
\end{lemma} 
\begin{proof}The proof follows the argument in  \cite[Proof of Theorem~5.2 and Lemma~5.3]{Bo}, where the case of $F$ being the subgraph of a periodic  function is considered. Although the geometric framework here is more general, we may  follow exactly the same line of argument up to the obvious changes due to the different setting. We note that in our case we may even simplify the aforementioned proof by taking  advantage of the fact that $F\in    \mathfrak{h}^{3}_{\sigma_0}(\Sigma)$ (while in \cite{Bo} only $W^{2,p}$-bounds were assumed). Indeed, under this assumption we have that $u_F$ is of class $H^3$ in a neighborhood of $\Sigma$, with the norm estimated by a constant depending on $\sigma_0$ (see the proof of Theorem~\ref{linearestimate}).
In turn, $\pa_{\nu_F}(Q(E(u_F)))\in H^{\frac12}(\pa F)$ with a bound depending on $\sigma_0$, which  is a much stronger information than the boundedness in $H^{-\frac12}(\pa F)$ proven in \cite{Bo}.  
\end{proof}

We conclude this section by showing that in a sufficiently small $H^{3}$-neighborhood of $G$ the stationary sets are isolated, once  we fix the areas enclosed by the connected components of the boundary. 
\begin{proposition}\label{stationary}
Assume that the reference set $G\subset\subset\Om$ is  a smooth strictly stable stationary set in the sense of Definition~\ref{def:stable} and let $\sigma_0$ be the constant provided by Lemma~\ref{lemma:j2>0near}. There exists $\sigma_1\in (0, \sigma_0)$ with the following property:  Let $F_1$, $F_2\in  \mathfrak{h}^{3}_{\sigma_1}(\Sigma)$, {defined in \eqref{mathfrakh}},  be stationary sets in the sense of Definition~\ref{def stationarity} and (with the same notation as in \eqref{numbered}) assume that $|F_{1, i}|=|F_{2,i}|$ for $i=1,\dots, m$. Then $F_1=F_2$.
\end{proposition}

\begin{proof}
Let $F_1$ and $F_2$ be in $\mathfrak{h}^3_{\sigma_1}(\Sigma)$, with $\sigma_1\in (0, \sigma_0)$ to be chosen, and   denote the components defined in \eqref{numbered} by $F_{i,1}, \dots, F_{i,m}$ for $i = 1,2$. We begin by constructing a vector field $X : \mathcal{N}_{\eta_0}(\Sigma)  \to \R^3$
such that the associated  flow $(\Phi_t)_{t\in ([0,1])}$  is admissible is sense of Definition~\ref{rm:mn} and takes the set $F_1$ to $F_2$. More precisely, it holds  $\Phi_0(F_1) = F_1$,  $\Phi_1(F_1) = F_2$ and $|\Phi_t(F_{1, i})| = |F_{1, i}|$ for every $t \in [0,1]$ and $i = 1,\dots, m$. The construction can be  done as in \cite[Proposition~3.4]{Morini}  (see also \cite[Lemma~2.8]{surf2D}) in such a way that 
 $|X(x)| \leq 2 |X(x) \cdot \nu_{F_t}(x)| $ for $x \in \pa F_t$ and for all $t\in [0,1]$, and that 
\[
\pa F_t = \{ x + h_{F_t}(x) \nu(x) : x \in \Sigma \} \qquad \text{with } \quad \| h_{F_t} \|_{H^3(\Sigma)} \leq C\sigma_1<\sigma_0,
\]
where the last inequality holds provided that $\sigma_1$ is small enough. 
Recalling \eqref{eq:J''},  \eqref{eq:pa2J},  using the Lemma \ref{lemma:j2>0near}  and by integrating by parts we get
\[
\begin{split}
\frac{d^2}{dt^2}\mathcal{J}(\Phi_t(F_1)) &= \pa^2\mathcal{J}(F_t)[ X \cdot \nu_{F_t}]   -\int_{\pa F_t}(H_{F_t}-Q(E(u_{F_t})))\Div_g ( (X \cdot \nu_{F_t}) X_\tau)\, d\Ha^2  \\
&\geq \frac{c_0}2 \|X \cdot \nu_{F_t}\|_{H^1(\pa F_t)}^2 + \int_{\pa F_t} \la  \nabla (H_{F_t}-Q(E(u_{F_t}))) ,  (X \cdot \nu_{F_t}) X_\tau \ra  \, d\Ha^2.
\end{split}
\]
We denote $R_t := H_{F_t}-Q(E(u_{F_t}))$ and  estimate the last term by \eqref{avain}, which we will show later   in the proof of Theorem \ref{thmstability}, to get 
\[
\begin{split}
 \int_{\pa F_t} \la  \nabla R_t,  (X \cdot \nu_{F_t}) X_\tau \ra  \, d\Ha^2 &\leq \left(\int_{\pa F_t} | \nabla  R_t|^2 \,  d\Ha^2  \right)^{1/2}\left(\int_{\pa F_t} |(X \cdot \nu_{F_t}) X_\tau|^2 \,  d\Ha^2  \right)^{1/2}\\
&\leq C \| h_{F_t} \|_{H^3(\Sigma)}^{\theta/2} \left(\int_{\pa F_t} |X \cdot \nu_{F_t}|^4 \,  d\Ha^2  \right)^{1/2} \\
&\leq C \sigma_1^{\theta/2} \| X \cdot \nu_{F_t}\|_{L^4(\pa F_t)}^2.
\end{split}
\]
Therefore we have by the Sobolev embedding  
\[
\begin{split}
\frac{d^2}{dt^2}\mathcal{J}(\Phi_t(F_1)) &\geq  \frac{c_0}2 \|X \cdot \nu_{F_t}\|_{H^1(\pa F_t)}^2  - C  \sigma_1^{\theta/2} \| X \cdot \nu_{F_t}\|_{L^4(\pa F_t)}^2\\
&\geq \frac{c_0}2 \|X \cdot \nu_{F_t}\|_{H^1(\pa F_t)}^2  - C  \sigma_1^{\theta/2}  \| X \cdot \nu_{F_t}\|_{H^1(\pa F_t)}^2 \geq \frac{c_0}{4} \|X \cdot \nu_{F_t}\|_{H^1(\pa F_t)}^2,
\end{split}
\] 
provided that  $\sigma_1$ is small enough. 

On the other hand by the stationarity of $F_1$ and $F_2$ we have
\[
\frac{d}{dt}\mathcal{J}(\Phi_t(F_1))_{\bigl|_{t=0}} = \frac{d}{dt}\mathcal{J}(\Phi_t(F_1))_{\bigl|_{t=1}} = 0. 
\]
This means that $\frac{d^2}{dt^2}\mathcal{J}(\Phi_t(F_1)) = 0$ and therefore $X \cdot \nu_{F_t} = 0$ on $\pa F_t$ for all $t \in (0,1)$. Therefore $t \mapsto \Phi_t(F_1)$ is constant and $F_1 = F_2$. 

\end{proof}

\section{Short time existence for the surface diffusion with a forcing term}\label{sec:forced}

In the following we shall  assume $n=3$. Given a smooth function $f:\Sigma\times[0,+\infty)\to\R$ we shall consider the following 
forced surface diffusion equation
\beq\label{forced}
V_t=\Delta_{\pa F_t}(H_{F_t}+f(\cdot,t) \circ\pi)
\eeq 
where   $V_t$ denotes the outer normal velocity of $\pa F_t$ and $\Delta_{\pa F_t}$ is the Laplace-Beltrami operator on $\pa F_t$ endowed with the metric induced by the Euclidean metric.   The goal in this section is 
to prove short time existence of a unique smooth solution of \eqref{forced} starting from $F_0$ which is close to the reference set $G$. This will be done in Theorem \ref{thm with f}.

\subsection{The flow in coordinates} \label{sec12}

Given a sufficiently smooth function $h:\Sigma\to (-\eta_0, \eta_0)$, where $\eta_0$ is introduced in \eqref{eta0},  we denote by $F_h$ the bounded open set whose boundary is given by
$$
\pa F_h=\{ x + h(x) \nu(x) : x \in \Sigma\},
$$ 
where $\nu$ is the outer unit normal to $\pa G$. Note that the projection  $\pi|_{\pa F_h}:\pa F_h\to \Sigma$ is invertible and we denote by $\pi^{-1}_{F_h}$ its inverse. In this case we have  
$\pi^{-1}_{F_h}(x) =  x + h(x) \nu(x)$.  
 
 We denote by $\nu$ the normal and by $k_1, k_2$ the principle curvatures of $\Sigma$, while   $\tau_1, \tau_2$  denote the corresponding eigenvectors on the tangent plane. The exterior normal to $F_h$   is
\beq \label{normal}
\nu_{F_h} \circ \pi^{-1}_{F_h} = \frac{1}{J} \big( (1+hk_1) (1+hk_2)\nu - (1+hk_2) \partial_{\tau_1} h \,  \tau_1- (1+hk_1) \pa_{\tau_2} h \, \tau_2 \big),
\eeq
where $J^2 = (1+h k_1)^2(1+h k_2)^2+(1+h k_1)^2(\partial_{\tau_1} h)^2+(1+h k_2)^2(\partial_{\tau_2} h)^2$. We recall (see \cite[p. 21]{MantegazzaBook}) that the mean curvature $H_{F_h}$ of  $\pa F_h$ can be written as
 $$
 H_{F_h}\circ\pi^{-1}_{F_h}=-(\nu_{F_h}\circ\pi^{-1}_{F_h}\cdot \nu)\Delta h+P(x, h, \nabla h),
 $$
where $P$ is a smooth function such that $P(\cdot, 0, 0)=H_G$, the mean curvature of the boundary of $G$.   We  rewrite the above formula as 
\beq \label{mean curvature}
 H_{F_h}\circ\pi^{-1}_{F_h} =-\Delta h+\la A(x, h, \nabla h), \nabla^2h \ra + H_G+a(x, h, \nabla h),
 \eeq
 where the tensor $A$ and the function $a$ are smooth and vanish when both $h$ and $\nabla h$ are $0$.

 Let us denote by $g_h$ the pull-back metric on $\Sigma$ induced by the diffeomorphism $\pi^{-1}_{F_h}:\Sigma\to \pa F_h$.  Since the manifold $(\pa F_h,g)$ endowed with the Euclidean metric $g$  is isometric to $(\Sigma, g_{h})$
 then for every smooth function $f$ defined on $\Sigma$ we have
 $$
\big( \Delta_{\pa F_h}(f\circ \pi)\big)\circ{\pi_{F_h}^{-1}}=\Delta_{g_h}f\,
 $$
 where $\Delta_{g_h}$ is the  Laplace-Beltrami operator on $\Sigma$ with respect to the  metric $g_h$. One can also check that (see \cite[p. 21]{MantegazzaBook}) 
 $$
 (g_h)_{ij}=g_{ij}+a_{ij}(\cdot, h, \nabla h),
 $$
 where the functions $a_{ij}$ are smooth and vanish when both $h$ and $\nabla h$ vanish,  and that we have the following expansion of the Christoffel symbols
\[
( \Gamma_{g_h})_{jk}^i =(\Gamma_{g})_{jk}^i + a_{jk}^i(x,h, \nabla h)   +  b_{jk}^{ilm}(x,h, \nabla h)  \frac{\pa^2 h}{\pa x_l \pa x_m}\, .
\]
 Above $ b_{jk}^{ilm}$ is a smooth function and  $a_{jk}^i$ is a smooth function which vanish when  $h$ and $\nabla h$ vanish.   We recall that the we may write the Laplace-Beltrami operator $\Delta_{g_h}$  as
\[
 \Delta_{g_h} f:= (g_h)^{ij}\widetilde\nabla_i\widetilde \nabla_jf,
 \]
 where $\widetilde\nabla_i\widetilde \nabla_j$ stands for the second order covariant derivatives with respect to $g_h$. Hence we get by the above formulas and  after some straightforward calculations that
\beq\label{laplacianogm}
\Delta_{g_{h}}f=\Delta f+ \la A_1(x, h, \nabla h),\nabla^2f \ra+\la A_2(x,h,\nabla h),\nabla f \ra +\la B(x, h, \nabla h), (\nabla^2h\otimes\nabla f) \ra .
\eeq
Concerning the equation of interest, assume that a smooth  flow $(F_t)_{t \in (0,T)}$ is a solution  of  \eqref{forced} and that  $\pa F_t$ can be written as 
$$
\pa F_t=\{ x + h(x,t) \nu(x) : x \in \Sigma\} .
$$
Then the normal velocity is given by $V_t=\pa_th (\nu_{F_t}\cdot \nu)$. Therefore, combining \eqref{mean curvature}  and  \eqref{laplacianogm} and after long but straightforward calculations, we may rewrite the equation \eqref{forced}   as
\beq\label{forced2}
\begin{split}
\frac{\partial h}{\partial t}=-&\Delta^2h+ \la A(x,h, \nabla h),\nabla^4h \ra \\
&+J_1(x,h, \nabla h, \nabla^2h, \nabla^3h)+J_2(x,h, \nabla h, \nabla^2h, \nabla f,\nabla^2f),
\end{split}
\eeq
where as usual $A$ is a smooth 4th-order tensor depending on $(x,h, \nabla h)$ vanishing when both $h$ and $\nabla h$ vanish, $J_1$ is given by
\beq \label{J1}
\begin{split}
J_1 = \la B_1, \,&(\nabla^3h\otimes \nabla^2h) \ra +\la B_2, \nabla^3h \ra +\la B_3, (\nabla^2h\otimes \nabla^2h\otimes\nabla^2h) \ra \\
&+\la B_4, (\nabla^2h\otimes \nabla^2h) \ra +\la B_5, \nabla^2h\ra+b_6 
\end{split}
\eeq
and $J_2$ is of the form
\beq\label{J2}
J_2=\Delta f+\la A_1, \nabla^2f \ra +\la A_2, \nabla f \ra +\la B, (\nabla^2h\otimes\nabla f) \ra.
\eeq
Here and throughout the paper we denote by $A$ (possibly with a subscript) a smooth tensor-valued function depending on $(x,h, \nabla h)$ and vanishing at $(x,0,0)$, while 
$B$  (possibly with a subscript) stands for a smooth tensor-valued function depending on $(x,h, \nabla h)$. We replace capital letters $A$ and $B$ with $a$ and $b$, respectively, in case of scalar valued functions.

\subsection{Short time existence and uniqueness}

Let us fix an initial set $F_0 \in \mathfrak{h}_{K_0}^3(\Sigma)$ which is close to $G$. Finding a solution of \eqref{forced} for a short time  with intial set $F_0$ is equivalent to 
finding a solution $h$ of \eqref{forced2} with initial datum $h(\cdot, 0) = h_{F_0} =:h_0$. This is the goal of this section and the result  is stated in the following theorem.

\begin{theorem}
\label{thm with f}
Let $f:\Sigma\times [0,+\infty)\to \R$ be a smooth function. 
Given  $\de_0>0$ and $K_0>1$, there exist $\eps_0$,  $T_0\in (0,1)$  with the following property:  if  $F_0\in \mathfrak{h}_{K_0}^3(\Sigma)$, {defined in \eqref{mathfrakh}},   
\beq\label{innominata}
\sup_{0\leq t\leq T_0}\|f(\cdot,t)\|_{L^\infty(\Sigma)}+\int_0^{T_0}\|f(\cdot,t)\|_{H^3(\Sigma)}^2\,dt\leq K_0,
\eeq
and  $\|h_0\|_{L^2(\Sigma)} < \eps_0$, where $h_0:=h_{F_0}$,  then the equation \eqref{forced2} has a unique smooth solution $(F_t)$ in $C^\infty(0,T_0; C^\infty(\Sigma)) \cap H^1(0,T_0; H^1(\Sigma))$  and 
\beq\label{Pitt1}
\sup_{0\leq t\leq  T_0}\|h(\cdot, t)\|_{L^2(\Sigma)}\leq \de_0.
\eeq
Moreover, for every integer $k \geq 0$ there exist   constants $C_k, q_k>0$, {independent of $\de_0$ and $K_0$},  such that 
\beq
\begin{split}\label{newonw}
 \sup_{0\leq t \leq T} &t^k\|h(\cdot, t)\|_{H^{2k+3}(\Sigma)}^2 +  \int_0^{T}t^k\|h(\cdot, t)\|_{H^{2k+5}(\Sigma)}^2\,dt \\
& \leq  C_k\biggl( \|h_0\|_{H^3(\Sigma)}^2 + \int_0^{T} \big(1+ \|f\|_{L^\infty(\Sigma)}^{q_k} + \sum_{i=0}^k t^i \|f(\cdot, t)\|_{H^{2i+3}(\Sigma)}^2\big)\, dt\biggr),
\end{split}
\eeq
for every $T\leq T_0$.
\end{theorem}

The proof of Theorem \ref{thm with f} is based on a fixed point argument in a carefully chosen function space and to this aim we need two lemmas. In the first one we estimate the derivatives of the nonlinear terms in \eqref{forced2}.

\begin{proposition}\label{prop:mostro}
Let $h$ and $f$ be of class $C^\infty(\Sigma)$. For every integer $k \geq1$ there exist $\tilde{C}_k>0$ and $p_k \geq 2$ such that given   $M_0>0$ there is  $\sigma_0 >0$ with the property that if  
\[
\|h\|_{H^3(\Sigma)}^2 \leq M_0  \qquad \text{and} \qquad \|h\|_{L^2(\Sigma)}\leq \sigma_0
\]
then
\begin{multline*}
\int_{\Sigma}|\nabla^k(\la A, \nabla^4h\ra)|^2+|\nabla^k J_1|^2+|\nabla^kJ_2|^2\, d\Ha^2\leq \frac14 \int_{\Sigma}|\nabla^{k+4}h|^2\, d\Ha^2 \\
+\tilde{C}_k\biggl( 1+  \|f\|_{L^{\infty}(\Sigma)}^{p_k}+\int_{\Sigma}|\nabla^{k+2} f|^2\, d\Ha^2\biggr),
\end{multline*}
where $A$, $J_1$, and $J_2$ are as in \eqref{forced2}, \eqref{J1}, and \eqref{J2}.
\end{proposition}

\begin{proof}
Recall that $A(x, h, \nabla h)$ vanishes at $(x, 0, 0)$ and thus given $\e>0$  there exists $\de\in (0,1)$ such that if $\|h\|_{C^1(\Sigma)}\leq \de$,
then by Leibniz formula  
$$
|\nabla^k(\la A, \nabla^4h\ra)|^2\leq \e |\nabla^{k+4}h|^2+C\sum_{i=1}^k|\nabla^i (A(x, h, \nabla h)|^2|\nabla^{k+4-i}h|^2.
$$
On the other hand, the assumptions on $h$ together with standard interpolation imply that  $\|h\|_{C^1}\leq \de$ and $\|h\|_{W^{2,4}}\leq 1$   when $\sigma_0$ is chosen small (depending on $M_0$). It turns out to be  convenient to set  $w:=\nabla h$. Since $\|w\|_\infty\leq \de <1$, one may check that 
\begin{align*}
\sum_{i=1}^k|\nabla^i (A(x, h, \nabla h)|^2|\nabla^{k+4-i}h|^2 &\leq C\sum_{i=1}^k|\nabla^{k+3-i} w|^2\\
&\quad+C\sum_{i=1}^k \sum_{\substack{1\leq j_1\leq \dots\leq j_{m-1}\leq i  \\ j_1+\dots+j_{m-1}\leq i \\ m\geq 2}}|\nabla^{j_1}w|^2\cdots |\nabla^{j_{m-1}}w|^2 |\nabla^{k+3-i}w|^2\\
&\leq C\sum_{i=1}^k|\nabla^{k+3-i} w|^2+C\sum_{\substack{1\leq j_1\leq \dots\leq j_{m}\leq k+2  \\ j_1+\dots+j_{m}\leq k+3 \\ m\geq 2}}|\nabla^{j_1}w|^2\cdots |\nabla^{j_{m}}w|^2. 
\end{align*}
{Then by H\"older's inequality we obtain
\[
\begin{split}
\int_{\Sigma}|\nabla^k(\la A, \nabla^4h\ra)|^2\, d\Ha^2\leq &\int_\Sigma \big( \varepsilon|\nabla^{k+3} w|^2 + C \sum_{i=1}^k |\nabla^{k+3-i} w|^2\big)\, d\Ha^2 \\
&+ C \sum_{\substack{1\leq j_1\leq \dots\leq j_{m} \leq k+2 \\ j_1+\dots+j_{m}\leq k+3 \\ m\geq 2}} \|\nabla ^{j_1} w\|_{\frac{2(k+3)}{j_l}}^{2} \cdots \|\nabla ^{j_m} w\|_{\frac{2(k+3)}{j_m}}^{2} .
\end{split}
\]
}Observe  that  for every $l=1,\dots, m-1$,    it holds by the interpolation  Lemma~\ref{aubinlemma}
$$
\|\nabla^{j_l}w\|_{\frac{2(k+3)}{j_l}}\leq C\| w\|_{H^{k+3}}^{\theta_l}\|  w\|_{\infty}^{1-\theta_l},
$$
where  $\theta_l=\frac{j_l}{k+3}$.  
To treat the last derivative we use a different interpolation:
$$
\|\nabla^{j_m}w\|_{\frac{2(k+3)}{j_m}}\leq C\|w\|_{H^{k+3}}^{\theta_m}\| \nabla  w\|_{4}^{1-\theta_m},
$$
where $\theta_m=  \frac{2j_m(k+2)}{(2k+3)(k+3)}   - \frac{1}{2k+3} < \frac{j_m}{k+3}$ (recall that $3 \leq j_m < k+3$). 
Therefore, recalling that $\|w\|_\infty$, $\|\nabla w\|_4\leq 1$, we get 
\[
\begin{split}
\int_{\Sigma}|\nabla^k(\la A, \nabla^4h\ra)|^2\, d\Ha^2\leq &\int_\Sigma \big( \varepsilon|\nabla^{k+4} h|^2 + C \sum_{i=1}^k |\nabla^{k+4-i} h|^2 \big)\, d\Ha^2 \\ 
&+
 C \sum_{\substack{1\leq j_1\leq \dots\leq j_{m} \leq k+2 \\ j_1+\dots+j_{m}\leq k+3 \\ m\geq 2}}
\ \prod_{l=1}^{m} \|w\|_{H^{k+3}}^{2\theta_l} .
\end{split}
\]
Observe that  for every choice of $j_1, \dots, j_{m}$ the sum of the corresponding $\theta_l$ satisfies
$$
\sum_{l=1}^{m} \theta_l <\sum_{l=1}^{m}\frac{ j_l}{k+3}\leq 1.
$$
Therefore by Young's inequality, by Remark~\ref{rm:aubinlemma}, and recalling that $\|  w\|_{\infty} \leq 1$, we conclude from the above inequality  that 
\beq\label{stima1}
\int_{\Sigma}|\nabla^k(\la A, \nabla^4h\ra)|^2\, d\Ha^2\leq \frac{1}{20} \int_\Sigma |\nabla^{k+4}h|^2\, d\Ha^2+ \tilde{C}_k .
\eeq
Using again $\|w\|_{\infty}\leq 1$, we have that  
\[
|\nabla^kJ_1| \leq  C\sum_{i=1}^k|\nabla^{k+3-i} w|+C\sum_{\substack{1\leq j_1\leq \dots\leq j_m\leq 2+k \\ j_1+\dots+j_m\leq 3+k \\ m\geq 2}}|\nabla^{j_1}w|\dots |\nabla^{j_m}w|.
\]
Therefore, arguing exactly as above, we have 
\beq\label{stima2}
\int_{\Sigma}|\nabla^kJ_1|^2\, d\Ha^2\leq \frac{1}{20} \int_\Sigma |\nabla^{k+4}h|^2\, d\Ha^2+ \tilde{C}_k .
\eeq
In order to control the derivatives of $J_2$ we need a slightly different argument, because we need to separate the terms involving $f$ and $h$ from each other. We recall \eqref{J2} and begin by estimating 
$$
|\nabla^k(\Delta f+\la A_1, \nabla^2f\ra)|\leq C\sum_{l=0}^k|\nabla^{l+2}f|+C\sum_{i=1}^k\sum_{\substack{1\leq j_1\leq \dots\leq j_m\leq i \\ j_1+\dots+j_m\leq i \\ m\geq 1}}|\nabla^{j_1}w|\dots |\nabla^{j_m}w||\nabla^{k+2 -i}f|.
$$
Therefore, using interpolation as above
\begin{align*}
\int_{\Sigma}|\nabla^k(\Delta f+ &\la A_1, \nabla^2f\ra)|^2\, d\Ha^2\leq C\Bigl(\|f\|_\infty^{2}+\|\nabla^{k+2}f \|_2^2\Bigr) \\
&\quad\qquad\qquad\qquad\qquad+ C\sum_{i=1}^k\sum_{\substack{1\leq j_1\leq \dots\leq j_m\leq i \\ j_1+\dots+j_m\leq i \\ m\geq 1}}
\prod_{l=1}^m\|\nabla^{j_l}w\|_{\frac{2(2+k)}{j_l}}^{2} \|\nabla^{k+2 -i}f\|_{\frac{2(2+k)}{2+k-i}}^{2}
 \\
 & \leq C\Bigl(\|f\|_\infty^{2}+\|\nabla^{k+2}f \|_2^2\Bigr) \\
&\quad+ C\sum_{i=1}^k\sum_{\substack{1\leq j_1\leq \dots\leq j_m\leq i \\ j_1+\dots+j_m\leq i \\ m\geq 1}}
\prod_{l=1}^m\|w\|_{H^{k+3}}^{2\theta(j_l)} \| w\|_{\infty}^{2(1-\theta(j_l))}\|\nabla^{2+k}f\|_2^{\frac{2(k+2-i)}{k+2}}\|f\|_\infty^{\frac{2i}{k+2}},
\end{align*}
where $\theta(j_l):=\frac{j_l(k+1)}{(k+2)^2}$. 
Observe that since $j_1+\dots+j_m\leq i$  
$$
\sum_{l=1}^{m}\Big(2\theta(j_l)+2\frac{(2+k-i)}{k+2}\Big)\leq \frac{2[(2+k)^2-i]}{(2+k)^2}<2.
$$
Therefore, using Young's inequality, we may conclude that 
\beq\label{stima3}
\int_{\Sigma}|\nabla^k(\Delta f+\la A_1, \nabla^2f\ra)|^2\, d\Ha^2\leq \frac{1}{20}\|\nabla^{k+4}h\|_2^2 + \tilde{C}_k \Bigl(1+ \|f\|_\infty^{p_k}+\|\nabla^{k+2}f \|_2^2\Bigr).
\eeq 
A similar argument, whose details are left to the reader, shows that  
$$
\int_{\Sigma}|\nabla^k\la A_2, \nabla f\ra+\la B, (\nabla^2h\otimes \nabla f)\ra|^2\, d\Ha^2\leq  \frac{1}{20} \|\nabla^{k+4}h\|_2^2+  \tilde{C}_k \Bigl(1+ \|f\|_\infty^{p_k}+\|\nabla^{k+2}f \|_2^2\Bigr).
$$
The conclusion then follows by combining this inequality with \eqref{stima1}, \eqref{stima2}, and \eqref{stima3}.
\end{proof}

In the  second lemma we ``linearize'' the terms $J_1$ and $J_2$ in the equation \eqref{forced2}. The argument  is similar to the previous one and therefore we postpone its proof until  the Appendix.

\begin{lemma}\label{lemmaJ}
Let $T\in (0,1)$ and  let  $h_1, h_2, f : \Sigma \times (0,T) \to \R$  be smooth functions  such that 
$$
\sup_{0\leq t\leq T}\|h_i(\cdot, t)\|_{H^3(\Sigma)}^2 +  \int_0^T\int_{\Sigma}|\nabla^5 h_i|^2\, d\Ha^2dt\leq M_0, 
$$
and
$$
\sup_{0\leq t\leq T}\|f(\cdot, t)\|_{L^\infty(\Sigma)}+\int_0^T\int_{\Sigma}|\nabla^3f|^2\, d\Ha^2dt\leq K_0.
$$
Then, there exists $\theta\in (0,1)$   with the following property:  for any $\e>0$ there exist $C=C(\e, K_0, M_0)>0$ and $\de=\de(\e, M_0)>0$ such that
if\, $\sup_{0\leq t\leq T}\|h_i(\cdot, t)\|_{L^2(\Sigma)}\leq \de$, $i=1,2$, then 
$$
\int_0^T\int_{\Sigma}|J_{h_2}-J_{h_1}|^2\, d\Ha^2dt\leq \e\int_0^T\int_{\Sigma}|\nabla^4h_2-\nabla^4h_1|^2\, d\Ha^2dt+CT^\theta
\sup_{0\leq t\leq T}\|h_2(\cdot, t)-h_1(\cdot, t)\|^2_{H^2(\Sigma)},
$$
where $J_h$ is defined as in \eqref{remainder}.
\end{lemma}

\medskip

\begin{proof}[\textbf{Proof of Theorem \ref{thm with f}}] Given $K_0$, let us define the set  $\mathcal{S}$  of functions in 
$C^\infty(0,T_0; C^\infty(\Sigma)) \cap H^1(0,T_0; H^1(\Sigma))$, which satisfy 
\beq \label{hirvio}
\sup_{0 \leq t \leq T_0} \|h(\cdot,t)\|_{L^2(\Sigma)} \leq \sigma_0 \quad \text{and} \quad  \sup_{0 \leq t \leq T_0} \|h(\cdot,t)\|_{H^3(\Sigma)}^2 +   \int_0^{T_0} \|h(\cdot,t)\|_{H^5(\Sigma)}^2 \, dt \leq M_0,
\eeq
where the constants $M_0$ and $\sigma_0$ will be chosen later. We also define a subclass $\mathcal{S}' \subset \mathcal{S}$ of functions  which satisfy the additional requirement \eqref{newonw}, where the constants  $C_k$ and $q_k$ will again be chosen later. The goal is to obtain a solution of   \eqref{forced2} in $\mathcal{S}'$ which is unique in $\mathcal{S}$.  

We begin by assuming that $h_0$ is smooth with   $\| h_0\|_{H^{3}(\Sigma)} < K_0$ and $2\|h_0\|_{L^2(\Sigma)} \leq\sigma_0$.  We now define a map $\mathscr{L}:  
\mathcal{S}\to  C^\infty(0,T_0; C^\infty(\Sigma))$ by setting $\mathscr{L}(h):=\tilde h$, where $\tilde h : \Sigma \times [0,\infty) \to \R$ 
is the solution of 
\beq \label{flow linear}
\begin{cases}
\displaystyle\frac{\pa \tilde{h}}{\pa t} =  -\Delta^2 \tilde{h} +J_h(x,t) \, \vspace{5pt}\\
\tilde h(\cdot, 0)=h_0
\end{cases}
\eeq
and where  we have set 
\beq \label{remainder}
J_h(x,t):= \la A(x,h, \nabla h),\nabla^4h \ra +J_1(x,h, \nabla h, \nabla^2h, \nabla^3h)+J_2(x,h, \nabla h, \nabla^2h, \nabla f,\nabla^2f)
\eeq
with $A$, $J_1$, $J_2$  as in \eqref{forced2}.

We note that the set $\mathcal{S}'$ is nonempty when the constants $C _k$ are chosen properly. To see this consider the solution  $\bar{h}$ of   
\beq\label{accabar}
\begin{cases}
\displaystyle\frac{\pa \bar{h}}{\pa t} =  -\Delta^2 \bar{h} \, \vspace{5pt} \\
\bar h(\cdot, 0)=h_0.
\end{cases}
\eeq
By classical regularity estimates $\bar h$ is smooth and satisfies $\sup_{0\leq t \leq 1} \|\bar h(\cdot, t)\|_{L^{2}(\Sigma)} \leq  \|h_0\|_{L^{2}(\Sigma)}$ and
\[
 \sup_{0\leq t \leq 1} t^k\|h(\cdot, t)\|_{H^{2k+3}(\Sigma)}^2 +  \int_0^{1}t^k\|h(\cdot, t)\|_{H^{2k+5}(\Sigma)}^2\,dt \leq C_k'\, \|h_0\|_{H^3(\Sigma)}^2 
\]
for all integers $k\geq 0$, and therefore $\bar h \in \mathcal{S}'$ provided that we choose $M_0$ sufficiently large. We remark that in  Steps~1 and 2 below we give an argument which can be applied to prove the above estimate.

\medskip

\textbf{Step 1:} In this step we prove that if $h \in \mathcal{S}$ then  $\tilde h = \mathscr{L}(h) \in \mathcal{S}$  for a suitable choice  of  $M_0$, $\sigma_0$ and $T_0$.

To prove this we multiply \eqref{flow linear} by $\Delta^3 \tilde{h}$. Integrating by parts both sides we get 
\begin{align*}
\frac{\pa}{\pa t} \frac12  \int_\Sigma |\nabla(\Delta\tilde{h})|^2 \, d \Ha^2=-\int_\Sigma \frac{\pa \tilde{h}}{\pa t} \Delta^3 \tilde{h} \, d \Ha^2&=
 \int_\Sigma (\Delta^2 \tilde{h}-J_h) \Delta^3 \tilde{h} \, d \Ha^2\\
 &= \int_\Sigma \big(-|\nabla(\Delta^2 \tilde{h})|^2+ \langle \nabla J_h, \nabla(\Delta^2 \tilde{h})\rangle\big) \, d \Ha^2.
\end{align*}
By Proposition~\ref{prop:mostro} it follows that if $\sigma_0$ is sufficiently   small  then
\[
\begin{split}
\frac{\pa}{\pa t} \frac12  \int_\Sigma |\nabla(\Delta\tilde{h})|^2 \, d \Ha^2 &\leq  -\frac12 \int_\Sigma   |\nabla(\Delta^2 \tilde{h})|^2   \, d \Ha^2+\frac12\int_\Sigma|\nabla J_h|^2\, d\Ha^2\\
&\leq -\frac12 \int_\Sigma   |\nabla(\Delta^2 \tilde{h})|^2   \, d \Ha^2 +   \frac38 \int_\Sigma   |\nabla^5 h|^2    \, d \Ha^2 \\
&\quad+\frac32\tilde C_1\biggl(1  +\|f\|_{L^{\infty}(\Sigma)}^{q_0}+\int_{\Sigma}|\nabla^{3} f|^2\, d\Ha^2\biggr) ,
\end{split}
\]
where $q_0 = p_{1}$ and $\tilde C_1$ are  from the Proposition~\ref{prop:mostro}.
 Integrate this over $(0,t)$ with $t \leq T_0$, where  $T_0$ will be chosen later, and get
\beq \label{crucial estimate}
\begin{split}
 \int_\Sigma &|\nabla(\Delta \tilde{h}(\cdot,t))|^2 \, d \Ha^2   -\int_\Sigma |\nabla(\Delta h_0)|^2 \, d \Ha^2 + \int_0^t \int_\Sigma   |\nabla(\Delta^2 \tilde{h})|^2 \,d\Ha^2 ds \\
 &\leq   \frac34 \int_0^{T_0} \int_\Sigma   |\nabla^5   h(\cdot, t)|^2    \, d \Ha^2    \, d t +  3 \tilde C_1 \int_0^{T_0}  \biggl(1+    \|f(\cdot, t)\|_{L^{\infty}(\Sigma)}^{q_0}+ \int_\Sigma  |\nabla^3 f(\cdot, t)|^2\, d \Ha^2\biggr) dt  . 
\end{split}
\eeq
From this estimate, from the fact that $h$ satisfies \eqref{hirvio},  $f$ satisfies \eqref{innominata}, $\| h_0\|_{H^3(\Sigma)} < K_0$ and  using Remark~\ref{rem aub 2} (with a sufficiently small $\e$) we obtain
\begin{multline}\label{stima4}
\sup_{0\leq t\leq T_0}\|\tilde h(\cdot, t)\|_{H^3}+\int_0^{T_0} \|\tilde h(\cdot, t)\|_{H^5}^2\, dt\\ \leq C\sup_{0\leq t\leq T_0}\|\tilde h(\cdot, t)\|_{L^2}+   K_0^2 +  \frac{4}{5}M_0 + 4 \tilde C_1 (( T_0+ T_0 K_0^{q_0} ) + K_0) .
\end{multline}
In order to estimate the $L^2$-norm of $\tilde h$,  we  multiply the equation \eqref{flow linear} by $\tilde{h}$. Recalling \eqref{remainder} and  using the $H^3$-bound on $h$ and the interpolation  Lemma~\ref{aubinlemma} we get  
\beq \label{for step 3}
\begin{split}
&\int_{\Sigma} \frac{\partial\tilde{h}}{\pa t} \tilde{h} \, d \Ha^2 =- \int_{\Sigma} \Delta^2 \tilde{h}\,  \tilde{h}    \, d \Ha^2  +  \int_{\Sigma} J_h   \tilde{h}      \, d \Ha^2 \\
&\leq  \int_{\Sigma} (  - |\Delta \tilde{h}|^2  + C_{\eta} \tilde{h}^2 )    \, d \Ha^2  +  \eta \int_{\Sigma} J_h^2   \, d \Ha^2 \\
&\leq C_{\eta} \int_{\Sigma} \tilde{h}^2    \, d \Ha^2  \\
&\quad+  C \eta\int_{\Sigma} \left(1+ (1+|\nabla^2 h|^2) |\nabla^3h|^2+  |\nabla^2 h|^6 +
(1+ |\nabla^2h|^2)(|\nabla f|^2+ |\nabla^2 f|^2 )\right)   \, d \Ha^2 \\
&\leq C_{\eta} \int_{\Sigma} \tilde{h}^2    \, d \Ha^2  +  C\eta\biggl(1+ \|f\|_{L^\infty}^2+ \int_{\Sigma}(|\nabla^5 h|^2+  |\nabla^3 f|^2)\, d \Ha^2\biggr),
\end{split}
\eeq
for some  $C>0$ depending  on $M_0$ and  $K_0$. Integrating this over $(0,t)$  and  using the fact that $h$ satisfies \eqref{hirvio}  and $f$  satisfies \eqref{innominata}  yield 
\[
\begin{split}
 \frac12  \int_\Sigma  \tilde{h}(\cdot,t)^2 \, d \Ha^2-  \frac12  \int_\Sigma  h_0^2 \, d \Ha^2 &\leq C_{\eta} T_0 \sup_{0\leq t \leq T_0}\|\tilde{h}(\cdot, t)\|_{L^2(\Sigma)}^2 \\
&+  \tilde C \eta\biggl(T_0 + T_0 K_0^2 + M_0 + K_0 \biggr).
\end{split}
\]
Hence, recalling that $\|h_0\|_{L^2(\Sigma)}\leq \frac{\sigma_0}{2}$  we have
$$
\sup_{0 \leq t \leq T} \|\tilde{h}(\cdot,t)\|_{L^{2}}^2 \leq \frac{\sigma_0^2}4 + 2C_{\eta} T_0\sup_{0 \leq t \leq T_0} \|\tilde{h}(\cdot,t)\|_{L^{2}}^2+ 2 \tilde C \eta\biggl(T_0 + T_0 K_0^2 + M_0 + K_0 \biggr) .
$$
From this inequality, choosing $\eta$ and $T_0$ sufficiently small (depending on $M_0$ and $K_0$)  we conclude that 
$$
\sup_{0 \leq t \leq T_0} \|\tilde{h}(\cdot,t)\|_{L^{2}(\Sigma)}\leq \sigma_0.
$$
In turn, since $\sigma_0\leq 1$, we may choose $M_0$  sufficiently large (depending on $K_0$) and $T_0$ smaller if needed to  deduce that  
from \eqref{stima4}  that 
$$
\sup_{0\leq t\leq T_0}\|\tilde h(\cdot, t)\|_{H^3(\Sigma)}+ \int_0^{T_0} \|\tilde h(\cdot, t)\|_{H^5(\Sigma)}^2\, dt\leq M_0.
$$
This concludes the proof of the fact that $\tilde h = \mathscr{L}(h)$ satisfies \eqref{hirvio} and thus belongs to $\mathcal{S}$.

\medskip

\textbf{Step 2:} Let us now prove that if $h \in \mathcal{S}'$ then  $\tilde h = \mathscr{L}(h) \in \mathcal{S}'$, i.e., it satisfies \eqref{newonw}  with $h$ replaced by $\tilde h$. We begin by observing that  the case $k= 0$ can be proven by a similar argument as the one used in Step 1, by combining \eqref{crucial estimate},  \eqref{for step 3} and replacing $T_0$ by $T \leq T_0$. We proceed by induction and assume that  \eqref{newonw}  holds for $k-1$ and prove it for $k$.  We argue similarly as in the previous step and multiply the equation \eqref{flow linear} by $\Delta^{2k +3} \tilde h$, and  after integrating by parts the left-hand side $(2k+3)$-times and the right-hand side $(2k+1)$-times and using Proposition~\ref{prop:mostro} with $k$ replaced by $2k +1$   we get
\[
\begin{split}
\frac{\pa}{\pa t} \frac12  \int_\Sigma |\nabla (\Delta^{k+1} \tilde{h})|^2 \, d \Ha^2 &\leq  -\frac12 \int_\Sigma   |\nabla (\Delta^{k+2} \tilde{h})|^2   \, d \Ha^2 + \frac12 \int_\Sigma  |\nabla^{2k+1} J_h |^2     \, d \Ha^2 \\
&\leq     -\frac12 \int_\Sigma   |\nabla (\Delta^{k+2} \tilde{h})|^2   \, d \Ha^2  +   \frac{3}{8} \int_\Sigma   |\nabla^{2k +5} h|^2    \, d \Ha^2 \\
&\quad+\frac32\tilde{C}_{2k+1}\biggl( 1+\|f (\cdot, t)\|_{L^{\infty}}^{p_{2k +1}} +\int_{\Sigma}|\nabla^{2k +3} f|^2\, d\Ha^2\biggr).
\end{split}
\]
From this estimate we obtain   
\[
\begin{split}
\frac{\pa}{\pa t}\biggl( t^k   \int_\Sigma &|\nabla(\Delta^{k+1} \tilde{h})|^2 \, d \Ha^2 \biggr)  \leq  k\, t^{k-1}   \int_\Sigma |\nabla(\Delta^{k+1} \tilde{h})|^2 \, d \Ha^2 -t^k \int_\Sigma   |\nabla(\Delta^{k+2} \tilde{h})|^2   \, d \Ha^2\\
 &+   \frac{3}{4}t^k  \int_\Sigma   |\nabla^{2k +5} h|^2      \, d \Ha^2 +  3\tilde{C}_{2k+1} t^k \biggl( 1+\|f (\cdot, t)\|_{L^{\infty}}^{p_{2k +1}} +\int_{\Sigma}|\nabla^{2k +3} f|^2\, d\Ha^2\biggr).
\end{split}
\]
Integrating this inequality over $(0, t)$ for $t\leq T$  yields
\[
\begin{split}
&\sup_{0\leq t \leq T} t^k \int_\Sigma |\nabla(\Delta^{k+1} \tilde{h})|^2 \, d \Ha^2 +  \int_0^{T}t^{k} \int_\Sigma   |\nabla(\Delta^{k+2} \tilde{h})|^2   \, d \Ha^2 dt \\
  &\leq k \int_0^{T} t^{k-1} \int_\Sigma |\nabla^{2k+3} \tilde{h}|^2\, d\Ha^2 dt \\
&+   \int_0^{T}  t^k \biggl(\frac34   \int_\Sigma   |\nabla^{2k +5} h|^2    \, d \Ha^2 + 3\tilde{C}_{2k+1}  (1+\|f(\cdot, t)\|_{L^{\infty}}^{p_{2k +1}} +\int_{\Sigma}|\nabla^{2k +3} f|^2\, d\Ha^2)\biggr)dt.
\end{split}
\]
By using the fact that $\tilde h$ satisfies \eqref{newonw} with $k-1$, and $h$ satisfies \eqref{newonw} we deduce
\[
\begin{split}
&\sup_{0\leq t \leq T} t^k \|\nabla(\Delta^{k+1} \tilde{h})\|_{L^2(\Sigma)}^2  +\int_0^{T}t^k \int_\Sigma   |\nabla(\Delta^{k+2} \tilde{h})|^2   \, d \Ha^2 dt \\
&\leq  (kC_{k-1} + \frac{3}{4}C_k +3\tilde{C}_{2k+1}) \biggr(    \int_0^{T_0} (\|h_0\|^2_{H^3(\Sigma)}+ 3+3 \|f(\cdot, t)\|_{L^{\infty}(\Sigma)}^{q_k } + \sum_{i=0}^k t^i   \| f(\cdot, t)\|_{H^{2i +3 }}^2 ) \, dt\biggr)
\end{split}
\]
when we choose $q_k \geq \max\{q_{k-1}, p_{2k+1} \}$. Using the fact that $\sup_{0\leq t \leq T_0} \| \tilde{h}(\cdot, t)\|_{L^2}  \leq \sigma_0$ and by  Remark~\ref{rem aub 2}, we obtain the estimate \eqref{newonw} for $\tilde h$ by choosing  $C_k$ large enough.

\medskip

\textbf{Step 3:} In this step we prove that the map $\mathscr{L}$ introduced in the previous step is a contraction with respect to a suitable norm, provided that $\sigma_0$ and $T_0$ are chosen sufficiently small.

To this aim,  let $h_1, h_2 \in \mathcal{S}$ and let $\tilde{h}_1, \tilde{h}_2 \in \mathcal{S}$  be the corresponding solutions of \eqref{flow linear}.
   Multiplying the equation satisfied by $\tilde h_i$ by $\Delta^2 (\tilde{h}_2-\tilde{h}_1)$, subtracting and integrating by parts we get
\[
\begin{split}
\frac{\pa }{\pa t} \frac12 &\int_{\Sigma} |\Delta (\tilde{h}_2(\cdot, t)-\tilde{h}_1(\cdot, t))|^2 \, d \Ha^2\\
 &=- \int_{\Sigma} |\Delta^2 (\tilde{h}_2-\tilde{h}_1)(\cdot, t)|^2  \, d \Ha^2 +  \int_{\Sigma}  \Delta^2  (\tilde{h}_2-\tilde{h}_1)(\cdot, t) ( J_{h_2}(\cdot,t) - J_{h_1}(\cdot,t))  \, d \Ha^2\\
&\leq -\frac12 \int_{\Sigma} |\Delta^2 (\tilde{h}_2-\tilde{h}_1)(\cdot, t)|^2  \, d \Ha^2 +  \frac12\int_{\Sigma}  |J_{h_2}(\cdot,t) - J_{h_1}(\cdot,t)|^2  \, d \Ha^2.
\end{split}
\]
Fix $\e>0$ small. By choosing $\sigma_0$ smaller in \eqref{hirvio} if needed, we may integrate the above inequality over $(0,t)$, with $t < T_0$, and use Remark~\ref{rem aub 2} and Lemma~\ref{lemmaJ} to  obtain
\beq\label{boh0}
\begin{split}
& \|\tilde h_2(\cdot, t)-\tilde h_1(\cdot, t)\|_{H^2(\Sigma)}^2+  \int_0^{T_0} \int_{\Sigma}  |\nabla^4 (\tilde{h}_2-\tilde{h}_1)|^2 \, d \Ha^2 dt\\
 &\leq C\|\tilde h_2(\cdot, t)-\tilde h_1(\cdot, t)\|_{L^2(\Sigma)}^2+ C \int_0^{T_0}  \int_{\Sigma}  |\tilde{h}_2-\tilde{h}_1|^2 \, d \Ha^2  dt  \\
 &\quad +\eps \int_0^{T_0}  \int_{\Sigma}  |\nabla^4 (h_2 - h_1) |^2  \, d \Ha^2 dt +   C T_0^\theta \sup_{0 \leq t \leq T_0}  \|h_2(\cdot, t)-h_1(\cdot, t)\|_{H^2(\Sigma)}^2\\
&\leq C\sup_{0\leq t\leq T_0}\|\tilde h_2(\cdot, t)-\tilde h_1(\cdot, t)\|_{L^2(\Sigma)}^2\\
&\quad + \eps \int_0^{T_0}  \int_{\Sigma}  |\nabla^4 (h_2 - h_1) |^2  \, d \Ha^2 dt +   C T_0^\theta \sup_{0 \leq t \leq T_0}  \|h_2(\cdot, t)-h_1(\cdot, t)\|_{H^2(\Sigma)}^2\,.
\end{split}
\eeq

Next we have to estimate the first term on the right-hand side. To this aim we multiply  the equations satisfied by $\tilde h_1$ and $\tilde h_2$  by $\tilde{h}_2-\tilde{h}_1$,   subtract and get
\[
\begin{split}
\frac{\pa }{\pa t} \frac12 &\int_{\Sigma} |\tilde{h}_2(\cdot, t)-\tilde{h}_1(\cdot, t)|^2 \, d \Ha^2 =\int_{\Sigma}  (\tilde{h}_2(\cdot, t)-\tilde{h}_1(\cdot, t)) \, \frac{\pa }{\pa t}(\tilde{h}_2-\tilde{h}_1)(\cdot, t)  \, d \Ha^2   \\
 &=- \int_{\Sigma} (\tilde{h}_2(\cdot, t)-\tilde{h}_1(\cdot, t)) \Delta^2 (\tilde{h}_2-\tilde{h}_1)(\cdot, t)  \, d \Ha^2 \\
 &\quad +  \int_{\Sigma}  (\tilde{h}_2(\cdot, t)-\tilde{h}_1(\cdot, t)) ( J_{h_2}(\cdot, t) - J_{h_1}(\cdot, t))  \, d \Ha^2\\
&\leq - \int_{\Sigma} |\Delta (\tilde{h}_2-\tilde{h}_1)(\cdot, t)|^2  \, d \Ha^2 + \frac12 \int_{\Sigma}  |\tilde{h}_2(\cdot, t)-\tilde{h}_1(\cdot, t)|^2  \, d \Ha^2\\
&\quad + \frac12 \int_{\Sigma}  | J_{h_2}(\cdot, t) - J_{h_1}(\cdot, t) |^2  \, d \Ha^2.
\end{split}
\]
Integrating  over $(0,t)$, with $t <T_0$, and using again Lemma~\ref{lemmaJ} we get
\[
\begin{split}
\int_{\Sigma} \,|\tilde{h}_2(\cdot, t)-&\tilde{h}_1(\cdot, t)|^2 \, d \Ha^2 \leq T_0   \sup_{0 \leq t \leq T_0}  \|\tilde{h}_2(\cdot, t)-\tilde{h}_1(\cdot, t)\|_{L^2(\Sigma)}^2  \\
&+ \e\int_0^{T_0}\int_{\Sigma}|\nabla^4h_1-\nabla^4h_2|^2\, d\Ha^2dt+CT_0^\theta
\sup_{0\leq t\leq T_0}\|h_1(\cdot, t)-h_2(\cdot, t)\|^2_{H^2(\Sigma)},
\end{split}
\]
from which it follows that 
\begin{multline}\label{boh}
\sup_{0 \leq t \leq T_0}  \|\tilde{h}_2(\cdot, t)-\tilde{h}_1(\cdot, t)\|_{L^2(\Sigma)}^2\leq 2\e\int_0^{T_0}\int_{\Sigma}|\nabla^4h_1-\nabla^4h_2|^2\, d\Ha^2dt\\
+
2 CT_0^\theta
\sup_{0\leq t\leq T_0}\|h_1(\cdot, t)-h_2(\cdot, t)\|^2_{H^2(\Sigma)},
\end{multline}
provided that $T_0 \leq \frac12$. Combining \eqref{boh0} and \eqref{boh}, and taking $\e$ small and $T_0$ smaller if needed, we deduce that
\begin{multline}\label{contra}
\sup_{0 \leq t \leq T_0}  \|\tilde{h}_2(\cdot, t)-\tilde{h}_1(\cdot, t)\|_{H^2(\Sigma)}^2 + \int_0^{T_0} \int_{\Sigma}  |\nabla^4 (\tilde{h}_2-\tilde{h}_1)|^2 \, d \Ha^2 dt\\
\leq \frac12\biggl(\sup_{0 \leq t \leq T_0}  \|{h}_2(\cdot, t)-{h}_1(\cdot, t)\|_{H^2(\Sigma)}^2 + \int_0^{T_0} \int_{\Sigma}  |\nabla^4 ({h}_2-{h}_1)|^2 \, d \Ha^2 dt\biggr).
\end{multline}

\medskip

\noindent {\bf Step 4.} (Conclusion) We may proceed with a standard argument, by recursively setting $h_1=\bar h$, with $\bar h$ defined as in \eqref{accabar},  and $h_n:=\mathcal{L}(h_{n-1})$ and  for every $n\geq 2$.   From \eqref{contra} we have that there exists $h$ such that
$h_n\to h$ in $L^{\infty}(0,T_0; H^2(\Sigma))\cap L^2(0, T_0; H^4(\Sigma))$. Moreover, from Step 1 and Step 2 we have also that 
$h_n\wto h$ weakly in $H^1_{loc}(0, T; H^k(\Sigma))$  and that $h$ satisfies 
\eqref{Pitt1} and \eqref{newonw}. Using these convergences one can easily pass to the limit in the equations satisfied by the $h_n$'s to conclude that $h$ is a solution of \eqref{forced}. We remark that the smoothness of $h$ in time follows from the equation and from the regularity in space of $h$. Note that the smoothness assumption on $h_0$ can be removed by a standard approximation argument.   
Finally, the uniqueness follows from the same argument used to prove \eqref{contra}.
\end{proof}

\section{Short time existence for the surface diffusion flow with elasticity}\label{sec:existence}

Here we will prove the existence of the flow
\beq 
\label{FLOW}
V_t = \Delta_{\pa F_t} (H_{F_t} -Q(E(u_{F_t})) )
\eeq
where $u_{F_t}$ is the minimizer of the elastic energy, that is the solution to \eqref{uf}, with $F$ replaced bu $F_t$. 

The most crucial point for the proof of the short time existence of \eqref{FLOW}, is to prove sharp regularity estimates for $u_{F}$ up to the boundary $\pa F$ in terms of 
 regularity of $\pa F$. We prove this in the theorem below.

\begin{theorem}
\label{linearestimate}
Let $K >0$, $\alpha\in (0,1)$, and let $k\geq 3$ be an integer. There exists   $C_k = C_k(K)>0$ such that if $h\in H^k(\Sigma)$ and
$F_h\in \mathfrak{h}^{1,\alpha}_K(\Sigma)$, {defined as in \eqref{mathfrakh}}, then 
\beq\label{prima}
\|Q(E(u_{F_h}))\circ\pi^{-1}_{F_h}\|_{H^{k-\frac32}(\Sigma)}\leq C_k(\|h\|_{H^k(\Sigma)} +1).
\eeq
Moreover  if $h_1$, $h_2\in H^{3}(\Sigma)$ and  $F_{h_i}\in \mathfrak{h}^{3}_K(\Sigma)$ for $i=1, 2$, then there exists  $C = C(K)>0$ such that
\beq\label{seconda}
\| u_{F_{h_2}}\circ \pi^{-1}_{F_{h_2}} -    u_{F_{h_1}} \circ \pi^{-1}_{F_{h_1}}\|_{H^{3/2}(\Sigma)}\leq C \|h_2-h_1\|_{H^2(\Sigma)}.
\eeq
\end{theorem}

\begin{proof}
We begin by proving \eqref{prima}. By standard approximation  argument  we may assume that $h$ is smooth, which implies that $u_{F_h}$ is smooth up to the boundary $\pa F_h$.  

We consider  a diffeomorphism $\Phi_h:\Om\setminus G\to \Om\setminus F_h$ such that  
$$
\Phi_h(x)=x+h(\pi(x))\nu(\pi(x))
$$
in $\mathcal N^+_{\eta_0}(G)$, where for any $\sigma>0$  $\mathcal N^+_\sigma(G)=\{x\in \Om\setminus G:\, d_G\leq \sigma \}$ is the one-sided neighborhood of $\Sigma$ . Note that we may construct $\Phi_h$ such that 
$\|\Phi_h-I\|_{H^k(\Om\setminus G)}+ \|\Phi^{-1}_h-I\|_{H^k(\Om\setminus G)}\leq C\|h\|_{H^k(\Sigma)}$.
 

Let us fix $x_0 \in \Sigma$. There exists a smooth diffeomorphism $\Phi$ from  a neighborhood $U$  of $x_0$ to a ball $B_{2R}$ which straightens the boundary such that $\Phi(U \cap (\Omega \setminus G)) = B_{2R}^+ = B_{2R} \cap \{x_3 >0\}$. Setting $v = u_{F_h} \circ \Phi_h\circ \Phi^{-1}$ and $\bar h:= h\circ \pi\circ \Phi^{-1}$,  $v$ is a solution of a system of the form
\beq\label{quasi eq}
\int_{B_{2R}^+} \mathbb{A}(x, \bar h, D \bar h) Dv : D\varphi \, dx = 0 
\eeq
for all $\varphi \in C^\infty( B_{2R}^+; \R^3)$ vanishing on $\pa B_{2R} \cap \{x_3 >  0\}$, where  the tensor $\mathbb{A}$ is smooth. In particular, by using the explicit definition of $\bar h$ and Lemma~\ref{lm:taldeitali} it holds $\| \bar h \|_{H^k(B_{2R}^+)} \leq C(k) (1+\| h \|_{H^k(\Sigma)})$ for every $k \in \mathbb{N}$. Moreover, by using Korn's inequality, one may check that $\mathbb{A}$ is  elliptic in the  sense that 
\beq 
\label{quasi elliptic}
\int_{B_{2R}^+} \mathbb{A}(x, \bar h, D \bar h) D \varphi : D\varphi \, dx  \geq c_0 \int_{B_{2R}^+} |D \varphi|^2 \, dx,  
\eeq
for all $\varphi \in C^\infty( B_{2R}^+; \R^3)$ vanishing on $\pa B_{2R} \cap \{x_3 >  0\}$.

Let us fix $k \geq 3$ and a multi-index $\beta=(\beta_1, \beta_2, 0)$, with $\beta_1+\beta_2= k-1$. 
By differentiating the equation \eqref{quasi eq} in the $\beta$-directions we have 
\[
\int_{B_{2R}^+} D^\beta (\mathbb{A}(x, \bar h, D \bar h) D v) : D\varphi \, dx = 0.  
\]
Let $\eta \in C_0^\infty(B_{2R})$ be a standard cut-off function such that $\eta \equiv 1$ in $B_{R}$ and $0 \leq \eta \leq 1$.  By choosing  $\varphi = D^\beta v  \eta^2$ as a test function in \eqref{quasi eq}  and by expanding the term $D^\beta ((\mathbb{A}(x, \bar h, D \bar h) D v)$ by Leibniz formula we deduce
\[
\begin{split}
\int_{B_{2R}^+}  &(\mathbb{A}(x, \bar h, D \bar h) D D^\beta v ) : D D^\beta v \eta^2 \, dx \leq 2 \int_{B_{2R}^+}  |\mathbb{A}(x, \bar h, D \bar h)| |D D^\beta v|  |D\eta | \eta |D^\beta v| \,  dx  \\
&+ C \sum_{i=1}^{k-1} \int_{B_{2R}^+}  |D^i \mathbb{A}(x, \bar h, D \bar h)| |D^{k-i} v|  ( |D D^\beta v| \eta^2 + |D^\beta v| |D\eta | \eta ) \,  dx 
\end{split}
\] 
By the ellipticity condition \eqref{quasi elliptic} we have
\[
\begin{split}
\frac{c_0}{2} \int_{B_{2R}^+} | D (D^\beta v)|^2 \eta^2 \, dx &\leq   c_0 \int_{B_{2R}^+} | D (D^\beta v \eta)|^2 \, dx  +  c_0 \int_{B_{2R}^+} | D^\beta v|^2 |D\eta|^2  \, dx  \\
&\leq  \int_{B_{2R}^+} (\mathbb{A}(x, \bar h, D \bar h) D (D^\beta v \eta) ) : D (D^\beta v \eta) \, dx  +  c_0\int_{B_{R}^+} | D^\beta v|^2 |D\eta|^2  \, dx \\
&\leq \int_{B_{2R}^+}  (\mathbb{A}(x, \bar h, D \bar h) D D^\beta v ) : D D^\beta v \eta^2 \, dx \\
&\qquad+  C \int_{B_{2R}^+}  |D D^\beta v|  |D\eta | \eta |D^\beta v| + | D^\beta v|^2 |D\eta|^2 \,  dx,
\end{split}
\]
where in the last inequality we have used fact that $\|\bar h \|_{C^{1,\alpha}} \leq C$, which in turn implies that $\mathbb{A}(x, \bar h, D \bar h)$ is bounded. Combining the previous estimates and using Young's inequality we obtain
\beq 
\label{quasi estimate}
 \int_{B_{R}^+} | D (D^\beta v)|^2  \, dx  \leq C \int_{B_{2R}^+} |D^{k-1} v |^2 \, dx  + C \sum_{i=1}^{k-1} \int_{B_{2R}^+}  |D^i \mathbb{A}(x, \bar h, D \bar h)|^2 |D^{k-i} v|^2 \,  dx. 
\eeq 
We denote $w = D \bar h$ and estimate by Leibniz formula  
\[
 \sum_{i=1}^{k-1} |D^i \mathbb{A}(x, \bar h, D \bar h)|^2 |D^{k-i} v|^2\leq C\sum_{i=1}^{k-1} |D^{k-i} v|^2+  C\sum_{i=1}^{k-1}\sum_{\substack{1\leq j_1\leq \dots\leq j_m\leq i \\ j_1+\dots+j_m\leq i \\ m\geq 1}}|D^{j_1}w|^2\cdots |D^{j_m}w|^2|D^{k-i}v|^2.
\]
Then by H\"older's inequality we get
\[
\begin{split}
 \sum_{i=1}^{k-1} \int_{B_{2R}^+}  &|D^i \mathbb{A}(x, \bar h, D \bar h)|^2 |D^{k-i} v|^2 \,  dx \leq C \| v\|_{H^{k-1}(B_{2R}^+)}^2 \\
&\quad + C\sum_{i=1}^{k-1} \sum_{\substack{1\leq j_1\leq \dots\leq j_m\leq i \\ j_1+\dots+j_m\leq i \\ m\geq 1}} \|D^{j_1}w \|^2_{\frac{2(k-1)}{j_1}}\dots \|D^{j_m}w \|^2_{\frac{2(k-1)}{j_m}} \|D^{k-i}v \|^2_{\frac{2(k-1)}{k-i-1}},
\end{split}
\]
where all the norms in the last line are evaluated in $B_{2R}^+$. Note that if $i = k-1$ then in the last term it is understood that $ \|D^{k-i}v \|_{\frac{k-1}{k-1-i}} =  \|D v \|_{L^\infty}$. Note  that  by standard Schauder estimates the assumption $\|h\|_{C^{1,\alpha}(\Sigma)}\leq K$  implies that $\|D v \|_{L^\infty(B^+_{2R})} \leq C$.  We use Lemma \ref{aubinlemma} to estimate 
\[
\|D^{j_l}w \|_{\frac{2(k-1)}{j_l}} \leq C \|w \|_{H^{k-1}}^{\theta(j_l)}\|w \|_{L^{\infty}}^{1- \theta(j_l)}    \leq C \|w \|_{H^{k-1}}^{\theta(j_l)}
\]
for $\theta(j_l): = \frac{j_l}{k-1}$. By the same lemma we also have 
\[
\|D^{k-i} v \|_{\frac{2(k-1)}{k-i-1}} \leq C \|v \|_{H^{k}}^{\theta}\| D v \|_{L^{\infty}}^{1- \theta} \leq C \|v \|_{H^{k}}^{\theta} 
\]
for $\theta = \frac{k-i-1}{k-1}$. Since $\theta(j_1) + \dots + \theta(j_m) \leq \frac{i}{k-1}$, from \eqref{quasi estimate} and from  the previous estimate we have by Young's inequality
\[
\begin{split}
\int_{B_{R}^+} | D (D^\beta v)|^2  \, dx  &\leq C   \| v\|_{H^{k-1}(B_{2R}^+)}^2 + C  \sum_{i=1}^{k-1} (\|w \|_{H^{k-1}(B_{2R}^+)}^{\frac{2i}{k-1}}+1)  \|v \|_{H^{k}(B_{2R}^+)}^{\frac{2(k-i-1)}{k-1}} \\
&\leq  \e \|D^k v\|_{L^2(B_{2R}^+)}^2 + C\| v\|_{H^{k-1}(B_{2R}^+)}^2 + C (1+\|h \|_{H^{k}(\Sigma)}^2).
\end{split}
\]
In order to control the remaining derivatives we use the equation \eqref{quasi eq} in the strong form 
\[
\Div(\mathbb{A}(x,\bar h, D\bar h) D v) = 0\,.
\] 
Indeed, observe that we have estimated all the derivatives of the type $D^\beta(Dv)$, where $\beta=(\beta_1, \beta_2, 0)$, 
with $\beta_1+\beta_2=k-1$. Using these estimates  and differentiating the equation $k-2$ times with respect to the horizontal directions and once in the vertical direction,  we may estimate 
$D^{\beta}(D_{x_3 x_3}v)$ for all $\beta=(\beta_1, \beta_2, 0)$, with $\beta_1+\beta_2=k-2$, by using an interpolation argument as before to control the lower order derivatives. Then we proceed by induction   by differentiating the equation $k-3$ times with respect to the horizontal directions and twice in the vertical direction, and so on, until we differentiate the equation $k-1$ times only in the vertical direction. As a result we obtain
\[
\int_{B_{R}^+} | D^k  v|^2  \, dx  \leq  \e \|D^k v\|_{L^2(B_{2R}^+)}^2 + C\| v\|_{H^{k-1}(B_{2R}^+)}^2 + C (1+\|h \|_{H^{k}(\Sigma)}^2).
\]
The previous estimate holds at every point on $\pa F_h$.  Thus we may cover    $\mathcal{N}_{\sigma_1}^+(F_h)$, with $\sigma_1<\frac{\eta_0}2$,  by a finite union of balls and use the previous estimate  in every ball of the covering. Precisely, we go back to the original map, set $u=u_{F_h}\circ \Phi_h$ for simplicity,  use Lemma~\ref{lm:taldeitali} and  conclude that there are $0 <\sigma_1 <\sigma_2 $  such that  
\[
\begin{split}
\int_{\mathcal{N}_{\sigma_1}^+} | D^k  u|^2  \, dx &\leq C \e \int_{\mathcal{N}_{\sigma_2}^+} | D^k  u|^2  \, dx +  C\|  u \|_{H^{k-1}(\mathcal{N}_{\sigma_2}^+)}^2 + C(1+ \|h \|_{H^{k}(\Sigma)}^2 )\\
&\leq 2 C \e \int_{\mathcal{N}_{\sigma_2}^+} | D^k  u|^2  \, dx +  C\|  u \|_{L^{2}(\mathcal{N}_{\sigma_2}^+)}^2 + C(1+ \|h \|_{H^{k}(\Sigma)}^2),
\end{split}
\]
where  the last inequality follows from standard interpolation inequality. Choosing $\e$ small we obtain 
\[
\int_{\mathcal{N}_{\sigma_1}^+} | D^k  u|^2  \, dx  \leq 2 \int_{\mathcal{N}_{\sigma_2}^+ \setminus \mathcal{N}_{\sigma_1}^+} | D^k  u|^2  \, dx +  C\|  u \|_{L^{2}(\mathcal{N}_{\sigma_2}^+)}^2 + C (1+\|h \|_{H^{k}(\Sigma)}^2).
\]
By standard interior regularity it holds 
\[
\int_{\mathcal{N}_{\sigma_2}^+ \setminus \mathcal{N}_{\sigma_1}^+} | D^k  u|^2  \, dx \leq C\|  u_{F_h} \|_{L^{2}(\Omega \setminus F_h)}^2 .
\]
 Again by standard interpolation we have that 
\[
\|  u \|_{H^{k}(\mathcal{N}_{\sigma_1}^+)} \leq C(1 + \|  u \|_{L^{2}(\mathcal{N}_{\sigma_1}^+)} + \|h \|_{H^{k}(\Sigma)}). 
\]
By the minimality and by Poincar\'e inequality we have that $\|  u_{F_h} \|_{L^{2}(\Omega \setminus F_h)}$ is bounded by the boundary value $w_0$.  Using again Lemma~\ref{lm:taldeitali} and the $C^{1}$ estimates on $u_{F_h}$,  we have from the above  inequality that  
\[
\| Q(E(u_{F_h}))\circ \Phi_h \|_{H^{k-1}(\mathcal{N}_{\sigma_1}^+)} \leq C(1 +  \|h \|_{H^{k}(\Sigma)} ).
\]
 From this inequality the first claim follows by the trace theorem. 
 
As for the second part of the lemma, let $\Phi_i$ be a diffeomorphism constructed as above from $\Om\setminus G$ to   $\Om\setminus F_{h_i}$. Note that, since $h_1$ and $h_2$ are bounded in $C^{1,\alpha}$,  we may construct the $\Phi_i$'s in such a way that 
$$
\|\Phi_2-\Phi_1\|_{H^1(\Om\setminus G)}\leq C \|h_2-h_1\|_{H^1(\Sigma)}.
$$

 As before we fix $x_0\in \Sigma$ and denote as before by $\Phi$ the diffeomorphism that straightens $\Sigma$. Setting $v_i=u_{F_{h_i}}\circ\Phi_i\circ \Phi^{-1}$ and $\bar h_i=h_i\circ\pi\circ \Phi$, we have that 
$$
\int_{B_{2R}^+} \mathbb{A}(x, \bar h_i, D \bar h_i) Dv_i : D\varphi \, dx = 0 
$$
for all $\varphi \in C^\infty( B_{2R}^+; \R^3)$ vanishing on $\pa B_{2R} \cap \{x_3 >  0\}$, where $\mathbb{A}$ is the same tensor as before. 

Differentiating the equations in the $x_j$-direction, $j=1,2$, and subtracting the two resulting equations  we obtain
\begin{align*}
\int_{B_{2R}^+}\mathbb{A}(x, \bar h_2, D \bar h_2)D(D_j(v_2-v_1)):&D\vphi \, dx = 
-\int_{B_{2R}^+}D_j(\mathbb{A}(x, \bar h_2, D \bar h_2))D(v_2-v_1):D\vphi \, dx\\
& - \int_{B_{2R}^+}[\mathbb{A}(x, \bar h_2, D \bar h_2)-\mathbb{A}(x, \bar h_1, D \bar h_1)]DD_jv_1:D\vphi \, dx\\
&- \int_{B_{2R}^+}D_j[\mathbb{A}(x, \bar h_2, D \bar h_2)-\mathbb{A}(x, \bar h_1, D \bar h_1)]Dv_1 :D\vphi \, dx.
\end{align*}
We choose   $\vphi=D_j(v_2-v_1)\eta^2$ as a testfunction and get by  arguing as before
\begin{align*}
\int_{B_{R}^+}|D(D_j(v_2-v_1))|^2\, dx  &\leq C \int_{B_{2R}^+}(1+|D^2 \bar h_2|^2+ |D^2 \bar h_1|^2)|Dv_2- Dv_1|^2 \, dx\\
&\quad+ C\int_{B_{2R}^+}(|\bar h_2-\bar h_1|^2+ |D \bar h_2- D\bar h_1|^2+|D^2\bar h_2-D^2\bar h_1|^2)|Dv_1|^2\, dx\\
&\quad + C\int_{B_{2R}^+}(|\bar h_2-\bar h_1|^2+ |D \bar h_2- D\bar h_1|^2)|D^2v_1|^2 \, dx \, .
\end{align*}
Recall first that as before  $\|Dv_1\|_{L^{\infty}} \leq C$. Moreover, we assume that $\|h_i\|_{H^3(\Sigma)} \leq K$ and therefore by the proof of the first stament we conclude that 
$\|v_i\|_{H^3(B_{2R}^+)} \leq C$. Using interpolation we get 
\[
\int_{B_{2R}^+}  |D^2 \bar h_1|^2|Dv_2- Dv_1|^2 \, dx \leq \| D^2 \bar h_1\|_{L^4}^2 \, \| Dv_2 - Dv_1\|_{L^4}^2 \leq C \|\bar h_1\|_{H^3}^2\,  \|v_2 - v_1\|_{H^2}^{\frac32} \|v_2 - v_1\|_{L^2}^{\frac12}\,.
\]
Estimating the other terms similarly and using the equation to estimate $D_{33}(v_2-v_1)$,  we get  for any $\e\in (0,1)$
\begin{align*}
\int_{B_{R}^+}|D^2(v_2-v_1)|^2 &\leq C\|v_2-v_1\|_{H^2(B_{2R}^+)}^{\frac32}\|v_2-v_1\|_{L^2(B_{2R}^+)}^{\frac12}+
C\|\bar h_2-\bar h_1\|_{H^2(B_{2R}^+)}^2\\
&\leq \e  \int_{B_{2R}^+}|D^2(v_2-v_1)|^2+ C\int_{B_{2R}^+}|v_2-v_1|^2+ C\|h_2-h_1\|_{H^2(\Sigma)}^2\,.
\end{align*}
Using a simple covering argument as before, going back to the original functions and arguing  as above we get  
 $$
 \|D^2(u_{F_{h_2}}\circ \Phi_{h_2}-u_{F_{h_1}}\circ \Phi_{h_1})\|_{L^2(\mathcal N_{\sigma_1}^+)}\leq  C\|u_{F_{h_2}}\circ \Phi_{h_2}-
 u_{F_{h_1}}\circ \Phi_{h_1}\|_{L^2(\mathcal N_{\sigma_2}^+)} +
 C\|h_2-h_1\|_{H^2(\Sigma)}.
 $$
 Observe now that writing down the equations satisfied by $u_{F_{h_i}}\circ \Phi_{h_i}$ in $\Om\setminus G$ and using 
as an admissible test function $\vphi=u_{F_{h_1}}\circ \Phi_{h_1}-u_{F_{h_2}}\circ \Phi_{h_2}$, one may check that  
$$
\|D(u_{F_{h_1}}\circ \Phi_{h_1}- u_{F_{h_2}}\circ \Phi_{h_2})\|_{L^2(\Om\setminus G)}\leq C\|\Phi_1-\Phi_2\|_{H^1(\Om\setminus G)}\leq C \|h_1-h_2\|_{H^1(\Sigma)}.
$$
The conclusion follows from this estimate and from  the previous one  by the Poincar\'e inequality.
\end{proof}

\begin{remark}
\label{rem after long}
Let $h_{F_i}$ and $u_{F_i}$ for $i= 1,2$ be as in Theorem~\ref{linearestimate}. The inequality at the end of the proof of the  lemma implies that 
\[
\|u_{F_{h_2}} \circ \pi^{-1}_{F_{h_2}} - u_{F_{h_1}} \circ \pi^{-1}_{F_{h_1}}  \|_{H^{1/2}(\Sigma)}\leq  C \|h_2-h_1\|_{H^1(\Sigma)}.
\]   
Moreover, if in addition to  the assumptions of the second part of Theorem~\ref{linearestimate}  we know also that  
$\|h_i\|_{C^{1}(\Sigma)}$ is sufficiently small for $i =1,2$,  then the   proof of the inequality \eqref{seconda}  also gives  the estimate    
\[
\| (Du_{F_{h_2}}) \circ \pi^{-1}_{F_{h_2}} - (Du_{F_{h_1}}) \circ \pi^{-1}_{F_{h_1}} \|_{L^2(\Sigma)}\leq C \|h_2-h_1\|_{H^2(\Sigma)}.
\]
\end{remark}

Let us consider the smooth flow $(F_{t})_{t \in (0,T_0)}$ with initial set $F_0$, which is a solution of  \eqref{forced} with smooth forcing term  $f : \Sigma \times [0,T_0) \to \R$. Here $T_0$ is the existence time provided by Theorem \ref{thm with f}.   For every given time $t \in (0,T_0)$ we consider the elastic equilibrium $u_t$ in $\Omega \setminus F_t$ defined in \eqref{uf} and  we use the regularity estimates from Theorem~\ref{linearestimate} to  establish the following lemma.

\begin{lemma}\label{prop:Q}
Let  $K_0>1$  be such that $|| Q(E(u_G))||_{L^\infty(\Sigma)} < K_0/4$.   There exist
$T>0$ and $\tilde \e>0$ with the following property:    if $\|h_0\|_{H^3(\Sigma)}<K_0$, and  $\|h_0\|_{L^2(\Sigma)}<\tilde \e$, and  $f$ is a  smooth function satisfying \eqref{innominata}  then the solution of \eqref{forced}, with initial datum $h_0$, provided by Theorem~\ref{thm with f} 
 exists for the time interval $(0, T)$ and it holds
\begin{equation}\label{prop:Q1}
 \sup_{0 \leq t \leq T} \|Q(E(u_t))\circ \pi_{F_t}^{-1}\|_{L^\infty(\Sigma)} +  \int_0^{T}  \|Q(E(u_t))\circ \pi_{F_t}^{-1}\|_{H^{3}(\Sigma)}^2\, dt  \leq K_0.
\eeq
Moreover,  for every  $k\in \N$ there exists $C'_k(K_0)>0$ such that 
\beq\label{prop:Q2}
\sum_{i=0}^k  \int_0^{T} \ t^i\|Q(E(u_t))\circ \pi_{F_t}^{-1}\|_{H^{2i+3}(\Sigma)}^2 \, dt \leq \frac12\biggl(C'_k(K_0) +  \sum_{i=0}^k  \int_0^{T} t^i\|f(\cdot, t)\|_{H^{2i+3}(\Sigma)}^2 \, dt\biggr).
\eeq

\end{lemma}

 \begin{proof}
 We begin by proving \eqref{prop:Q1}.  Let us fix $\alpha \in (0,1)$. Given $\de_0>0$ to be chosen later and taking $\tilde \e$ equal to the corresponding $\e_0$,  let $h(\cdot, t)$ be the solution  defined on  $(0,T_0)$, provided by Theorem~\ref{thm with f}.   Note that from \eqref{Pitt1} and \eqref{newonw} we have  
  $\sup_{0\leq t\leq T_0}\|h(\cdot, t)\|_{H^3}\leq C(K_0)$ and $\sup_{0\leq t\leq T_0}\|h(\cdot, t)\|_{L^2}\leq \delta_0$. In turn, by interpolation 
   $\sup_{0\leq t\leq T_0}\|h(\cdot, t)\|_{C^{1,\alpha}}\leq C\delta_0^\theta <1$ for some $\theta \in (0,1)$. Recall also  that by choosing $\tilde \e$ small we can make $\delta_0$ as small as we wish.  By standard elliptic estimates we have that 
  \[
\sup_{0\leq t\leq T_0}  \| u_t\circ \pi_{F_t}^{-1} - u_G \|_{C^{1,\alpha}(\Sigma)}\leq \omega(\delta_0),
  \]
and $\omega(\delta_0) \to  0$ as $\delta_0 \to 0$.  In turn,  we conclude that for every $t \in (0,T_0)$ it holds
 \[
\begin{split}
 \|Q(E(u_t))\circ \pi_{F_t}^{-1}\|_{L^\infty} \leq    \|Q(E(u_t)) - Q(E(u_G)) \circ \pi_{F_t}^{-1}\|_{L^\infty}  + \|Q(E(u_G)) \circ \pi_{F_t}^{-1}\|_{L^\infty} \leq \frac{K_0}{3}
  \end{split}
\]
provided  $\tilde \e$ (and thus $\de_0$) is small enough. 

Concerning the second term on the left-hand side of \eqref{prop:Q1}, we have by  a well-known interpolation result  and by \eqref{prima} for $k= 5$ from Theorem~\ref{linearestimate}
\[
  \begin{split}
  \int_0^T\|Q(E(u_t))\circ &\pi_{F_t}^{-1}\|^2_{H^{3}(\Sigma)}\, dt \\
&\leq C \int_0^T  \|Q(E(u_t))\circ \pi_{F_t}^{-1}\|^{2\theta}_{H^{\frac72}(\Sigma)}\|Q(E(u_t))\circ \pi_{F_t}^{-1}\|^{2(1-\theta)}_{L^\infty(\Sigma)}\, dt\\
  &\leq C \int_0^T(1+\|h(\cdot, t)\|_{H^{5}(\Sigma)}^{2\theta}) K_0^{2(1-\theta)}\, dt \\
  &\leq \eta \int_0^T \|h(\cdot, t)\|_{H^{5}(\Sigma)}^{2}\, dt+C_{\eta} K_0^2\,  T \\
&\leq \eta C \left( K_0^2 +  \int_0^T  \big( 1+ \| f(\cdot, t)\|_{L^\infty(\Sigma)}^{q_0}  +  \|f(\cdot, t)\|_{H^{3}(\Sigma)}^2\big)\, dt \right) +C_{\eta} K_0^2\,  T\\
&\leq  \eta C \left( K_0^2 + T+  T  K_0^{q_0}  +K_0 \right) +C_{\eta} K_0^2\,  T ,
  \end{split}
\]
where the second last inequality follows from \eqref{newonw}.  The inequality  \eqref{prop:Q1}  follows by  choosing $\eta$ and $T\leq T_0$ sufficiently small.

The inequality \eqref{prop:Q2} follows by a similar argument.  For all $i=1,\dots, k$ we have again by interpolation and  by \eqref{prima}  that 
\[
  \begin{split}
  \int_0^Tt^i\|Q(E(u_t))\circ & \pi_{F_t}^{-1}\|^2_{H^{2i+3}(\Sigma)}\, dt \\
  &\leq C \int_0^T t^i \|Q(E(u_t))\circ \pi_{F_t}^{-1}\|^{2\theta}_{H^{2i+\frac72}(\Sigma)}\|Q(E(u_t))\circ \pi_{F_t}^{-1}\|^{2(1-\theta)}_{L^\infty(\Sigma)}\, dt\\
  &\leq C_k\int_0^Tt^i (1+\|h(\cdot, t)\|_{H^{2i+5}}^{2\theta}) K_0^{2(1-\theta)}\, dt \\
  &\leq \eta\int_0^Tt^i \|h(\cdot, t)\|_{H^{2i+5}}^{2}\, dt+C_{k,\eta} K_0^2 \, T. 
  \end{split}
\]
  The conclusion  then  follows by estimating the last integral by means of \eqref{newonw} and choosing $\eta$  sufficiently small and $C_k'(K_0)$ sufficiently large. 
   \end{proof}

\begin{theorem}
\label{thm surf}
Let $K_0>1$  be such that $|| Q(E(u_G))||_{L^\infty(\Sigma)} < K_0/4$ and fix $\de_0>0$. There exist
$T\in (0,1)$ and $\e_1\in (0,1)$ with the following property:  if   $F_0 \in  \mathfrak{h}_{K_0}^3(\Sigma)$, {defined in \eqref{mathfrakh}},  with   $\|h_0\|_{L^2(\Sigma)}<\e_1$  then there exists a  unique solution $h$ to \eqref{FLOW} in  $H^1(0, T; H^1(\Sigma))\cap L^{\infty}(0, T; H^3(\Sigma))$.  Moreover, the solution belongs to  $H^1_{loc}(0,T; H^k(\Sigma))$ for every $k\geq 1$ and  it holds  
\beq\label{solfa}
\sup_{0\leq t\leq T}\|h(\cdot, t)\|_{L^2(\Sigma)}<\de_0
\eeq
 and 
\beq\label{solfa2}
 \sup_{0\leq t \leq T} t^k\|h(\cdot, t)\|_{H^{2k+3}(\Sigma)}^2 +  \int_0^{T}t^k\|h(\cdot, t)\|_{H^{2k+5}(\Sigma)}^2\,dt  \leq C(k, K_0). 
\eeq
\end{theorem}
\begin{proof}
\
We divide the proof into three steps.

\medskip

\noindent{\bf Step 1.} 
  Let $K_0$, $T$ be as in  Lemma~\ref{prop:Q}. Let $\mathcal{S}$  be the set of functions  in $C^\infty(0,T; C^\infty(\Sigma))$ that satisfy
$$
 \sup_{0 \leq t \leq T} \| f(\cdot, t)\|_{L^\infty(\Sigma)} +  \int_0^{T}  \|f(\cdot, t)\|_{H^{3}(\Sigma)}^2\, dt  \leq K_0
$$
and
$$
 \sum_{i=0}^k  \int_0^{T} \big(t^i\|f(\cdot, t)\|_{H^{2i+3}(\Sigma)}^2\big)\, dt \leq C_k'(K_0)
$$
for every $k \in \mathbb{N}$, where $C_k'(K_0)$ are the constants from \eqref{prop:Q2}. 
We define a map $\mathscr{L} : \mathcal{S} \to \mathcal{S}$ as  $\mathscr{L}(f)(\cdot, t):=-Q(E(u_t))\circ\pi^{-1}_{F_t}$ for all $t\in (0, T)$, where $F_t$ is the solution of \eqref{forced} with initial datum $h_0$ and  forcing term $f$, and where $u_t$ stands for $u_{F_t}$, that is for the elastic equilibrium in $\Om\setminus F_t$.
Lemma  \ref{prop:Q} implies that the map $\mathscr{L} : \mathcal{S} \to \mathcal{S}$ is well defined,   provided that $\e_1\leq\tilde\e$.
Note also that $ \mathcal{S}$ is clearly  nonempty as the zero function  belongs to $\mathcal{S}$.

We will show that $\mathscr{L} : \mathcal{S} \to \mathcal{S}$ is a contraction with respect to  a suitable norm.

\medskip

\noindent{\bf Step 2.}  Fix $\mu\in (0,1)$. Let $f_1$ and $f_2$ be two smooth functions in $\mathcal{S}$ and let $h_1$ and $h_2$ be the corresponding solutions of \eqref{forced2}  with intial datum  $h_0$. The goal in this step is to show that  it holds
 \beq\label{boh2}
\int_0^T \int_{\Sigma}(h_2(\cdot, t)-h_1(\cdot, t))^2\, d\Ha^2 dt \leq \mu \int_0^T\int_{\Sigma}(f_2(\cdot, t)-f_1(\cdot, t))^2d\Ha^2dt,
 \eeq
by possibly decreasing the time $T$.  We recall that by Theorem  \ref{thm with f} we have that 
\[
\sup_{0 \leq t \leq T} \| h(\cdot, t)\|_{L^2(\Sigma)} \leq \delta_0 \qquad \text{and} \qquad \sup_{0 \leq t \leq T} \| h(\cdot, t)\|_{H^3(\Sigma)} \leq C(K_0)\,,
\]
provided that $\e_1<\e_0$.
By interpolation these imply that $\sup_{0 \leq t \leq T} \| h(\cdot, t)\|_{C^{1,\alpha}(\Sigma)} \leq C \delta_0^\theta < 1$ for some $\theta \in (0,1)$. In turn, by standard 
Schauder estimates the corresponding elastic equilibria in $F_{h(\cdot, t)}$ are uniformly  bounded in $C^{1,\alpha}$ up to the boundary, i.e.,  $\sup_{0\leq t\leq T}  \| u_t\circ \pi_{F_t}^{-1}\|_{C^{1,\alpha}(\Sigma)}\leq C$.  We will use these facts repeatedly in the proof. 

We denote by $F_{t,i}$ the set  related to $h_i(\cdot, t)$ with   $\pa F_{t,i}= \{ x + h_i(x,t) \nu(x) : x \in \Sigma\}$. 
We multiply \eqref{forced} for $i=1,2$ by $\big((h_2-h_1)\circ\pi\big)  \big(({J}_i\circ\pi)\nu_{F_{t,i}}\cdot (\nu\circ\pi)\big)^{-1}$, where ${J}_i$ stands for the tangential Jacobian on $\Sigma$ of the map $x\mapsto x+h_i(x)\nu(x)$ and $\pi$ for the  projection on $\Sigma$.   We then get
\begin{multline*}
\int_{\pa F_{t,i}}(\pa_t h_i(\cdot, t)\circ\pi )   \frac{(h_1-h_2)\circ\pi}{{J}_i\circ\pi}\, d\Ha^2  \\
=
\int_{\pa F_{t,i}} \Delta_{\pa F_{t,i}}[H_{\pa F_{t,i}}+f_i(\cdot, t)\circ\pi)]\big((h_2-h_1)\circ\pi\big) \big(({J}_i\circ\pi) \nu_{F_{t,i}}\cdot (\nu\circ\pi)\big)^{-1}\, d\Ha^2.
\end{multline*}
Recall that, denoting by $\pa_{\tau_1} h_i$ and $\pa_{\tau_2} h_i$ the tangential derivatives of $h_i$ in the directions of the principal curvatures, we have   
\[
{J}_i=\sqrt{(1+h_i k_1)^2(1+h_ik_2)^2+(1+h_i k_1)^2(\pa_{\tau_1} h_i)^2+(1+h_i k_2)^2(\pa_{\tau_2} h_i)^2},
\]
where $k_1$, $k_2$ are the principal curvatures of $\Sigma$. Therefore we have by the formula for the outer normal \eqref{normal} that 
\[
\big(({J}_i\circ\pi ) \nu_{F_{t,i}} \cdot (\nu\circ\pi)\big)^{-1}=\frac1{(1+h_i k_1)(1+h_ik_2)}\circ\pi=:R(\cdot, h_i)\circ\pi.
\] 
By integrating by   parts we get
\begin{multline*}
\int_{\pa F_{t,i}}(\pa_t h_i(\cdot, t)\circ\pi )   \frac{(h_1-h_2)\circ\pi}{{J}_i\circ\pi}\, d\Ha^2  \\
=
\int_{\pa F_{t,i}} (H_{\pa F_{t,i}}+f_i(\cdot, t)\circ\pi)\Delta_{\pa F_{t,i}}[(h_1-h_2)\circ\pi \, R(\cdot, h_i)\circ\pi]\, d\Ha^2
\end{multline*}
Rewriting the integrals above on $\Sigma$ and subtracting, we have
\begin{multline*}
\frac12\frac{\pa}{\pa t}\int_{\Sigma}(h_2-h_1)^2\, d\Ha^2\\=
 \int_{\Sigma} \bigl(J_2 H_{\pa F_{t,2}}\circ \pi^{-1}_{F_{t,2}}-J_1 H_{\pa F_{t,1}}\circ \pi^{-1}_{F_{t,1}}+J_2f_2-J_1f_1\bigr)\Delta_{\pa F_{t,2}}[(h_2-h_1)\circ\pi R(\cdot, h_2)\circ\pi]\circ\pi^{-1}_{F_{t,2}}\, d\Ha^2\\
 +  \int_{\Sigma} J_1(H_{\pa F_{t,1}}\circ \pi^{-1}_{F_{t,1}}+f_1)\bigl(\Delta_{\pa F_{t,2}}[(h_2-h_1)\circ\pi R(\cdot, h_2)\circ\pi]\circ\pi^{-1}_{F_{t,2}}\\
 -
 \Delta_{\pa F_{t,1}}[(h_2-h_1)\circ\pi R(\cdot, h_1)\circ\pi]\circ\pi^{-1}_{F_{t,1}}\bigl)\, d\Ha^2.
\end{multline*}
We recall \eqref{mean curvature} and \eqref{laplacianogm}, where the coefficients $A$, $A_1$ and $A_2$ vanish  as $(h, \nabla h) = 0$.  
We recall also that $\|h_i(\cdot, t)\|_{C^{1,\alpha}}$ is  small uniformly in time and that $f_i$ are uniformly bounded with respect to time.
After straighforward calculations we have 
\begin{align*}
&\frac12\frac{\pa}{\pa t}\int_{\Sigma}(h_2-h_1)^2\, d\Ha+\frac12 \int_\Sigma|\Delta(h_2-h_1)|^2\, d\Ha^2\leq \e\int_{\Sigma}|\nabla^2(h_2-h_1)|^2\, d\Ha^2 \\
&+ C\int_{\Sigma}(1+|\nabla ^2h_1|+|\nabla^2h_2|)(|h_2-h_1|+|\nabla (h_2-h_1)|)\cdot\\
&\qquad\qquad\qquad\qquad\qquad\qquad\qquad\qquad\cdot(|h_2-h_1|+|\nabla (h_2-h_1)|+|\nabla^2(h_2-h_1)|)\, d\Ha^2\\
&+C\int_{\Sigma}|f_2-f_1|\left((1+|\nabla ^2h_1|+|\nabla^2h_2|)( |h_2 - h_1|+|\nabla (h_2-h_1)|)+|\nabla^2(h_2-h_1)|\right) d\Ha^2\\
&+C \int_\Sigma (1+|\nabla ^2h_1|^2+|\nabla^2h_2|^2)( |h_2 - h_1|^2+|\nabla (h_2-h_1)|^2)\, d\Ha^2 =:RHS.
\end{align*}
Using Young's Inequality we obtain 
\begin{align*}
&RHS\leq \e\int_{\Sigma}|\nabla^2(h_2-h_1)|^2\, d\Ha^2\\
&+ C\int_{\Sigma}(1+|\nabla^2 h_1|^2+|\nabla^2 h_2|^2)(|h_2-h_1|^2+|\nabla(h_2-h_1)|^2)\, d\Ha^2+C\int_{\Sigma}|f_2-f_1|^2\, d\Ha^2.
\end{align*}

Observe now that by interpolation, by controlling the second derivatives of $h_i$ with the $H^3$-norms, and using the fact that $\|h(\cdot, t)\|_{H^3}$ is bounded uniformly with respect to time we have
\begin{align*}
&\int_{\Sigma}(1+|\nabla^2 h_1|^2+|\nabla^2 h_2|^2)(|h_2-h_1|^2+|\nabla(h_2-h_1)|^2)\, d\Ha^2\\
&\leq C (1+ \|\nabla^2 h_1\|_{L^4}^2+ \|\nabla^2 h_2\|_{L^4}^2)\|h_2-h_1\|_{W^{1,4}}^2 \leq C\|h_2-h_1\|_{H^2}^{\frac32}\|h_2-h_1\|_{L^2}^{\frac12}.
\end{align*}
From the previous inequalities  we get
\begin{align*}
\frac12\frac{\pa}{\pa t}\int_{\Sigma}(h_2-h_1)^2\, d\Ha^2 & \leq -\frac12 \int_\Sigma|\Delta(h_2-h_1)|^2\, d\Ha^2+
\e\int_\Sigma|\nabla^2(h_2-h_1)|^2\, d\Ha^2\\
&\qquad+C_\e\int_{\Sigma}(|\nabla(h_2-h_1)|^2+(h_2-h_1)^2+(f_2-f_1)^2) \, d\Ha^2 .
\end{align*}
Using now  Remark~\ref{rem aub 2} we in turn obtain 
\begin{align*}
&\frac12\frac{\pa}{\pa t}\int_{\Sigma}(h_2-h_1)^2\, d\Ha^2 +\frac14 \int_\Sigma|\nabla^2(h_2-h_1)|^2\, d\Ha^2\leq C\int_{\Sigma}(|h_2-h_1|^2+(f_2-f_1)^2)\, d\Ha^2.
\end{align*}
Integrating this with respect to time over  $(0,t)$, with $t\in (0, T)$,  we have
\begin{align}\label{bohh}
&\int_{\Sigma}(h_2(\cdot, t)-h_1(\cdot, t))^2\, d\Ha^2+\frac12 \int_0^t\int_\Sigma|\nabla^2(h_2(\cdot,s)-h_1(\cdot, s))|^2\, d\Ha^2ds\\
&\leq C\int_0^t\int_{\Sigma}(|h_2(\cdot,s)-h_1(\cdot, s)|^2+(f_2(\cdot, s)-f_1(\cdot, s))^2) \, d\Ha^2ds.\nonumber
\end{align}
Integrating the above inequality with respect to time over $(0,T)$  we obtain \eqref{boh2} when $T$ is sufficiently small.

\medskip

\noindent{\bf Step 3.} 
Here we finally prove that the map $\mathscr{L} : \mathcal{S} \to \mathcal{S}$ is a contraction with respect to the $L^2(0,T; L^2(\Sigma))$-norm. To be more precise, let 
$f_1$ and $f_2$ be two functions in $\mathcal{S}$ and $h_1$ and $h_2$ the corresponding solutions of \eqref{forced}. For simplicity we denote  the elastic equilibrium 
for $F_{h_i}$ as $u_i(\cdot, t):=u_{F_{t,i}}$, for $i =1,2$.  Then $\mathscr{L}(f_i) = -Q(E(u_i))\circ\pi^{-1}_{F_{t,i}}$ and our goal is to show 
\beq
\label{thm contraction}
\int_0^{T} \|Q(E(u_2(\cdot, t)))\circ{\pi^{-1}_{F_{t,2}}}-Q(E(u_1(\cdot, t)))\circ{\pi^{-1}_{F_{t,1}}}\|^2_{L^2(\Sigma)}dt\\
\leq \frac12 \int_0^{T} \|f_2(\cdot,t) - f_1(\cdot, t)\|_{L^2(\Sigma)}^2 \, dt. 
\eeq

Let us fix $t \in (0,T). $ We begin by proving  
\beq
\begin{split}\label{napoli}
\|Q(E(u_2(\cdot, t)))\circ{\pi^{-1}_{F_{t,2}}}&-Q(E(u_1(\cdot, t)))\circ{\pi^{-1}_{F_{t,1}}}\|_{L^2(\Sigma)}\\
&\leq C\|\nabla(u_2(\cdot, t)\circ \pi^{-1}_{F_{t,2}})-\nabla(u_1(\cdot, t)\circ \pi^{-1}_{F_{t,1}})\|_{L^2(\Sigma)}\\
&\qquad+ \e\|\nabla^2(h_2(\cdot, t)-h_1(\cdot, t))\|_{L^2(\Sigma)}^2  + C\|h_2(\cdot, t)-h_1(\cdot, t)\|_{H^1(\Sigma)}.
\end{split}
\eeq
To shorten the notation we  denote $U_i := Du_i \circ{\pi^{-1}_{F_{t,i}}}$, $\nu_i = \nu_{F_{t,i}}\circ{\pi^{-1}_{F_{t,i}}}$  and $h_i = h_i(\cdot, t)$ for $i = 1,2$. Recall that $Q(E(u_i(\cdot, t)))\circ{\pi^{-1}_{F_{t,i}}} =\frac12 \C U_i : U_i $. We may  thus write 
\[
\|Q(E(u_2(\cdot, t)))\circ{\pi^{-1}_{F_{t,2}}}-Q(E(u_1(\cdot, t)))\circ{\pi^{-1}_{F_{t,1}}}\|_{L^2(\Sigma)}  =\frac12 \|\C(U_2 + U_1) :  (U_2 -  U_1)\|_{L^2(\Sigma)}.
\]
We estimate this  simply as
\beq \label{napoli 2}
\begin{split}
\|\C (U_2 + U_1) : (U_2 -  U_1)\|_{L^2(\Sigma)}  \leq &\|\C (U_2 + U_1) :  \big((U_2 -  U_1) (I - \nu \otimes \nu)\big) \|_{L^2(\Sigma)}  \\
&+  \| \C (U_2 + U_1) : \big( (U_2 -  U_1)  \, (\nu \otimes \nu)  \big) \|_{L^2(\Sigma)} 
\end{split}
\eeq
Note that by the second condition in \eqref{uf} it holds $\C  U_i[\nu_i]= \C E(u_i)\circ\pi^{-1}_{F_{t,i}}[\nu_i]=0$ on $\Sigma$. We use this equality to estimate the last term in \eqref{napoli 2} by
\[
\begin{split}
&\| \C (U_2 + U_1) :   \big((U_2 -  U_1)  \, (\nu \otimes \nu)\big)  \|_{L^2(\Sigma)}    \\
&\leq \| \C (U_2 + U_1) : \big((U_2 -  U_1)  \, ( \nu \otimes (\nu - \nu_2) )\big)  \|_{L^2(\Sigma)} + \| \C  U_1 :  \big( (U_2 - U_1)  \, (\nu \otimes  \nu_2  ) \big)  \|_{L^2(\Sigma)} \\
&= \| \C (U_2 + U_1) : \!\big((U_2 -  U_1)  \, ( \nu \otimes (\nu - \nu_2) )\big)  \|_{L^2(\Sigma)} + \| \C  U_1 : \!  \big((U_2 - U_1) \, (\nu \otimes  (\nu_2 - \nu_1)  ) \big)  \|_{L^2(\Sigma)}. 
\end{split}
\]
Using the expression \eqref{normal} for the normal $\nu_2$ and the uniform $C^{1,\alpha}$-bound for $h_i$ we deduce that $\|\nu  - \nu_2\|_{L^\infty(\Sigma)} \leq C  \delta_0^\theta$ and 
$\| \nu_2 - \nu_1 \|_{L^2(\Sigma)} \leq C\| h_2 - h_1 \|_{H^1(\Sigma)}$. Moreover, by the $C^{1,\alpha}$-bound  for $u_i$  we have that $\| U_i \|_{L^\infty} \leq C$ and by the second inequality in Remark \ref{rem after long}  it holds $\| U_2 - U_1 \|_{L^2(\Sigma)} \leq C \| h_2 - h_1 \|_{H^2(\Sigma)}$. Therefore we may estimate the above inequality as 
\[
\| \C (U_2 + U_1) : \big(  (U_2 -  U_1)(\nu \otimes \nu)  \big)\|_{L^2(\Sigma)} \leq  \e \| h_2 - h_1 \|_{H^2(\Sigma)}  +  C \| h_2 - h_1 \|_{H^1(\Sigma)}. 
\]
Thus we deduce by \eqref{napoli 2} that 
\begin{multline*}
\| \C (U_2 + U_1) : (U_2 -  U_1)\|_{L^2(\Sigma)}  \leq    \| \C (U_2 + U_1) : \big( ( U_2 -  U_1) (I - \nu \otimes \nu)\big) \|_{L^2(\Sigma)}  \\
 +\e \| h_2 - h_1 \|_{H^2(\Sigma)}   + C\| h_2 - h_1 \|_{H^1(\Sigma)}. 
\end{multline*}
The inequality \eqref{napoli} then follows from \eqref{cov tang} as 
\[
\begin{split} 
& \|\C (U_2 + U_1) :  \big(( U_2 -  U_1) (I - \nu \otimes \nu)\big) \|_{L^2(\Sigma)}  \\
&=\| \C (U_2 + U_1) :  \big(( D u_2(\cdot, t) \circ{\pi^{-1}_{F_{t,2}}}- D u_1(\cdot, t) \circ{\pi^{-1}_{F_{t,1}}})  (I - \nu \otimes \nu )\big) \|_{L^2(\Sigma)}\\
&\leq C \| \big(D u_2 (\cdot, t)  \circ \pi^{-1}_{F_{t,2}} - D u_1(\cdot, t)  \circ \pi^{-1}_{F_{t,1}}\big)_\tau\|_{L^2(\Sigma)}\\
&\leq C\|\big[(D u_2 (\cdot, t)  \circ \pi^{-1}_{F_{t,2}})D\pi_{F_{t,2}}^{-1} - (D u_1(\cdot, t)  \circ \pi^{-1}_{F_{t,1}})D\pi_{F_{t,1}}^{-1}\big]_\tau\|_{L^2(\Sigma)}\\
&\quad+C\|\big[(D u_2 (\cdot, t)  \circ \pi^{-1}_{F_{t,2}})(D\pi_{F_{t,2}}^{-1} - D\pi_{F_{t,1}}^{-1})\big]_\tau\|_{L^2(\Sigma)}\\
&\quad+C\|\big[(D u_2 (\cdot, t)  \circ \pi^{-1}_{F_{t,2}}-D u_1 (\cdot, t)  \circ \pi^{-1}_{F_{t,1}})(I - D\pi_{F_{t,1}}^{-1})\big]_\tau\|_{L^2(\Sigma)}\\
&\leq  C\| \nabla ( u_2 (\cdot, t)  \circ \pi^{-1}_{F_{t,2}})- \nabla (u_1(\cdot, t)  \circ \pi^{-1}_{F_{t,1}})\|_{L^2(\Sigma)} + C \| h_2 - h_1 \|_{H^1(\Sigma)}+\e \| h_2 - h_1 \|_{H^2(\Sigma)},
\end{split}
\]
where in the last inequality we used the second estimate in Remark~\ref{rem after long} and the fact that the $C^1$-norm of $h_1$ is small.

We proceed  by using  \eqref{napoli} and interpolation to deduce 
\begin{align*}
&\|Q(E(u_2(\cdot, t)))\circ{\pi^{-1}_{F_{t,2}}}-Q(E(u_1(\cdot, t)))\circ{\pi^{-1}_{F_{t,1}}}\|_{L^2(\Sigma)} \\
&\leq C\|\nabla(u_2(\cdot, t)\!\circ\! \pi^{-1}_{F_{t,2}})-\nabla(u_1(\cdot, t)\!\circ \!\pi^{-1}_{F_{t,1}})\|^{\frac12}_{H^{\frac12}(\Sigma)}
\|\nabla(u_2(\cdot, t)\!\circ \!\pi^{-1}_{F_{t,2}})-\nabla(u_1(\cdot, t)\!\circ\! \pi^{-1}_{F_{t,1}})\|^{\frac12}_{H^{-\frac12}(\Sigma)}\\
&\quad+\e \| h_2 - h_1 \|_{H^2(\Sigma)}  +C\|h_2(\cdot, t)-h_1(\cdot, t)\|_{H^1(\Sigma)}.
\end{align*}
By  the estimate  \eqref{seconda} in Theorem~\ref{linearestimate} we have 
\[
\begin{split} 
\|\nabla(u_2(\cdot, t)&\circ \pi^{-1}_{F_{t,2}})-\nabla(u_1(\cdot, t)\circ \pi^{-1}_{F_{t,1}})\|_{H^{\frac12}(\Sigma)}\\
&\leq \| u_2(\cdot, t)\circ \pi^{-1}_{F_{t,2}}- u_1(\cdot, t)\circ \pi^{-1}_{F_{t,1}}\|_{H^{3/2}(\Sigma)} \leq C\|h_2(\cdot, t)-h_1(\cdot, t)\|_{H^2}.
\end{split}
\]
Moreover by using the well-known inequality  $\|\nabla g\|_{H^{-\frac12}(\Sigma)}\leq C \|g\|_{H^{\frac12}(\Sigma)}$ and  Remark~\ref{rem after long}
we have
\begin{multline*}
\|\nabla(u_2(\cdot, t)\circ \pi^{-1}_{F_{t,2}})-\nabla(u_1(\cdot, t)\circ \pi^{-1}_{F_{t,1}})\|_{H^{-\frac12}(\Sigma)} \\ \leq C
\|u_2(\cdot, t)\circ \pi^{-1}_{F_{t,2}}-u_1(\cdot, t)\circ \pi^{-1}_{F_{t,1}}\|_{H^{\frac12}(\Sigma)}
\leq C\|h_2(\cdot, t)-h_1(\cdot, t)\|_{H^1(\Sigma)}.
\end{multline*}
Collecting the previous three inequalities, using standard interpolation
\[
\|h_2(\cdot, t)-h_1(\cdot, t)\|_{H^1(\Sigma)} \leq C \|h_2(\cdot, t)-h_1(\cdot, t)\|_{H^2(\Sigma)}^{1/2}\|h_2(\cdot, t)-h_1(\cdot, t)\|_{L^2(\Sigma)}^{1/2}, 
\]
and by  Young's inequality we obtain
\begin{multline*}
\|Q(E(u_2(\cdot, t)))\circ{\pi^{-1}_{F_{t,2}}}-Q(E(u_1(\cdot, t)))\circ{\pi^{-1}_{F_{t,1}}}\|^2_{L^2} \\
\leq 2\e\|\nabla^2(h_2(\cdot, t)-h_1(\cdot, t))\|_{L^2}^2 + C_\e\|h_2(\cdot, t)-h_1(\cdot, t)\|_{L^2}^2.
\end{multline*}
Integrating  the previous inequality over $(0, T)$ and using  \eqref{boh2} and \eqref{bohh}, we obtain
\[
\begin{split}
\int_0^{T} \|Q(E(u_2(\cdot, t)))&\circ{\pi^{-1}_{F_{t,2}}}-Q(E(u_1(\cdot, t)))\circ{\pi^{-1}_{F_{t,1}}}\|^2_{L^2}dt\\
 &\leq \big((C_\e +\e C)\mu + \e C\big) \int_0^{T} \|f_2(\cdot,s) - f_1(\cdot, s)\|_{L^2}^2 \, d\Ha^1ds\\
&\leq \frac12 \int_0^{T} \|f_2(\cdot,s) - f_1(\cdot, s)\|_{L^2}^2 \, d\Ha^1ds, 
\end{split}
\]
provided that $\e$ and then $\mu$ are chosen sufficiently small. This proves \eqref{thm contraction} and we conclude that $\mathscr{L} : \mathcal{S} \to \mathcal{S}$ is a contraction with respect to the $L^2(0,T; L^2(\Sigma))$-norm. 

\medskip

\noindent {\bf Step 4.} (Conclusion) We may proceed with a standard argument, by recursively setting $f_1=0$, $f_n:=\mathcal{L}(f_{n-1})$ and  for every $n\geq 1$ letting 
$h_n$ be the solution to \eqref{forced} with $f$ replaced by $f_n$. From Step 2 and Step 3 we have that there exist $f$ and $h$ such that
$f_n\to f$ and $h_n\to h$ in $L^2(0,T; L^2(\Sigma))$. Moreover,  using \eqref{prop:Q2} and \eqref{newonw}, we conclude easily that 
 for every $n\geq 1$ the functions $h_n$ satisfy \eqref{solfa} and \eqref{solfa2} for every $k\in \N$, with constants depending only on $k$ and $K_0$.   
 Thus,  we have that $h_n\wto h$ weakly in $H^1(0, T; H^1(\Sigma))\cap L^{\infty}(0, T; H^3(\Sigma))$. Moreover using the equation satisfied by $h_n$ and \eqref{newonw} we also have that $\pa_t h_n$ is bounded in 
 $L^2_{loc}(0, T; H^k(\Sigma))$ for every $k\in \N$. Therefore we have that  $h_n\wto h$ weakly in $H^1_{loc}(0, T; H^k(\Sigma))$  and thus strongly in $L^2_{loc}(0,T; H^k(\Sigma))$ and that $h$ satisfies \eqref{solfa} and \eqref{solfa2}. Using these convergences one can easily pass to the limit in the equations satisfied by the $h_n$'s to conclude that $h$ is a solution of \eqref{FLOW}. The uniqueness follows from the same argument used in Step 2 and Step 3.

\end{proof}

\section{Asymptotic stability}\label{sec:stability}

In this section we study the flow when the initial set is close to a smooth strictly stable stationary set $G$, which will be our reference set, i.e., we set $\Sigma = \pa G$.
Throughout this section we denote
$$
R_t: =  H_{F_t}  - Q(E(u_{F_t})). 
$$
Moreover, in what follows we shall drop the subscript $\pa F_t$ (and similar) in all the covariant differential operators, when no danger of confusion arises. 
Here is the main result. 

\begin{theorem}
\label{thmstability}
Let $G \subset\subset\Om$ be a regular strictly stable stationary set   in the sense of Definition~\ref{def:stable}. There exists $\delta>0$ such that if $F_0 \in \mathfrak{h}^3_{\delta}(\Sigma)$,  then the unique solution $(F_t)_{t>0}$ of the flow \eqref{FLOW} with intial datum $F_0$ is defined for all times $t>0$. 

Moreover $F_t \to F_{\infty}$  exponentially fast, where $F_\infty$ is the 
unique stationary set near  $G$  such that $|F_{\infty, i}|=|F_{0,i}|$ for $i=1,\dots, m$.  In particular, if $|F_{0,i}|=|G_i|$ for $i=1, \dots, m$, then $F_t \to G$  exponentially fast. Here $G_i$ denote the open bounded sets enclosed by the components 
$\Gamma_{G,1}, \dots, \Gamma_{G,m}$ of $\pa G$,  $F_{\infty, i}$ and $F_{0,i}$ are diffeomorphic to $G_i$, and $\pa F_{0, i}$ and $\pa F_{\infty, i}$ are the components of $\pa F_0$ and $\pa F_\infty$ respectively.

\end{theorem}

\begin{remark}\label{rm:precise}
 By  exponential convergence of $F_t$ to $F_\infty$ we mean precisely the following: 
writing $\pa F_t:=\{x+\tilde h(x, t)\nu_{F_\infty}(x):x\in \pa F_\infty\}$, we have that for every $k \in \mathbb{N}$ there exists $c_k >0$ and $C_k >1$ such that  
\[
\|\tilde h(\cdot, t)\|_{C^k(\pa F_\infty)} \leq C_k e^{-c_k t}
\]
for  $t \geq 1$.
\end{remark}

The proof of stability is based on the following energy identity.

\begin{proposition} 
\label{magic formula}
Let  $(F_t)_{t \in [0,T)}$ be the solution of  \eqref{FLOW} provided by Theorem~\ref{thm surf}.  Then the function 
$$
t\mapsto \int_{\pa F_t} |\nabla R_t|^2  \, d \Ha^2 
$$
is absolutely continuous and 
for almost every $t \in (0,T)$ we have the following energy identity 
\begin{multline}\label{magic}
\frac{d}{dt} \left(\int_{\pa F_t} |\nabla R_t|^2  \, d \Ha^2 \right) = -  2 \pa^2 \mathcal{J}(F_t)[\Delta R_t]
\\- 2 \int_{\pa F_t} B_{F_t}[\nabla R_t, \nabla  R_t] \, (\Delta R_t)  \, d \Ha^{2} +  \int_{\pa F_t} H_{F_t} |\nabla R_t|^2  \, (\Delta R_t)  \, d \Ha^{2},
\end{multline}
 where $\pa^2 \mathcal{J}(F_t)$ is defined as in \eqref{eq:pa2J}.  
\end{proposition}

The proof of the proposition is similar to \cite[Proposition 4.3]{surf2D} (see also \cite[Lemma~4.4]{AFJM}) and therefore we shift it to the appendix.

In order to control the two last terms in \eqref{magic} we  need the following interpolation result on the evolving boundaries. The proof of the next lemma is precisely the same as \cite[Lemma 4.7]{AFJM} and therefore we omit it. 
\begin{lemma}
\label{nasty}
If $F \subset  U$ is such that  $\pa F = \{ x + h_F(x) \nu(x) : x \in \Sigma\}$ with $\|h_F\|_{C^{1,\alpha}(\Sigma)} \leq M$,    then for every smooth function $f \in C^\infty(\pa F)$ it holds
\[
\begin{split}
\int_{\partial F} |B_{ F}| |\nabla  f|^2 |\Delta  f| \, d \Ha^{2} \leq C \, \left(1+ \|H_{F}\|_{L^6(\pa F)}^3 \right) \| \nabla  \Delta f\|_{L^2(\pa F)}^2 \,  \|\nabla f \|_{L^2(\pa F)}\,.
\end{split}
\]
The constant $C$ depends only on $M$ and $\Sigma$. 
\end{lemma}

We are now ready to prove Theorem  \ref{thmstability}.

\begin{proof}[\textbf{Proof of Theorem \ref{thmstability}.}]
 
 For any set $F \in \mathfrak{h}_{1}^3(\Sigma)$  consider  
$$
 D(F):=\int_{F\Delta G}\mathrm{dist\,}(x, \Sigma)\, dx
$$
 and note that 
 \beq\label{L2 estimate}
\frac{1}{C} \|h_F\|_{L^2(\pa G)}^2  \leq  D(F)  \leq C \|h_F\|_{L^2(\pa G)}^2
 \eeq
for a constant  depending only on $G$.  Moreover,  we define 
\[
R_F := H_F - Q(E(u_F))  
\]
which is defined on $\pa F$. 

\medskip 

\noindent {\bf Step 1.}{\it (Preliminary estimates)}  In this step we  show  that if $F \in  \mathfrak{h}_{1}^3(\Sigma)$  and $\|h_F\|_{C^1(\Sigma)}\leq \de$ for $\de$ sufficiently small, then  it holds 
\beq \label{avain}
\frac{1}{C}\,  \| h_F \|_{H^3(\Sigma)}^{1/\theta} \leq  D(F)  + \int_{\pa F} |\nabla R_F|^2 \, d\Ha^2 \leq C \,  \| h_F \|_{H^3(\Sigma)}^\theta.
\eeq
for $\theta \in (0,1)$ and  for constant  $C>1$. 

We begin by proving the first inequality. We  use   interpolation,   \eqref{prima} and the  second inequality in Remark \ref{rem after long} to  deduce that 
\beq \label{step 0 elastic}
\begin{split}
&\| \nabla \big(Q(E(u_F))\circ\pi_F^{-1} - Q(E(u_G))\big)  \|_{L^2(\Sigma)} \\
&\leq C\|Q(E(u_F))\circ\pi_F^{-1} - Q(E(u_F))\circ \pi_F^{-1}  \|_{H^{\frac32}(\Sigma)}^{\theta'}  \|Q(E(u_F)) \circ \pi_F^{-1} - Q(E(u_G))  \|_{L^2(\Sigma)}^{1-\theta'} \\
&\leq (C  + \|Q(E(u_F))\circ \pi_F^{-1}  \|_{H^{3/2}(\Sigma)}^{\theta'})  \| (D u_F) \circ \pi_F^{-1} -  (D u_G)  \|_{L^2(\Sigma)}^{1-\theta'} \\
&\leq (C + \|h_F\|_{H^{3}(\Sigma)}^{\theta'}) \|h_F\|_{H^{2}(\Sigma)}^{1-\theta'}\\
&\leq C  \|h_F\|_{H^{2}(\Sigma)}^{1-\theta'}
\end{split}
\eeq
for $\theta' \in (0,1)$.  Since $G$ is a stationary set it holds $\nabla R_G = 0$ on $\Sigma$. Therefore  we conclude by the above inequality that 
\[
\begin{split}
\| \nabla \big(H_F \circ \pi_F^{-1} - &H_G\big) \|_{L^2(\pa F)}^2 \\
&\leq  2\int_{\Sigma} |\nabla (R_F\circ \pi_F^{-1})|^2 \, d\Ha^2  +2 \| \nabla \big(Q(E(u_F))\circ\pi_F^{-1} -  Q(E(u_G)) \big) \|_{L^2(\Sigma)}^2 \\
&\leq 2C\int_{\pa F} |\nabla R_F|^2 \, d\Ha^2  + C \|h_F\|_{H^{2}(\Sigma)}^{2(1-\theta')}.
\end{split}
\]    
We use \eqref{schifio2}, \eqref{mean curvature} and the fact that $\|h_F\|_{C^1(\Sigma)}\leq \de$ to deduce with straightforward calculations  
\[
 \| h_F \|_{H^3(\Sigma)}^2 \leq C \| \nabla (H_F \circ \pi_F^{-1} -  H_G) \|_{L^2(\Sigma)}^2 + C  \| h_F \|_{H^2(\Sigma)}^2 .
\]
Therefore, from the two previous inequalities and  by interpolation we  obtain  that 
\[
\begin{split}
 \| h_F \|_{H^3(\Sigma)}^2 &\leq C \int_{\pa F} |\nabla R_F|^2 \, d\Ha^2  + C \|h_F\|_{H^{2}}^{2(1-\theta')} + C \| h_F \|_{H^2}^2 \\
&\leq C \int_{\pa F} |\nabla R_F|^2 \, d\Ha^2 + \frac12\| h_F \|_{H^3}^2 + C  \| h_F \|_{L^2}^{\theta''},
 \end{split}
\]
for a suitable $\theta''\in (0,1)$.  The first inequality in \eqref{avain} then follows from the previous  the previous estimate and from \eqref{L2 estimate}, recalling that since $\|h_F\|_{H^3(\Sigma)}\leq 1$ we have also $\|\nabla R_F\|_{L^2(\pa F)}\leq C$.

To prove the second inequality in \eqref{avain} we argue similarly as  above and use \eqref{mean curvature} to conclude that 
\[
\| \nabla (H_F \circ \pi_F^{-1} - H_G) \|_{L^2(\Sigma)}^2 \leq  C \, \| h_F \|_{H^3(\Sigma)}^2
\]
Moreover by   \eqref{step 0 elastic} we have that 
\[
\| \nabla \big(Q(E(u_F))\circ\pi_F^{-1} -   Q(E(u_G))\big)  \|_{L^2(\Sigma)} \leq C \,  \|h_F\|_{H^{2}(\Sigma)}^{1-\theta'}
\]
for $\theta' \in (0,1)$. Therefore since  $G$ is a critical set we obtain 
\[
\begin{split}
\int_{\pa F} |\nabla &R_F|^2 \, d\Ha^2  \leq C \int_{\Sigma} |\nabla (R_F \circ \pi_F^{-1} - R_G)|^2 \, d\Ha^2 \\
&\leq C \int_{\Sigma} |\nabla (H_F \circ \pi_F^{-1} - H_G) |^2 \, d\Ha^2  + C  \int_{\Sigma} |\nabla \big(Q(E(u_F)) \circ \pi_F^{-1} -  
 Q(E(u_G))\big)|^2 \, d\Ha^2 \\
&\leq C \| h_F \|_{H^3(\Sigma)}^2 +C \|h_F\|_{H^{2}(\Sigma)}^{2(1-\theta')} \leq  C  \|h_F\|_{H^{3}(\Sigma)}^{\theta} .
\end{split}
\] 
Hence, we have \eqref{avain}.

\medskip 

\noindent {\bf Step 2.}{\it (Global existence)}
Let us assume that  the initial set  $F_0 $ is in $ \mathfrak{h}_{\delta}^3(\Sigma)$ with $\delta < \e_1$, where $\e_1\in (0,1)$ is the constant provided by Theorem~\ref{thm surf} corresponding to the choice $\de_0=1$, $K_0=\max\{2,5\|Q(E(u_G))\|_{L^{\infty}(\Sigma)} \}$. Then the flow $(F_t)_{t \in [0,T)}$ starting from $F_0$ which is a solution of \eqref{FLOW} exists for a time interval $(0,T)$, with  $T$ bounded from below by a positive constant which depends only $G$. Let  $\sigma>0$  be a small number which will be chosen later.  Note that by  \eqref{avain} and by continuity we have  
\beq \label{chiave}
D(F_t)  + \int_{\pa F_t} |\nabla R_t|^2 \, d\Ha^2 \leq C\|h(\cdot, t)\|_{H^3(\Sigma)}^\theta\leq C\de^\theta< \sigma 
\eeq
for some time interval  $(0,T')$, where the last inequality holds provided that $\de$ is small enough.  
Note that by  \eqref{avain}  it follows that  
\beq\label{max}
\| h(\cdot,t) \|_{H^3(\Sigma)} < C\sigma^\theta < \min\{\e_1, \sigma_1\}\qquad\text{ for every $t \in (0,T')$}
\eeq
 when $\sigma$ is small enough, where $\sigma_1$ is the constant provided by Proposition~\ref{stationary}. In particular,  we conclude from Theorem \ref{thm surf} that  as long as the flow $(F_t)_{t \in (0,T)}$ satisfies \eqref{chiave}  it is well defined. In other words, if $(0,T^*)$ is the maximal time of existence and if  it satisfies \eqref{chiave} for every $t \in (0,T^*)$, then  $T^* = \infty$, i.e., the flow exists for all times. 

Let us denote by $[0,T')$ the maximal time interval where the flow satisfies  \eqref{chiave}.  We claim that if  $\| h_0 \|_{H^3(\Sigma)} <  \delta$ for  $\delta$ small enough, then the flow satisfies \eqref{chiave} for every $t \in (0,T^*)$ and thus $T^*=T'=+\infty$ . 

We start by recalling that by Lemma~\ref{lemma:j2>0near} and \eqref{max}, since $\sigma_1<\sigma_0$, we have
$$
\partial^2 \mathcal{J}(F_t)[\Delta R_t]\geq\frac{c_0}{2}\| \Delta R_t \|_{H^1(\pa F_t)}^2\qquad\text{ for every $t \in (0,T')$.}
$$
Thus,  from the energy identity \eqref{magic}, using also Lemma~\ref{nasty}    and  again  \eqref{chiave},   we may  estimate
\beq \label{step 1 long}
\begin{split}
\frac{d}{dt} \int_{\pa F_t} |\nabla R_t|^2  \, d \Ha^{2}  &\leq  -  2 \partial^2 \mathcal{J}(F_t)[\Delta R_t] + C \int_{\pa F_t} |B_{F_t}| |\nabla R_t|^2 \, |\Delta R_t|  \, d \Ha^{2} \\
&\leq -c_0 \| \Delta R_t \|_{H^1(\pa F_t)}^2  + C \,  (1+ \|H_{F_t}\|_{L^6(\pa F_t)}^3 ) \| \nabla  \Delta R_t\|_{L^2(\pa F_t)}^2 \,  \|\nabla R_t \|_{L^2(\pa F_t)} \\
&\leq -c_0 \| \Delta R_t \|_{H^1(\pa F_t)}^2  + C \sqrt{\sigma} \, \| \nabla  \Delta R_t\|_{L^2(\pa F_t)}^2  \\
&\leq -\frac{c_0}{2} \| \Delta R_t \|_{H^1(\pa F_t)}^2,
\end{split}
\eeq
where the last inequality holds by taking $\sigma$ smaller if needed.

Next we show that 
\beq \label{sort of poincare}
\|\nabla R_t \|_{L^2(\pa F_t)} \leq C  \| \Delta R_t \|_{L^2(\pa F_t)} 
\eeq
for some constant which depends on   $\Sigma$. Let us fix  a component  of $\pa F_t$ and denote it by $\Gamma_t$. Since $F_t$ is diffeomorphic to $G$ we denote the component of $\Sigma$
 diffeomorphic to $\Gamma_t$ by $\Gamma$. Since $\Gamma$ is smooth, compact and connected Riemannian manifold we conclude by \cite[Theorem 3.67]{AubinBook2} that the Poincar\'e inequality holds on $\Gamma$, i.e., for every $\varphi \in C^\infty(\Gamma)$ with $\int_{\Gamma} \varphi \, d \Ha^2 = 0$ it holds
\[
\| \varphi \|_{L^2(\Gamma)}  \leq C \|\nabla \varphi \|_{L^2(\Gamma)} .
\] 
Therefore since $\Gamma_t = \Phi_t(\Gamma)$ with $\Phi_t(x) = x + h(x,t) \nu(x)$ and $\| h(\cdot,t)\|_{C^{1,\alpha}} \leq  C $   the Poincar\'e inequality holds also on $\Gamma_t$. In particular, we have  
\[
\| R_t - \bar{R}_t \|_{L^2(\Gamma_t)}  \leq C \|\nabla R_t \|_{L^2(\Gamma_t)} , 
\]
where  $\bar{R}_t$ denotes the average of $R_t$ on $\Gamma_t$ and the constant depends on $\Sigma$. Then by integration by parts we get
\[
\begin{split}
\int_{\Gamma_t} |\nabla R_t|^2 \, d\Ha^2 &= - \int_{\Gamma_t} (R_t - \bar{R}_t) \Delta  R_t \, d\Ha^2   \\
&\leq \| R_t - \bar{R}_t \|_{L^2(\Gamma_t)} \|\Delta R_t \|_{L^2(\Gamma_t)} \leq C \|\nabla R_t \|_{L^2(\Gamma_t)} \|\Delta R_t \|_{L^2(\Gamma_t)}. 
\end{split}
\]
We obtain \eqref{sort of poincare} by repeating the above argument for every component of $\pa F_t$. 

By \eqref{step 1 long} and  \eqref{sort of poincare} we conclude that 
\[
\frac{d}{dt} \int_{\pa F_t} |\nabla R_t|^2  \, d \Ha^{2}   \leq - c \,  \int_{\pa F_t} |\nabla R_t|^2  \, d \Ha^{2}  
\]
for every $t \in (0, T')$. Integrating this over $(0,t)$, using \eqref{avain} and $\| h_0\|_{H^3(\Sigma)} \leq \delta$ yield
\beq \label{step 1 key estimate}
 \int_{\pa F_t} |\nabla R_t|^2  \, d \Ha^{2} \leq C e^{-c t} \delta^\theta.
\eeq
On the other hand by differentiating $D(F_t)$ with respect to time and using the same calculations as in \cite[Lemma 3.3]{surf2D} we get 
\[
\begin{split}
\frac{d}{dt} D(F_t) &= \int_{\pa F_t} d_G \, \Delta R_t \, d \Ha^{2} =-  \int_{\pa F_t}\la  \nabla d_G , \nabla R_t \ra  \, d \Ha^{2} \\
&\leq  \Ha^2(\pa F_t)^{1/2} \, \left(  \int_{\pa F_t} | \nabla R_t |^2  \, d \Ha^{2} \right)^{1/2} \leq C e^{-\frac{c}{2} t} \delta^{\frac{\theta}{2}}.
\end{split}
\]
Integrating this over $(0,t)$, using \eqref{L2 estimate}  and $\| h_0\|_{H^3(\Sigma)} \leq \delta$ yield
\beq \label{step 1 short}
D(F_t) \leq D(F_0) + C e^{-\frac{c}{2} t} \delta^{\frac{\theta}{2}} \leq C \delta^2 + C e^{-\frac{c}{2} t} \delta^{\frac{\theta}{2}}  < \sigma
\eeq
when $\delta$ is chosen small enough. Hence, we have that \eqref{chiave} holds for the whole life span of the flow $(0,T^*)$ and by the previous discussion this implies that 
$T^* = \infty$.

\medskip

\noindent {\bf Step 3.}{\it (Convergence)}  Combining \eqref{avain} and \eqref{chiave} we have that $\sup_{t>0} \| h(\cdot, t) \|_{H^3(\Sigma)} \leq C\sigma^\theta$. Therefore there exists a subsequence such that 
\[
 h(\cdot, t_m) \to    h_\infty(\cdot) \qquad \text{in } \, H^2(\Sigma).
\]
We denote the target set by $F_\infty$, i.e.,  $\pa F_\infty = \{ x + h_\infty(x) \nu(x) : x \in \Sigma \}$.
By \eqref{step 1 key estimate}  we deduce that $\nabla R_{F_\infty} = 0$, i.e., $F_\infty$ is a stationary set. We will show that $F_t \to F_\infty$ exponentially fast. 

To this aim we define
\[
D_\infty(F) :=\int_{F\Delta F_\infty}\mathrm{dist\,}(x,  F_\infty)\, dx.
\]
Repeating the calculations leading to  \eqref{step 1 short} we get
\[
\Big|\frac{d}{dt} D_\infty(F_t) \Big| = \Big| \int_{\pa F_t} d_{F_\infty} \, \Delta R_t \, d \Ha^{2} \Big| \leq  \Ha^2(\pa F_t)^{1/2} \, \left(  \int_{\pa F_t} | \nabla R_t |^2  \, d \Ha^{2} \right)^{1/2} \leq C e^{-\frac{c}{2} t} \delta^{\frac{\theta}{2}},
\]
 where the last inequality follows from \eqref{step 1 key estimate}. This implies that  $ \lim_{t \to \infty} D_\infty(F_t)$ exists and  the choice of $F_\infty$ implies that $D_\infty(F_t) \to 0$. Therefore  integrating the above inequality over $(t,\infty)$ we get
\[
D_\infty(F_t) \leq C e^{-\frac{c}{2} t} \delta^{\frac{\theta}{2}}
\]
for every $t >0$. We change the reference set from $\Sigma = \pa G$ to $\pa F_\infty$ and  write $\pa F_t = \{ x + \tilde{h}(x,t) \nu_{F_\infty}(x) : x \in \pa F_\infty \}$.  Then by inequality  \eqref{L2 estimate}, with $\pa G$ replaced by $\pa F_\infty$,   and by the above inequality we have
\[
\| \tilde{h}(\cdot,t) \|_{L^2(\pa F_\infty)} \leq C e^{-\frac{c}{4} t} \delta^{\frac{\theta}{4}}.
\]
Moreover, since $\|h(\cdot, t) \|_{H^3(\Sigma)} \leq C\sigma^\theta$ for all $t>0$ then also $\|\tilde{h}(\cdot,t) \|_{H^3(\pa F_\infty)} \leq C$ for all $t>0$. By Theorem \ref{thm surf} we  conclude that  $\|\tilde{h}(\cdot,t) \|_{H^{2k+3}(\pa F_\infty)} \leq C(k,\sigma)$ for all $t \geq 1$ and for every $k \in \mathbb{N}$.  Thus we deduce by interpolation that 
\[
 \| \tilde{h}(\cdot, t) \|_{C^k(\pa F_\infty)} \leq C_k e^{-c_k t} \qquad \text{for all }\, t \geq 1
\]
for some constants $c_k>0$ and $C_k>1$ depending on $k$ and $K_0$. 

To conclude the proof, for every $t\in [0, +\infty]$ denote by $(\Gamma_{F_t,i})_{i=1,\dots, m}$  the connected components of $\pa F_t$, 
numbered according to \eqref{numbered}. Denote also   by $F_{t,i}$ the bounded open set enclosed by $\Gamma_{F_t,i}$ and recall that 
 the flow preserves the volume of each $F_{t,i}$. Indeed, 
$$
\frac{d}{ds}|F_{t+s,i}|_{|_{s=0}}=\int_{\Gamma_{F_t, i}} V_t\, d\Ha^2=
\int_{\Gamma_{F_t, i}} \Delta R_t\, d\Ha^2=0.
$$
Thus, recalling \eqref{max} and Proposition~\ref{stationary}, we may conclude that 
$F_\infty$ is the unique stationary set in $\mathfrak{h}^{3}_{\sigma_1}(\pa G)$
 such that $|F_{\infty, i}|=|F_{0,i}|$ for $i=1,\dots, m$.

\end{proof}

\section{Evolution of epitaxially strained elastic films}\label{sec:graphs}

In this section we briefly describe how our main results read in the context of evolving periodic graphs.

In this framework, given a (sufficiently regular) non-negative  function $h:\R^2\to [0,+\infty)$, $1$-periodic with respect to  both variables $x_1, x_2$, the free energy associated with it  reads
\begin{equation}\label{functional}
\mathcal{J}(h):=\int_{\Omega_{h}}Q(E(u_h))\,dx+\Ha^2(\Gamma_{h})\,,%
\end{equation}
where $x=(x_1,x_2, x_3)\in\mathbb{R}^{2}$,  $\Gamma_{h}$, $\Omega_{h}$ denote the graph and the subgraph  of $h$, respectively,  over the periodic cell, i.e.,
\begin{align*}
&\Omega_{h}:=\{(x_1,x_2, x_3)\in(0,1)^2\times\mathbb{R}:\,0<x_3<h(x_1, x_2)\}\,, \\
&\Gamma_{h}:=\{(x_1,x_2, x_3)\in(0,1)^2\times\mathbb{R}:\, x_3=h(x_1, x_2)\},
\end{align*}
and $u_h$ is the elastic equilibrium in $\Om_h$, namely the solution of  the elliptic system
\beq\label{leiintro}
\begin{cases}
	\Div\,\C E(u_h)=0 \quad \text{in $\Om_{h}$},\\
	\C E(u_h)[\nu_{\Om_h}]=0 \quad \text{on $\Gamma_h$,}\\
	D u_h(\cdot, x_3)  \quad \text{ is $1$-periodic,}\\
	 u(x_1, x_2, 0)=e_0(x_1, x_2, 0)\,,
\end{cases}
\eeq
for a suitable fixed constant $e_0\neq 0$.  The above energy relates to a variational model for epitaxial growth, see the introduction.  Precisely,  the graph $\Gamma_h$ describes the (free) profile of the elastic film, which occupies the region $\Om_h$ and is grown on a (rigid) and much thicker substrate, while the \emph{mismatch strain} constant $e_0$ appearing in the Dirichlet condition for $u_h$ at the interface $\{x_1=0\}$ between film and substrate measures the mismatch between the characteristic atomic distances in the lattices of the two materials.   
  In this framework, the (local) minimizers of \eqref{functional} under an area constraint on $\Om_h$ describe the equilibrium configurations of epitaxially strained elastic films, see \cite{FFLM, FFLM2, FFLM3, FM09} and the references therein. 

In the context of periodic graphs, given an initial $1$-periodic profile $ h_0\in H^3_{loc}(\R^2)$ (in short $h_0\in H^3_{\rm per}\big((0,1)^2\big)$), we look for a local-in-time solution $h(\cdot, t)$ of the following problem:
\beq\label{flow2}
\begin{cases}
	\frac 1{J_t}\pa_t h=\Delta_{\Gamma_t}\left(H_t+Q(E(u_t))\right) 
	& \text{on $\Gamma_{t}$ and for all $t\in(0, T)$,}\\
		h(\cdot, t) \text{ is $1$-periodic}& \text{for all $t\in(0, T)$,}\\
	h(\cdot, 0)= h_0\,,
\end{cases}
\eeq
where  
$J_t:=\sqrt{1+|D h(\cdot, t)|^2}$,    
 $u_t$ stands for the solution of  \eqref{leiintro}, with $\Om_{h_t}$ in place of $\Om_h$,  we wrote $\Gamma_t$ instead of 
 $\Gamma_{h_t}$, and $H_t$ denotes the mean  curvature of  $\Gamma_t$.  Note that in the first equation of \eqref{flow2} we have $+Q(E(u_t))$ instead of $-Q(E(u_t))$. This is due to the fact that in 
 \eqref{functional}  the vector $\nu_{\Om_h}$ now points outwards with respect to the elastic body.

Although the setting is a bit  different from that of the previous sections,  the short-time existence theory of Section~\ref{sec:existence} clearly extends also to the present situation, with the same arguments. In this way we improve upon the results of \cite{FFLM3} at least in the case of isotropic surface energy density. 

Also the stability analysis of Section~\ref{sec:stability} applies  without any essential changes, thus showing that strictly stable stationary $1$-periodic configurations  are exponentially stable in the sense of Theorem~\ref{thmstability}. 

A particular class of critical configurations to which our stability theorem applies  are the flat configurations, that is, in the case of constants 
 profiles $h\equiv d$, provided that $d>0$ is sufficiently small. Indeed in \cite[Proposition~7.3]{Bo} it is shown that if $d$ is sufficiently small 
 then the flat configuration $h\equiv d$ is strictly stable for the functional $\mathcal{J}$.  Therefore, we may state the following theorem.

\begin{theorem}\label{th:2dliapunov}
  There exists $d_0>0$ with the following property: Let $d\in (0, d_0)$. Then, there exists $\de>0$ such that  if 
  $$
  \|h_0-d\|_{H^3((0,1)^2)}\leq \de \quad\text{and}\quad  \int_{(0,1)^2}h_0\, dx=d\,,
  $$
  then the unique solution $h(\cdot, t)$ of \eqref{flow2} exists for all $t>0$ and for every integer $k\geq 1$ we have
  $$
  \|h(\cdot, t)-d\|_{C^k([0,1]^2)}\leq C_k \mathrm{e}^{-c_k t}\quad\text{for all $t>1$}
  $$
  and for suitable positive constants $C_k, c_k$.
\end{theorem}

\section{Appendix: technical lemmas}
In this appendix we collect a few technical results and we give the proof of Lemma~\ref{lemmaJ} and of Proposition~\ref{magic formula}.
\begin{lemma}\label{lm:taldeitali}
Let $\Sigma$  be an $m$-dimensional  smooth compact manifold in $\R^n$ and let $k\geq1$. If $f$, $g\in H^k(\Sigma)\cap L^\infty(\Sigma)$, then $fg\in H^k(\Sigma)$ and 
$\|fg\|_{H^k(\Sigma)}\leq C\big(\|f\|_{H^k(\Sigma)}\|g\|_{L^\infty(\Sigma)}+\|g\|_{H^k(\Sigma)}\|f\|_{L^\infty(\Sigma)}\big)$. Moreover, if $A \in C^\infty(\R)$ then 
$A(f) \in H^k(\Sigma)$ and  $\|A(f)\|_{H^k(\Sigma)}\leq C (1+\|f\|_{H^k(\Sigma)})$ where the constant depends on $A$ and on $\|f\|_{L^\infty(\Sigma)}$. 

Finally, if $U\subset\R^m$ is an open set $\Phi:\overline U\to \Phi(\overline U)\subset\Sigma$ is a diffeomorphism   of class $H^k\cap C^{1}$, $k\geq 1$, and
$f\in H^k(\Phi(U))\cap C^1({\Phi(\overline U)})$, then $\|f\circ \Phi\|_{H^k(U)}\leq C(\|Df\|_\infty, \|D\Phi\|_\infty)(\|f\|_{H^k}+\| \Phi\|_{H^k})$.
\end{lemma}
\begin{proof}
The first two statements of the lemma are classical, see for instance \cite[Propositions~3.7~and~3.9]{Taylor}. The third one can be proven by an induction argument from the first one. 
\end{proof}

We now prove Lemma~\ref{lemmaJ}.
\begin{proof}[Proof of Lemma~\ref{lemmaJ}]
First, recall \eqref{remainder} and observe that from the assumption on $h_i$ we have  $\sup_{0\leq t\leq T}\|h_i(\cdot, t)\|_{C^{1,\alpha}(\Sigma)}\leq C\de^{\theta'}$ for a suitable $C>0$ and $\theta'\in (0,1)$. We begin by estimating for $\e>0$
\beq\label{lemmaJ1}
\begin{split}
&\int_0^T\int_{\Sigma} |\la A(x, h_2, \nabla h_2), \nabla^4 h_2 \ra - \la A(x, h_1, \nabla h_1), \nabla^4 h_1  \ra|^2\, d\Ha^2dt \\
&\leq 2\int_0^T\int_{\Sigma}|A(x, h_2, \nabla h_2)|^2 |\nabla^4 h_2 -\nabla^4h_1|^2\,d\Ha^2dt\\
&\quad+ 
2\int_0^T\int_{\Sigma}|\nabla^4 h_1|^2|A(x, h_2, \nabla h_2)-  A(x, h_1, \nabla h_1)|^2\, d\Ha^2dt \\
&\leq \e \int_0^T\int_{\Sigma} |\nabla^4 h_2 -\nabla^4h_1|^2\,d\Ha^2dt \\
&\quad+ C
\int_0^T\int_{\Sigma}|\nabla^4 h_1|^2(  |h_2-h_1|^2+ |\nabla h_2- \nabla h_1|^2)\, d\Ha^2dt. 
\end{split}
\eeq
To estimate the last term, we use the Sobolev inequality and the interpolation Lemma~\ref{aubinlemma}, and have 
\beq \label{lemmaJ2}
\begin{split}
& \int_0^T\int_{\Sigma}|\nabla^4 h_2|^2(  |h_2-h_1|^2+ |\nabla h_2- \nabla h_1|^2)\, d\Ha^2dt \\
&\leq C\int_0^T\|h_2(\cdot, t)-h_1(\cdot, t)\|_{W^{1,4}}^2\|\nabla^4 h_2(\cdot, t)\|^2_{L^4}\, dt\\
&\leq C\sup_{0\leq t\leq T}\|h_2(\cdot, t)-h_1(\cdot, t)\|^2_{H^2}\int_0^T\| h_2(\cdot, t)\|_{H^5}^{\frac53}\|\nabla h_2(\cdot, t)\|_{L^{\infty}}^{\frac13}\, dt\\
&\leq C\delta^{\frac{\theta'}3} \sup_{0\leq t\leq T}\|h_2(\cdot, t)-h_1(\cdot, t)\|^2_{H^2}T^{\frac16}\biggl(\int_0^T\|h_2(\cdot, t)\|_{H^5}^2\,dt\biggr)^{\frac56}\\
&\leq C(M_0)T^{\frac16}\sup_{0\leq t\leq T}\|h_2(\cdot, t)-h_1(\cdot, t)\|^2_{H^2}.
\end{split}
\eeq
Concerning the estimate of
$$
\int_0^T\int_\Sigma|J_1(x, h_2, \nabla h_2, \nabla^2 h_2, \nabla^3h_2) - J_1(x, h_1, \nabla h_1, \nabla^2 h_1, \nabla^3h_1)|^2\, d\Ha^2dt
$$
we observe that
\begin{align*}
& \int_0^T\int_\Sigma|\la B_1(x, h_2, \nabla h_2),  \nabla^3 h_2\otimes \nabla^2h_2\ra - \la B_1(x, h_1, \nabla h_1),  \nabla^3 h_1\otimes \nabla^2h_1\ra|^2\, d\Ha^2dt \\
&\leq C\int_0^T\int_\Sigma| B_1(x, h_2, \nabla h_2)-  B_1(x, h_1, \nabla h_1)|^2 |\nabla^3 h_2 \otimes \nabla^2h_2|^2\, d\Ha^2dt \\
&\quad +C\int_0^T\int_\Sigma| B_1(x, h_1, \nabla h_1)|^2  |\nabla^3 h_2 -\nabla^3 h_1 |^2|\nabla^2h_2|^2\, d\Ha^2dt\\
&\quad+C\int_0^T\int_\Sigma| B_1(x, h_1, \nabla h_1)|^2  |\nabla^2 h_2 -\nabla^2 h_1|^2|\nabla^3h_1|^2\, d\Ha^2dt\\
&\leq C\int_0^T\int_\Sigma (|h_2-h_1|^2 + |\nabla h_2 -\nabla h_1|^2)|\nabla^3h_2|^2|\nabla^2h_2|^2\, d\Ha^2dt\\
&\quad +C\int_0^T \int_\Sigma|\nabla^2 h_2 -\nabla^2 h_1|^2|\nabla^3h_1|^2\, d\Ha^2dt\\
&\quad+ C \int_0^T\int_\Sigma |\nabla^3 h_2 -\nabla^3 h_1|^2 |\nabla^2 h_2|^2\, d\Ha^2dt=:I_1+I_2+I_3.
\end{align*}
By a simple interpolation argument, we have
\begin{align*}
I_3 &\leq \int_0^T\|\nabla^3 h_2-\nabla^3h_1\|_{L^4}^2\|\nabla^2 h_2\|_{L^4}^2\, dt \leq CM_0
\int_0^T\|h_1-h_2\|_{H^4}^{\frac32}\|\nabla^2 h_2-\nabla^2h_1\|_{L^2}^{\frac12}\\
&\leq \e \int_0^T\|\nabla^4 h_2-\nabla^4 h_1\|_{L^2}^2\, dt+C_\e(M_0) T \sup_{0\leq t\leq T}\|h_2(\cdot, t)-h_1(\cdot, t)\|_{H^2}^2. 
\end{align*}
Similarly
\begin{align*}
I_2 &\leq \int_0^T\|\nabla^2 h_2-\nabla^2h_1\|_{L^4}^2\|\nabla^3 h_1\|_{L^4}^2\, dt \\ 
&\leq C\int_0^T\|h_2- h_1\|_{H^4}^{\frac12}\|h_2- h_1\|_{H^2}^{\frac32}\| h_1\|_{H^5}^{\frac12}
\|\nabla^3 h_1\|_{L^2}^{\frac32}\, dt\\
&\leq \e \int_0^T\|\nabla^4 h_2-\nabla^4 h_1\|_{L^2}^2\, dt+C_\e (M_0) \sup_{0\leq t\leq T}\|h_2(\cdot, t)-h_1(\cdot, t)\|_{H^2}^2 
 \int_{0}^T1+\| h_1\|_{H^5}^{\frac23}\, dt\\
&\leq \e \int_0^T\|\nabla^4 h_2-\nabla^4 h_1\|_{L^2}^2\, dt+C_\e (M_0)T^{\frac23} \sup_{0\leq t\leq T}\|h_2(\cdot, t)-h_1(\cdot, t)\|_{H^2}^2. 
\end{align*}
Finally, arguing similarly as above,
\begin{align*}
I_1 &\leq \int_0^T\| h_1-h_2\|_{W^{1,6}}^2\|\nabla^3 h_2\|_{L^6}^2\|\nabla^2h_2\|_{L^6}^2\, dt \\ 
&\leq C M_0 \sup_{0\leq t\leq T}\|h_2(\cdot, t)-h_1(\cdot, t)\|_{H^2}^2 \int_0^T \| h_2\|_{H^5}^{\frac23}\|\nabla^3 h_2\|_{L^2}^{\frac43} \, dt \\
&\leq C(M_0) T^{\frac23} \sup_{0\leq t\leq T}\|h_2(\cdot, t)-h_1(\cdot, t)\|_{H^2}^2\,.
\end{align*}
Since the difference of the remaining  terms in $J_1$ can be treated in a similar (in fact easier) way, we conclude that
\begin{align}
&\int_0^T\int_\Sigma|J_1(x, h_2, \nabla h_2, \nabla^2 h_2, \nabla^3h_2 - J_1(x, h_1, \nabla h_1, \nabla^2 h_1, \nabla^3h_1))|^2\, d\Ha^2dt\label{lemmaJ3}\\
&\leq \e \int_0^T\|\nabla^4h_2(\cdot, t)-\nabla^4h_1(\cdot, t)\|_{2}^{2}\, dt
+C_\e(M_0)T^\theta \sup_{0\leq t\leq T}\|h_2(\cdot, t)- h_1(\cdot, t)\|_{H^2}^2.\nonumber
\end{align}
We are left  to show that 
\begin{align}
&\int_0^T\int_\Sigma|J_2(x, h_2, \nabla h_2, \nabla^2 h_2, \nabla f, \nabla^2f - J_2(x, h_1, \nabla h_1, \nabla^2 h_1, \nabla f, \nabla^2f))|^2\, d\Ha^2dt\label{lemmaJ4}\\
&\leq \e \int_0^T\|\nabla^4h_1(\cdot, t)-\nabla^4h_2(\cdot, t)\|_{2}^{2}\, dt
+C_\e(M_0, K_0)T^\theta \sup_{0\leq t\leq T}\|h_2(\cdot, t)- h_1(\cdot, t)\|_{H^2}^2\nonumber \,.
\end{align}
As before we only prove the estimate for
$$
I_4:=\int_0^T\int_\Sigma|\la A_1(x, h_2, \nabla h_2)- A_1(x, h_1, \nabla h_1),\nabla^2 f\ra|^2\, d\Ha^2dt,
$$
the other terms being similar (or easier).  Using once again Lemma~\ref{aubinlemma} we have
\begin{align*}
I_4&\leq \int_0^T\|h_2-h_1\|_{W^{1,4}}^{2}\|\nabla^2 f\|^2_{L^4}\, dt\\
&\leq C \sup_{0\leq t\leq T}\|h_2(\cdot, t)- h_1(\cdot, t)\|_{H^2}^2\int_0^T\|\nabla^3 f\|^{\frac32}_{2}\|f\|_{L^\infty}^{\frac12}\, dt\\
&\leq C K_0^{\frac12}
\sup_{0\leq t\leq T}\|h_2(\cdot, t)- h_1(\cdot, t)\|_{H^2}^2\int_0^T\|\nabla^3 f\|^{\frac32}_{L^2}\, dt\\
&\leq C K_0^{\frac54} T^{\frac14}\sup_{0\leq t\leq T}\|h_2(\cdot, t)- h_1(\cdot, t)\|_{H^2}^2.
\end{align*}
The conclusion then follows by collecting \eqref{lemmaJ1}-\eqref{lemmaJ4}.
\end{proof}

Finally we give the proof of  Proposition \ref{magic formula}. 

\begin{proof}[Proof of Proposition \ref{magic formula}]
The proof is similar to the proof of \cite[Lemma 3.3]{surf2D}. For this reason we adopt the same notation as there and extend every function on $\pa F_t$ using the signed distance function $d_{F_t}$. In particular, the normal $\nu_t = \nu_{F_t}$, the second fundamental form $B_t = B_{F_t}$ and the mean curvature $H_t = H_{F_t}$ are extended to a tubular neighborhood of $\pa F_t$. Recall that $D_\tau $ denotes the tangential gradient  defined in \eqref{tang grad} and $\Div_\tau$ denotes the tangential divergence, which is defined as $\Div_\tau X = \Div X - (DX \nu_t ) \cdot \nu_t$. The Laplace-Beltrami operator on $F_t$ can be written as $\Delta v = \diver_\tau(D_\tau v)$, the second fundamental form as $B_t = D_\tau\nu_t$ and the mean curvature as  $H_{t} = \diver_\tau \nu_t$.

The regularity properties of $h$ stated in Theorem~\ref{thm surf} imply that for every integer $k\geq 1$  $\nabla^k h\in H^1_{loc}(0,T; L^2(\Sigma))$. Therefore, in what follows all the time derivatives are well defined almost everywhere. In turn, this allows us to differentiate $u_t:=u_{F_t}$ with respect to time.  More precisely, setting $\dot{u}_t:=\frac{\pa u_{t+s}}{\pa s}\big|_{s=0}$, we can argue as in  \cite[Theorem~4.1]{Bo} to conclude that $\dot{u}$ solves 
\begin{equation}\label{eq u dot}
\int_{\Om \setminus F_t} \C E(\dot u_t) : E(\vphi) \, dx =- \int_{\pa F_t}\text{div}_\tau (\Delta R_t\, \C E(u_t) ) \cdot \vphi \, d \Ha^2 
\end{equation}
for all $\varphi \in H^1(\Omega \setminus F_t; \R^3)$ such that $\varphi = 0$ on $\pa_D\Omega$. Note also that $\dot u_t=0$ on $\pa_D\Omega$.

 Let us fix time $t>0$. To continue we  observe that, by redefining the velocity field  X assosiated with the flow \eqref{FLOW}  if needed (in a time interval centered at $t$), we may assume that $X_t$ has only a normal component on $\pa F_t$; that is, 
 $$
 X_t=(X_t\cdot \nu_t)\nu_t = (\Delta R_t) \nu_t\qquad\text{on $\pa F_t$.}
$$ 
Since we extended all the geometric quantities by means of the gradient of the signed distance  from $F_t$  we have the following equality (see \cite{CMM})
\[
\dot{\nu}_t=-D_\tau(X_t\cdot\nu_t)=-D_\tau(\Delta R_t) \ \qquad\text{on $\pa F_t$,}. 
\]
 This implies (see the proof of \cite[eq.~(5.15)]{AFJM})
\beq \label{nupuntobis}
 \dot{H}_t :=  \frac{\partial }{\partial s} H_{t+s} \bigl|_{s=0}=  - \Delta^2 R_t \qquad \text{on} \,  \pa F_t.
\eeq
Moreover we have (see \cite{CMM}) 
\beq \label{nupuntobis---what the fuck that means---}
 \pa_{\nu_t} H_t = - |B_t|^2 \qquad \text{on} \,  \pa F_t.
\eeq
Denoting by $D_{\tau_{t+s}}$ the tangential gradient on $\pa F_{t+s}$ and by $J_\tau\Phi_s$ the  tangential Jacobian
of $\Phi_s$,   we have  
\beq\label{der1SD}
\begin{split}
&\frac{d}{ds} \left(\frac{1}{2}   \int_{\partial F_{t+s}} |D_\tau R_{t+s}|^2\, d \Ha^{2} \right)\Bigl|_{s=0} \\
&=   \frac{d}{ds} \left(\frac{1}{2}   \int_{\partial F_{t}} (|D_{\tau_{t+s}} R_{t+s}|^2\circ\Phi_s )\, J_\tau\Phi_s\, d \Ha^{2}\right) \Bigl|_{s=0} \\
&= \frac{1}{2}   \int_{\partial F_{t}} |D_\tau R_{t}|^2 \diver_\tau (\Delta  R_t\, \nu_t )\, d \Ha^{2} +   \int_{\partial F_{t}} D_\tau R_{t} \cdot \frac{\pa }{\pa s}  \left( D_{\tau_{t+s}} R_{t+s}\circ \Phi_s \right)\Bigl|_{s=0}\, d \Ha^{2}\\
&= \frac{1}{2}   \int_{\partial F_{t}} H_t |D_\tau  R_{t}|^2 \Delta R_t\, d \Ha^{2} +   \int_{\partial F_{t}} D_\tau R_{t} \cdot \frac{\pa }{\pa s}  \left( D_{\tau_{t+s}} R_{t+s}\circ \Phi_s \right)\Bigl|_{s=0}\, d \Ha^{2}
\end{split}
\eeq
We write the last term  as
\[
D_{\tau_{t+s}} R_{t+s}\circ \Phi_s   = \left[ I - \nu_{t+s}\circ \Phi_s  \otimes \nu_{t+s}\circ \Phi_s   \right] DR_{t+s}\circ \Phi_s 
\]
and get  (recall $\dot{\Phi} = X_t =  (\Delta R_t) \nu_t$)
\[
\begin{split}
\frac{\pa }{\pa s} \big( D_{\tau_{t+s}} R_{t+s}&\circ \Phi_s \big)\Bigl|_{s=0} =    \left[ I - \nu_t \otimes \nu_t  \right] (D \dot{R}_t + D^2R_t X_t )    + (-  \dot{\nu}_t \otimes \nu_t - \nu_t \otimes  \dot{\nu}_t) DR_t  \\
&= D_\tau \dot{R}_t + \Delta R_t \left(  (I - \nu_t \otimes \nu_t)D^2R_t\right)[\nu_t]  +(DR_t \cdot  \nu_t)\, D_\tau\Delta R_t   -    (DR_t \cdot \dot{\nu}_t) \nu_t.
\end{split}
\]
Note that $D_\tau  (DR_t \cdot  \nu_t) = B_t D_\tau R_t + \big((I - \nu_t \otimes \nu_t) D^2R_t\big)[\nu_t]$. Thus we have
\[
\begin{split}
 D_\tau R_{t} \cdot \frac{\pa }{\pa s}  \left( D_{\tau_{t+s}} R_{t+s}\circ \Phi_s \right)\Bigl|_{s=0}\!\! &=  (D_\tau R_{t} \cdot D_\tau \dot{R}_t)  - \Delta R_t  ( B_t [D_\tau R , D_\tau R_t])\\
&\quad + \Delta R_t \big(D_\tau R \cdot D_\tau (DR_t \cdot \nu_t) \big) + (D_\tau R_{t}  \cdot  D_\tau\Delta R_t) \,   (DR_t\cdot  \nu_t).
\end{split}
\]
Therefore by integrating by parts the first and the third terms we obtain
\[
\begin{split}
 &\int_{\partial F_{t}} D_\tau R_{t} \cdot \frac{\pa }{\pa s}  \left( D_{\tau_{t+s}} R_{t+s}\circ \Phi_s \right)\Bigl|_{s=0}\, d \Ha^{2}  \\
&= \int_{\partial F_{t}}   (D_\tau R_{t} \cdot D_\tau \dot{R}_t)  - \Delta R_t   \big(B_t [D_\tau R , D_\tau R_t]\big)\, d \Ha^{2}  \\
&\,\,\,\,\,\,\,\,+  \int_{\partial F_{t}}    \Delta R_t  \big( D_\tau R \cdot  D_\tau  (DR_t \cdot  \nu_t)\big)+  (D_\tau R_{t} \cdot  D_\tau\Delta R_t)  \, (DR_t \cdot  \nu_t)\, d \Ha^{2} \\
&=  \int_{\partial F_{t}}  -  \Delta R_{t} \,  \dot{R}_t - \Delta R_t   \big(B_t [D_\tau R , D_\tau R_t]\big) \,  d \Ha^{2}  \\
&\,\,\,\,\,\,\,\,+  \int_{\partial F_{t}}  -   (DR_t \cdot \nu_t)  \, \Div_\tau (\Delta R_t    D_\tau R_t)   + (D_\tau R_{t} \cdot  D_\tau\Delta R_t)  \, (DR_t \cdot  \nu_t) \, d \Ha^{2}\\
&=  \int_{\partial F_{t}} -  \Delta R_{t} \,  \dot{R}_t  -  (DR_t \cdot \nu_t)  \,  (\Delta R_t)^2 - \Delta R_t   \big(B_t [D_\tau R , D_\tau R_t]\big)\, d \Ha^{2} .
\end{split}
\]

Let us denote $u_t = u_{F_t}$ and $\dot{u}_t = \frac{\pa }{\pa t}u_t $.  By \eqref{nupuntobis}   it holds
\[
\dot{R}_t  = \dot{H}_t  + \frac{\pa}{\pa t}Q(E({u}_t )) = - \Delta^2 R_t + \C E(\dot{u}_t ):E(u_t)
\]
and by \eqref{nupuntobis---what the fuck that means---}  we have
\[
(DR_t, \nu_t) = \pa_{\nu_t} H_{t} + \pa_{\nu_t} Q(E(u_t))   = - |B_t|^2 + \pa_{\nu_t} Q(E(u_t)).
\]
Therefore we get
\[
\begin{split}
&\int_{\partial F_{t}} D_\tau R_{t} \cdot \frac{\pa }{\pa s}  \left( D_{\tau_{t+s}} R_{t+s}\circ \Phi_s \right)\Bigl|_{s=0}\, d \Ha^{2}=  \int_{\partial F_{t}}     \Delta R_{t} \,  \Delta^2 R_{t}  -   \C E(\dot{u}_t ):E(u_t)  \Delta R_{t}\, d\Ha^2 \\
&\quad\qquad\qquad\qquad\qquad    +\int_{\partial F_{t}}|B_t|^2 (\Delta R_t)^2-\pa_{\nu_t}Q(E(u_t))  \,  (\Delta R_t)^2  -\Delta R_t    \big(B_t [D_\tau R , D_\tau R_t]\big) \, d \Ha^{2}\,. 
\end{split}
\]
Observe now that using the second equation in \eqref{uf} and  \eqref{eq u dot} we have
\begin{multline*}
\int_{\partial F_{t}} \C E(\dot{u}_t ):E(u_t)  \Delta R_{t}\, d\Ha^2=
\int_{\partial F_{t}} \C E({u}_t ):D(\dot{u}_t)  \Delta R_{t}\, d\Ha^2 \\=
\int_{\partial F_{t}} \C E({u}_t ):D_\tau(\dot{u}_t)  \Delta R_{t}\, d\Ha^2
=-\int_{\partial F_{t}}\Div_\tau(\Delta R_t \C E(u_t))\cdot \dot{u}_t= 2\int_{\Om \setminus F_t} Q (E(\dot{u}_t)) \, dx.
\end{multline*}
Collecting the previous three identities we then get 
\begin{multline*}
\int_{\partial F_{t}} D_\tau R_{t} \cdot \frac{\pa }{\pa s}  \left( D_{\tau_{t+s}} R_{t+s}\circ \Phi_s \right)\Bigl|_{s=0}\, d \Ha^{2}=
- \int_{\partial F_{t}}     |\nabla \Delta_\tau R_{t}|^2  + 2 Q(E(\dot{u}_t))  \Delta R_{t}\, d\Ha^2\\
    +\int_{\partial F_{t}} |B|^2 (\Delta R_t)^2    - \pa_{\nu_t} Q(E(u_t)) \,  (\Delta R_t)^2  - B_{t}[\nabla R_t, \nabla  R_t] \,\Delta R_t  \, d \Ha^{2}.
\end{multline*}
We notice that the first four terms coincide with $- \pa^2 J(F_t)[\Delta R_t]$ (see \eqref{eq:pa2J}).   Thus, combining the last identity with \eqref{der1SD}, we obtain \eqref{magic}. 
\end{proof}

\end{document}